\documentclass[10pt,a4paper]{article}

\usepackage{amsmath}
\usepackage{amsfonts}
\usepackage{amsthm}
\usepackage{amssymb}
\usepackage{url}
\usepackage[usenames,dvips]{color}
\usepackage{stmaryrd}
\usepackage[dvips]{graphicx}
\usepackage{psfrag, float}  
\def\figlabel#1{\label{#1}}
\newtheorem{theorem}{Theorem}[section]
\newtheorem{lemma}[theorem]{Lemma}
\newtheorem{corollary}[theorem]{Corollary}

\newtheorem{remark}[theorem]{Remark}

\newtheorem{proposition}[theorem]{Proposition}

\newcommand{\be}{\begin{equation}}
\newcommand{\ee}{\end{equation}}
\newcommand{\eps}{\varepsilon}
\newcommand{\ga}{\gamma}
\newcommand{\dps}{\displaystyle}
\newcommand{\RR}{\mathbb{R}}
\newcommand{\NN}{\mathbb{N}}
\newcommand{\CC}{\mathbb{C}}
\newcommand{\TT}{\mathbb{T}}
\newcommand{\ZZ}{\mathbb{Z}}
\newcommand{\MM}{\mathcal{M}}
\newcommand{\WW}{\mathcal{W}}

\newcommand{\GG}{\mathcal{G}}

\newcommand{\BB}{\mathcal{B}}
\newcommand{\LL}{\mathcal{L}}

\newcommand{\DD}{\mathcal{D}}
\newcommand{\XX}{\mathcal{X}}
\newcommand{\OO}{\mathcal{O}}
\newcommand{\FF}{\mathcal{F}}

\newcommand{\VV}{\mathcal{V}}

\newcommand{\TTT}{\mathcal{T}}

\newcommand{\EE}{\mathcal{E}}
\newcommand{\JJ}{\mathcal{J}}

\newcommand{\CCC}{\mathcal{C}}

\newcommand{\ii}{^{-1}}
\newcommand{\de}{\delta}
\newcommand{\pa}{\partial}
\newcommand{\la}{\lambda}

\newcommand{\kk}{\kappa}
\newcommand{\rr}{\rho}
\newcommand{\si}{\sigma}

\newcommand{\ol}{\overline}

\renewcommand{\Re}{\mathrm{Re\, }}
\renewcommand{\Im}{\mathrm{Im\,}}
\newcommand{\wt}{\widetilde}

\newcommand{\al}{\alpha}

\begin{document}
\title{Exponentially and non-exponentially small splitting of separatrices for the pendulum with a fast
meromorphic perturbation}
\author{Marcel Guardia\thanks{\tt marcel.guardia@upc.edu} \ and Tere M. Seara\thanks{\tt tere.m-seara@upc.edu}}
\maketitle
\medskip
\begin{center}$^{*}$
Department of Mathematics\\
Mathematics Building, University of Maryland\\
College Park, MD 20742-4015, USA 
\end{center}
\smallskip
\begin{center}$^\dagger$
Departament de Matem\`atica Aplicada I\\
Universitat Polit\`ecnica de Catalunya\\
Diagonal 647, 08028 Barcelona, Spain
\end{center}

\begin{abstract}
In this paper we study the  splitting of separatrices phenomenon which arises when one considers a Hamiltonian System of one degree of freedom with a fast periodic or quasiperiodic and meromorphic in the state variables perturbation. The obtained results are different from the previous ones in the literature, which  mainly assume algebraic or trigonometric polynomial dependence on the state variables. As a model, we consider the pendulum equation with several meromorphic perturbations and we show the sensitivity of the size of the splitting on the  width of the analyticity strip of the perturbation with respect to the state variables. We show that the size of the splitting is exponentially small if the strip of analyticity is wide enough. Furthermore, we see that the splitting
 grows as the width of the analyticity strip shrinks, even becoming non-exponentially small for very narrow strips. Our results prevent from using polynomial truncations of the meromorphic perturbation to compute the size of the splitting of separatrices.
\end{abstract}

\section{Introduction}\label{sec:intro}
Exponentially small splitting of separatrices appears in analytic dynamical systems  with different time scales. A paradigmatic example are analytic Hamiltonian systems of one degree of freedom with a fast non-autonomous periodic or quasiperiodic 
perturbation. Namely, systems of the form 
\begin{equation}  \label{eq:GeneralHam}
\begin{split}
H\left(x,y,\frac{t}{\eps}\right)&=H_0(x,y)+\mu\eps^{\eta}H_1\left(x,y,\frac{t}{
\eps}\right)\\
&=\frac{y^2}{2}+V(x)+\mu\eps^{\eta}H_1\left(x,y,\frac{t}{\eps}\right),
\end{split}
\end{equation}
where $\eps>0$ is a small parameter, $\mu\in\RR$, $\eta\geq 0$, and  $H_0$ has a hyperbolic critical point
whose invariant manifolds coincide along a separatrix.

This phenomenon was first pointed out by Poincar\'e \cite{Poincare99} but it was
not until the last decades  when this problem started to be
studied rigorously (see for instance \cite{HolmesMS88, HolmesMS91,DelshamsS92, Fontich93, ChierchiaG94,
Gelfreich94, Fontich95, Sauzin95,DelshamsGJS97, DelshamsS97, Gelfreich97,
Treshev97, GalGM99, Gelfreich00, Sauzin01,DelshamsGS04, BaldomaF04, BaldomaF05, Baldoma06, Olive06,
GuardiaOS10, BaldomaFGS11}). 
Nevertheless, the results which obtain asymptotic formulas for the splitting only deal with
Hamiltonian systems whose perturbation is an algebraic or trigonometric polynomial with respect to the state variables  $x$ and $y$. In all these cases the splitting of separatrices
is exponentially small with respect to the parameter $\eps$. Moreover, the
imaginary part of the complex singularity of the time-parameterization of the
unperturbed separatrix closest to the real axis plays a significant role. 


All the previous works dealing with the periodic case, show that  under certain non-degeneracy conditions, the distance between the invariant manifolds is of order 
\begin{equation}\label{eq:OrdreSplitting}
 d\sim \mu\eps^q e^{-\dps\tfrac{a}{\eps}},
\end{equation}
where $a$ is the imaginary part of the complex singularity of the time-parameterization of the
unperturbed separatrix closest to the real axis and $q\in\RR$. Moreover, for $\eta>\eta^*$, where $\eta^*$  depends on the properties of both $H_0$ and $H_1$, the splitting
is well predicted by the Poincar\'e-Arnol'd-Melnikov method (see
\cite{Melnikov63}, and \cite{GuckenheimerH83} for a more modern exposition of this method). This
case is usually called \emph{regular case}. In the \emph{singular case}
$\eta=\eta^*$ the splitting is exponentially small as \eqref{eq:OrdreSplitting} but the first order does not
coincide with the Melnikov prediction (see \cite{BaldomaFGS11} and references therein). In the quasiperiodic case, under certain hypotheses, one can also show that the Melnikov method predicts correctly the splitting and the size of both the Melnikov function and the splitting depends strongly on $a$ \cite{DelshamsGJS97, Sauzin01, DelshamsGS04}. However, this case is much less understood and there are very few  results.

Nevertheless, many of the models known, for instance in celestial mechanics, are
not algebraic or trigonometric polynomials in the state variables but involve functions with a finite strip
of analyticity (see, for instance, \cite{SimoL80,Xia92, MartinezP94,FejozGKR11}). 
As far as the authors know, the only result dealing with the exponentially small splitting of separatrices in the periodic case for
non-entire perturbations is \cite{Gelfreich97a}. However, the author considers
models with a strip of analyticity very big with respect to $\eps$ so that he
can deal with them as if they were polynomial. In the quasiperiodic case, as far as the authors know, there are not previous results.

The goal of this paper is to study how the splitting of separatrices behavior
depends on the width of the analyticity strip when one considers a meromorphic
perturbation. Essentially, we see that the size of the splitting depends
strongly on this width and that, in general, the singularity of the separatrix
does not play any role in this size. We consider also the case when the strip
tends to infinity as $\eps\rightarrow 0$ and we see how the size of the splitting tends to
the size known for the entire cases. In the other limiting case, namely when the strip of analyticity shrinks to the real line as $\eps\rightarrow 0$, we  see that even if the perturbation is still analytic, the splitting becomes algebraic in $\eps$ both in the periodic and the quasiperiodic case.

We focus our study in particular examples, which allow us to analyze in great
detail the behavior of the splitting. Nevertheless, we expect the same to happen
for fairly general systems.

We work with time periodic and quasiperiodic perturbations of the classical pendulum. More concretely, we consider the following model,
\begin{equation}\label{def:ToyModel:ode}
\ddot x=\sin x+\mu\eps^\eta \frac{\sin x}{(1+\al\sin x)^2 }f\left(\frac{t}{\eps}\right),
\end{equation}
where $f(\tau)$ is an analytic function which depends either periodically or quasiperiodically on $\tau$.

The associated system
\begin{equation}\label{def:sistema}
 \left\{
\begin{split}
\dot x &=y\\
\dot y&=\sin x+\mu\eps^\eta \frac{\sin x}{(1+\al\sin x)^2 }f\left(\frac{t}{\eps}\right)\\
\end{split}
\right.
\end{equation}
is Hamiltonian with Hamiltonian function
\begin{equation}\label{def:ToyModel:Hamiltonian}
H\left(x,y,\frac{t}{\eps}\right)=\frac{y^2}{2}+\cos x-1 +\mu\eps^\eta
\psi(x)f\left(\frac{t}{\eps}\right),
\end{equation}
where $\psi(x)$ is defined by $\psi'(x)=-\sin x/(1+\al\sin x)^2$ and $\psi(0)=0$.
Here $\al\in[0,1)$ is a parameter which changes the width of the analyticity strip of $\psi$, which is given by
\begin{equation}\label{def:Strip}
 |\Im x|\leq \ln \left(\frac{1+\sqrt{1-\al^2}}{\al}\right).
\end{equation}
When $\alpha=0$, the system is entire in $x$ and $y$ and has been previously studied  for particular choices of $f$ in \cite{Treshev97,DelshamsGJS97, OliveSS03, Olive06}, whereas when $\al=1$, $\psi$ is not defined in $x=3\pi/2$. In this paper, we consider any $\al\in(0,1)$ either independent of or dependent on $\eps$.

In the periodic case, as experts know, the only important property to obtain the asymptotic formula for the splitting is that one  has to require that the first harmonics of $f$ are different from zero. Thus, we choose
\[
 f(\tau)=\sin\tau.
\]
Dealing with any other function with non-zero first Fourier coefficients is analogous. 
In this setting, one can rephrase system \eqref{def:sistema} as a Hamiltonian system of two degrees of freedom considering $\tau=t/\eps$ as a new angle and $I$ its conjugate action, which gives the Hamiltonian
\begin{equation}\label{def:ToyModel:Hamiltonian:2dof}
\begin{split}
K\left(x,y,\tau,I\right)=&\frac{I}{\eps}+H\left(x,y,\tau\right)\\
=&\frac{I}{\eps}+\frac{y^2}{2}+\cos x-1 +\mu\eps^\eta
\psi(x)\sin\tau.
\end{split}
\end{equation}


In the quasiperiodic case, we consider the same model
\eqref{def:ToyModel:ode} with $f(\tau)=F(\tau,\gamma \tau)$, where 
\begin{equation}\label{def:GoldenMean}
 \ga=\frac{\sqrt{5}+1}{2}
\end{equation}
is the golden mean number and $F:\TT^2\rightarrow \RR$. Note that if one takes $\al=0$, one recovers the model considered in \cite{DelshamsGJS97}. On the function $F$ we assume the same  hypotheses that are assumed in that article. Namely, if one considers its Fourier expansion in the angles $\theta=(\theta_1,\theta_2)$,
\begin{equation}\label{def:F:Fourier}
 F(\theta_1,\theta_2)=\sum_{k\in\ZZ^2}F^{[k]}e^{ik\cdot \theta},
\end{equation}
we assume that there exist constants $r_1,r_2>0$ such that
\begin{equation}\label{hyp:QP1}
 \sup_{k=(k_1,k_2)\in\ZZ^2}\left|F^{[k]}e^{r_1|k_1|+r_2|k_2|}\right|<\infty.
\end{equation}
Furthermore, we assume that there exist $a$ and $k_0$ such that
\begin{equation}\label{hyp:QP2}
F^{[k]}>ae^{-r_1|k_1|-r_2|k_2|}
\end{equation}
for all $|k_1|/|k_2|$ which are continuous fraction convergents of $\gamma$ and $|k_2|>k_0$. An example of function satisfying these hypotheses is 
\[
 F(\theta_1,\theta_2)=\frac{\cos\theta_1\cos\theta_2}{(\cosh r_1-\cos\theta_1)(\cosh r_2-\cos\theta_2)}.
\]

Introducing the angle coordinates $(\theta_1,\theta_2)$ and their conjugate actions $(I_1,I_2)$, system \eqref{def:sistema} can be seen as a 3 degrees of freedom Hamiltonian System with Hamiltonian
\begin{equation}\label{def:ToyModel:Hamiltonian:QP}
\begin{split}
K\left(x,y,\theta,I\right)&=\frac{\omega \cdot I}{\eps}+H\left(x,y,\theta\right)\\
&=\frac{\omega \cdot I}{\eps}+\frac{y^2}{2}+\cos x-1 +\mu\eps^\eta
\psi(x)F(\theta_1,\theta_2),
\end{split}
\end{equation}
where $\omega=(1,\gamma)$ is the frequency vector.

We have chosen these particular models for several reasons. First, the
hyperbolic critical point $(0,0)$ of the unperturbed pendulum persists when the perturbation is added.
This fact is not crucial but simplifies the computations. Second, with the chosen function $\psi$, the size of the Melnikov function depends on the strip of analyticity of the perturbations, as is expected to happen for general systems. In Remark 
 \ref{remark:Singularitats} in Section
\ref{sec:Melnikov}, we consider the non-generic model
\[
\ddot x=\sin x+\mu\eps^\eta \frac{\sin x}{(1-\al\cos x)^2 }\sin\frac{t}{\eps},
\]
which has the same strip of analyticity \eqref{def:Strip}. However, due to certain cancellations, the size of Melnikov function does not depend on this strip. 

In the quasiperiodic case, we have chosen a very specific function $F$. On one hand, we have chosen the frequency vector $\omega=(1,\ga)$, where $\ga$ is the golden mean \eqref{def:GoldenMean}. The size of the splitting strongly depends on the diophantine properties of the chosen frequency  and, in fact, its rigorous study has been only done, as far as the authors know, for quadratic frequencies (see \cite{Sauzin01, LochakMS03, DelshamsG03}). On the other hand, the chosen function $F$ has finite strip of analyticity in the angles $(\theta_1,\theta_2)$. This is the only kind of systems for which it is known that the Melnikov function predicts correctly the size of the splitting (see \cite{Simo94, SimoV01}). In fact, in the quasiperiodic case, the width of the strip of analyticity of $F$ also plays a crucial role in the size of the splitting.

Finally, as we have already explained, this particular
choice  of the perturbation makes everything easily computable. This allows us to
obtain explicit formulas for the first order of the splitting of separatrices, using the Melnikov function,
and see how it depends on the width of the analyticity strip of $\psi$. Then, we can compare our results for $\al$ small with the existing previous ones for $\al=0$.

As we have already said, when $\al=1$ system \eqref{def:sistema} is not defined at $x=3\pi/2$. Therefore, it has no sense to study the splitting problem for $\al$ too close to 1. Indeed, the perturbation is small in the real line provided
\[
 \frac{\eps^\eta}{(1-\al)^2}\ll 1.
\]
Nevertheless, as usually happen in the exponentially small splitting problems (see \cite{GuardiaOS10}), we will see that the splitting problem has sense under the slightly weaker hypothesis
\[
 \frac{\eps^{\eta}}{(1-\al)^{3/2}}\ll1.
\]

\begin{figure}[h]
\begin{center}
\psfrag{xx}{$x$}\psfrag{xxxxx}{$y$}
\includegraphics[height=6cm]{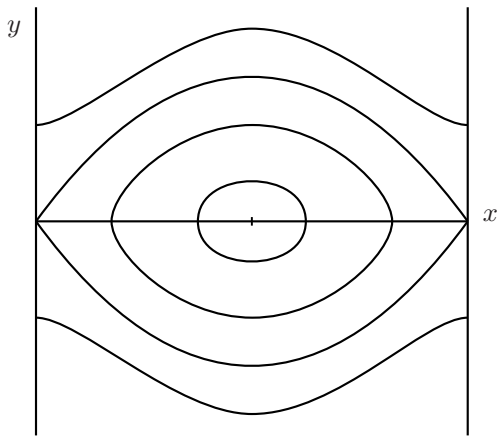}
\end{center}
\caption{\figlabel{fig:Pendol} Phase portrait of the pendulum. In it, one can see the two separatrices. We focus our study on the upper one.}
\end{figure}

When $\mu=0$, the system is the classical pendulum. It has a hyperbolic critical point at $(0,0)$ whose invariant manifolds
coincide along two separatrices (see Figure \ref{fig:Pendol}). We focus our attention on the
positive one, which can be parameterized as
\begin{equation}\label{def:separatrix}
 x_0(u)=4\arctan\left(e^u\right),\,\,\,\, y_0(u)=\dot x_0(u)=\frac{2}{\cosh u},
\end{equation}
whose singularities are at $u=i\pi/2+ik\pi$, $k\in\ZZ$.

The goal of this paper is to measure how this separatrix splits when one
takes $\mu>0$, paying special attention on how this splitting depends on the relative size
between $\al$ and $\eps$. 
The previous results in the periodic case \cite{Treshev97,OliveSS03, Olive06} consider $\al=0$. They  obtain an asymptotic formula for the distance between the invariant manifolds, which is  of the form
\[
 d\sim \mu\eps^{\eta-2}e^{-\dps\tfrac{\pi}{2\eps}}.
\]
In this paper, we see that in the periodic case
\begin{itemize}
 \item If $\al\leq\eps^2$, that is when the strip of analyticity is of order $2\ln(1/\eps)$ or bigger, the distance between the invariant manifolds coincides with the results obtained in the case $\al=0$, when the perturbation is a trigonometric polynomial.
\item If $\eps^2\ll\al$ and $1-\al\gg \eps^2$, that is when the strip of analyticity is between $2\ln(1/\eps)$ and $\eps$, the distance is exponentially small but the first order is given by a completely different formula 
\[
 d\sim\mu\eps^qe^{-\dps\tfrac{c}{\eps}},
\]
where $q\in\RR$ and $0<c<\pi/2$ are constants which depend on $\al$. Therefore, the splitting is bigger than it was in the polynomial case. Moreover, we see that it monotonously increases with $\al$.
\item Finally, if $0<1-\al\geq\eps^2$, that is when the strip of analyticity is narrower than $\eps$, the distance between the invariant manifolds is non-exponentially small,
\[
 d\sim\mu\eps^q,
\]
where $q\in\RR$ is a constant which depend on $\al$.
\end{itemize}
In particular, the second statement shows that even if $\al$ is a small parameter, our study prevents of using the classical approach in perturbation theory. It consists in expanding the perturbation term in powers of $\al$ and studying the splitting of the first order, which has a polynomial perturbation. We will see that even for $\al\sim\eps$, this approach  leads to a wrong result for the splitting. Therefore, we show that  the study of the splitting of separatrices for meromorphic perturbations cannot be reduced to the study of simplified polynomial models. In the quasiperiodic case, an analogous phenomenon happens.

The structure of the paper goes as follows. First in Section \ref{sec:Melnikov}
we give the main results considering the periodic case. In Proposition \ref{prop:Melnikov} we show the
behavior of the Melnikov function with respect to $\al$. In Corollary \ref{coro:NonExpSmall} we analyze the range of the parameter $\al$ for which the Melnikov function is not exponentially small. Then, in Theorems \ref{th:Main} and \ref{th:Main:NonExp},  we show for which range of parameters
$\eps$, $\al$ and $\eta$, the Melnikov function predicts correctly the splitting.

In Section \ref{sec:QP} we consider the quasiperiodic case. First in Proposition \ref{prop:Melnikov:QP}    we study the size of the Melnikov function and in Theorem \ref{th:Main:QP} we prove its validity. As in the periodic case, we also show that if $1-\al\geq\eps^2$ both the Melnikov function and the splitting of separatrices are not exponentially small. These results are given in Corollary \ref{coro:NonExpSmall:QP} and Theorem \ref{th:Main:NonExp:QP} respectively.

Section \ref{sec:SketchProof} is  devoted to
prove Theorems \ref{th:Main} and \ref{th:Main:NonExp}, and Section \ref{sec:SketchProof:QP} is devoted to prove Theorems \ref{th:Main:QP} and \ref{th:Main:NonExp:QP}. 

Finally, in Appendix \ref{sec:SingularCase} we give some heuristic ideas
about how to deal with the so-called singular  periodic case, namely, when the parameter $\eta$ reaches a certain limiting value and therefore the Melnikov function does not predict correctly the size of the splitting.

\subsection{Heuristics on the relation between regularity and Arnol'd diffusion}

Even if in this paper we study the so called isochronous case, where the frequency of the perturbation is fixed, the same kind of study would
apply to an anisochronous case (see \cite{Sauzin01}). In the anisochronous case one can encounter both situations, the case of  rationally dependent frequencies, which leads  to a periodic in time perturbation, and the case of rationally independent frequencies, which  leads to a quasiperiodic perturbation.

Anisochronous systems have attracted a lot of attention because they are good models to study the phenomenon called Arnol'd diffusion.
The name  comes from the fact that was V. Arnol'd who produced in 1964 \cite{Arnold64} the first example showing a possible mechanism that leads to global instabilities in nearly  integrable Hamiltonian Systems.

Arnol'd proved the presence of instabilities in the following particular model
\[
H(I_1,I_2,\varphi_1,\varphi_2,t)=\frac{I_1^2}{2}+\frac{I_2^2}{2}+ \eps (\cos \varphi_1-1)+ \mu \eps (\cos \varphi_1-1) (\sin \varphi_2+ \sin t),
\]
showing the existence of orbits whose action $I_2$ changes drastically.

Nevertheless, it is expected that instabilities exist in fairly general nearly integrable Hamiltonian Systems.   Chierchia and Gallavotti in \cite{ChierchiaG94} proposed the study of the following generalization of the Arnol'd model
\begin{equation}\label{def:GeneralizedArnold}
H(I_1,I_2,\varphi_1,\varphi_2,t)=\frac{I_1^2}{2}+\frac{I_2^2}{2}+ \eps (\cos \varphi_1-1)+ \mu \eps h(\varphi_1,\varphi_2, t;\eps).
\end{equation}
In this setting, they coined the terminology \emph{a priori stable} versus \emph{a priori unstable}. A priori stable refers to consider $\mu =\eps^\eta$ with $\eta\geq 0$ and  a priori unstable refers to  $\eps=1$ and $\mu$ small. In the a priori stable case the unperturbed system, $\eps=0$ is completely integrable in the sense that it is written in global action-angle variables. In the a priori unstable case the unperturbed system, $\mu=0$, even if it is integrable in the sense that it has conserved quantities, presents some hyperbolicity, namely has partially hyperbolic tori with homoclinic trajectories.

There have been some recent works proving the existence of instabilities for a priori unstable systems using the ideas proposed by V. Arnol'd (see \cite{DelshamsLS06a, DelshamsH09} and see \cite{ChengY04,Treschev04, Bernard08} for proofs using other methods).
One of the main steps in the proof is to  see that the stable and unstable manifolds of the partially hyperbolic tori, which coincide when $\mu=0$, split producing chains of heteroclinic orbits. To detect the splitting of these invariant manifolds becomes an essential step in these geometric methods.
The main reason that makes  very difficult to detect this splitting when $\mu$ and $\eps$ are small is that one expects this splitting to be exponentially small in $\eps$ and, thus, very difficult to study.  This was the origin of the distinction between a priori stable and unstable systems,
Arnol'd overcame this difficulty taking the parameter $\mu$ exponentially small in $\eps$. It is an open problem to see whether Arnol'd mechanism works in the a priori stable setting, namely taking $\mu=\eps^\eta$ with $\eta\geq 0$.

Nonetheless, the regularity of the system plays a crucial role in the classification between a priori stable and unstable systems.
The first observation, that is commonly accepted, is that in the $\CCC^r$ case, even in the a priori stable setting, the splitting is not exponentially small. Nevertheless, as far as the authors know, the only proof of this fact is the paper \cite{DelshamsGJS99} where it is seen that for quasiperiodic $\CCC^r$ perturbations  the splitting is polynomially small in $\eps$. Therefore the distinction between a priori stable and  unstable systems regarding the splitting problem only has sense for analytic perturbations. Moreover, the results stated in Theorems \ref{th:Main:NonExp} and \ref{th:Main:NonExp:QP} show that for analytic perturbations with narrow strip of analyticity, the size of the splitting is not exponentially small. In particular, we see that the splitting is polynomial in $\eps$ in both the periodic and the quasiperiodic case. Therefore, one would expect that for  anisochronous systems \eqref{def:GeneralizedArnold} with a narrow strip of analyticity the splitting at the resonances  and in the nonresonant zones are of the same order and polynomial with respect to $\eps$.

On the other hand, it is a commonly accepted fact that the bigger the size of the splitting the faster the diffusion.  In fact, Nekhroshev type lower bounds  of the diffusion time (or upper bounds of the stability time) increase with the regularity of the system \cite{Nekhoroshev77,MarcoS02, Bounemoura10} and this agrees with the fact that the splitting decreases with the regularity.
For analytic Hamiltonians, these bounds are bigger for bigger strips of analyticity \cite{Poschel93,DelshamsG96a}. The results in this paper show that this is also consistent with the fact that the splitting decreases when the strip of analyticity increases.

\section{The Melnikov function and its
validity in the periodic case}\label{sec:Melnikov}

As we want to deal with a time periodic perturbation  of the pendulum equation, the dynamics is better understood in the three-dimensional extended phase space $(x,y,\tau)\in \TT\times\RR\times\TT$. In this space, $\Lambda=\{(0,0,\tau);\tau\in\TT\}$ is a hyperbolic periodic orbit and, for $\mu=0$,  its 2-dimensional stable and unstable invariant manifolds coincide along the homoclinic manifold 
\[
 \WW^u(\Lambda)=\WW^s(\Lambda)=\left\{(x,y,\tau): H_0(x,y)=0\right\}=\left\{(x,y,\tau)=(x_0(u),y_0(u),\tau); (u,\tau)\in\RR\times\TT\right\},
\]
where $(x_0(u),y_0(u))$ is the parameterization of the separatrix given in \eqref{def:separatrix}.

Our goal is to study how these manifolds split when $\mu\neq 0$. To this end we need to introduce some notion of distance between them. As the manifolds are graphs for $\mu=0$, the same happens for $\mu\eps^\eta>0$ small enough  in suitable domains. More concretely, in Section \ref{sec:SketchProof} we will see that one can parameterize the perturbed stable and unstable manifolds as
\[\left\{
 \begin{split}
  x&=x_0(u)\\
y&=y^{u,s}(u,\tau)\\
\tau&=\tau.
 \end{split}
\right. 
\]
Here, $(u,\tau)\in (-\infty, U)\times\TT$, for certain $U>0$, for the unstable manifold and $(u,\tau)\in (-U,\infty)\times\TT$ for the stable one. Taking into account that the invariant manifolds are Lagrangian, in Section \ref{sec:SketchProof} 
we use the Hamilton-Jacobi equation to see that the functions $y^{u,s}$ can be given as 
\[
 y^{u,s}(u,\tau)=\frac{1}{y_0(u)}\pa_u T^{u,s}(u,\tau),
\]
for certain generating functions $T^{u,s}$.

Therefore, a natural way to measure the difference between the manifolds is to compute 
\[
 D(u,\tau)=\pa_u T^{s}(u,\tau)-\pa_u T^{u}(u,\tau).
\]
If one considers a perturbative approach taking $\mu$ as small parameter, one can easily see that the first order in $\mu$ of the function $D(u,\tau)$ is given by the Melnikov function
\begin{equation}\label{def:Melnikov0}
\MM\left(u,\tau\right)=\int_{-\infty}^{+\infty} \left\{H_0,H_1\right\}\left(x_0(u+s),y_0(u+s),\tau+\frac{s}{\eps}\right)ds.
\end{equation}
In other words, one has that 
\begin{equation}\label{def:PrimerOrdreMu}
 D(u,\tau)=\mu\eps^\eta \MM(u,\tau)+\OO\left(\mu^2\eps^{2\eta}\right).
\end{equation}
To see that the manifolds split, we can choose a transversal section to the unperturbed separatrix to measure their distance. The simplest one is $x=\pi$, which corresponds to compute $D(0,\tau)$ (see \eqref{def:separatrix}). Then, the zeros of $D(0,\tau)$ correspond to homoclinic orbits of the perturbed system and the distance between the invariant manifolds is given by
\begin{equation}\label{def:distancia}
 d(\tau)=y^s(0,\tau)-y^u(0,\tau)=\frac{1}{2}D(0,\tau).
\end{equation}
Nevertheless, when $\al$ is close to 1, namely $\al=1-C\eps^r$ with $r>0$ and $C>0$, the perturbative term in \eqref{def:ToyModel:ode} has a non uniform bound for $x\in [0,2\pi]$. Indeed, its maximum (in absolute value), which is $\mu\eps^\eta/(1-\al)^2=\mu\eps^{\eta-2r}/C^2$, is reached at $x=3\pi/2$. Therefore, it is natural to expect that the invariant manifolds of the perturbed system remain $\mu\eps^\eta$-close to the unperturbed separatrix only before they reach a neighborhood of the section $x=3\pi/2$. For this reason, in this case we will measure the distance in this section, where one expects that the perturbed manifolds are closer. This section, by \eqref{def:separatrix}, corresponds to $u=\ln (1+\sqrt{2})$. Therefore, we define the  $\wt d(\tau)$, the distance between the perturbed invariant manifolds at the section $x=3\pi/2$, which is given by
\begin{equation}\label{def:distancia:moguda}
 \wt d(\tau)=y^s\left(\ln \left(1+\sqrt{2}\right),\tau\right)-y^u\left(\ln \left(1+\sqrt{2}\right),\tau\right)=2\sqrt{2}D\left(\ln \left(1+\sqrt{2}\right),\tau\right).
\end{equation}
For $\al$ close to 1, one would expect that the distance $d(\tau)$ is bigger than $\wt d(\tau)$ since the manifolds deviate from the homoclinic after they cross the section $x=3\pi/2$. Nevertheless, when $\al$ is not close to 1, both quantities are equivalent.

\subsection{The Melnikov function}
In this section we compute the Melnikov function associated to Hamiltonian \eqref{def:ToyModel:Hamiltonian} with $f(\tau)=\sin\tau$ and we study its dependence on $\eps$ and $\al$. Let us recall that the Melnikov function  \eqref{def:Melnikov0}  is given by 
\begin{equation}\label{def:Melnikov}
\MM\left(u,\tau;\eps,\al\right)=4\int_{-\infty}^{+\infty} 
\frac{\sinh(u+s)\cosh(u+s)}{\left(\cosh^2(u+s)-2\alpha\sinh(u+s)\right)^2}
\sin\left(\tau+\frac{s}{\eps}\right)ds.
\end{equation}
As it is well known, the evaluation of this
integral can be done using Residuum Theory. To this end, one has to look for
the  poles of the function 
\begin{equation}\label{def:MelnikovIntegrand}
 \beta(u)= \frac{\sinh(u)\cosh(u)}{\left(\cosh^2(u)-2\alpha\sinh(u)\right)^2}
\end{equation}
closest to the real axis. For $\alpha=0$, $\beta(u)$ has order 3 poles
at $\rr_\pm^0 =\pm i\pi/2$. Nevertheless, for $\alpha>0$ these poles bifurcate
into a combination of zeros and poles. Since the size of the Melnikov function
\eqref{def:Melnikov} depends on the location of these poles, one has to study
their dependence on $\alpha$. We will see that the relative size between the parameter $\al$ and the period $2\pi\eps$ of the perturbation will lead to a significantly different size of
\eqref{def:Melnikov}.

Considering the denominator of $\beta(u)$, namely $(\cosh^2(u)-2\alpha\sinh(u))^2$, one can easily see
that the poles of \eqref{def:MelnikovIntegrand} are the solutions of
\[
 \sinh u=\al\pm i\sqrt{1-\al^2},
\]
This equation has four families of solutions given by $\rr_-+2\pi ki$,
$\rr_++2\pi ki$, which are solutions of 
\[
 \sinh u=\al+ i\sqrt{1-\al^2}.
\]
and their conjugate families  $\ol \rr_-+2\pi ki$ and $\ol \rr_++2\pi ki$ for $k\in\ZZ$. The singularities
$\rr_-$, $\rr_+$, $\ol \rr_-$ and $\ol\rr_+$ are the closest to the real axis
(see Figure \ref{fig:Singularitats}) and they satisfy
\[
0<\Im \rr_-<\frac{\pi}{2}<\Im \rr_+<\pi,                                 
\]
for $\al\in (0,1)$.

\begin{figure}[h]
\begin{center}
\psfrag{x}{$i\frac{\pi}{2}$}\psfrag{xx}{$-i\frac{\pi}{2}$}
\psfrag{zzzzz}{$\RR$}\psfrag{xxx}{$\Im u=\pi$}\psfrag{xxxx}{$\Im u=\pi$}
\psfrag{z}{$\rr_-$}\psfrag{zz}{$\rr_+$}\psfrag{zzz}{$\ol\rr_-$}\psfrag{zzzz}{$\ol\rr_+$}
\includegraphics[height=6cm]{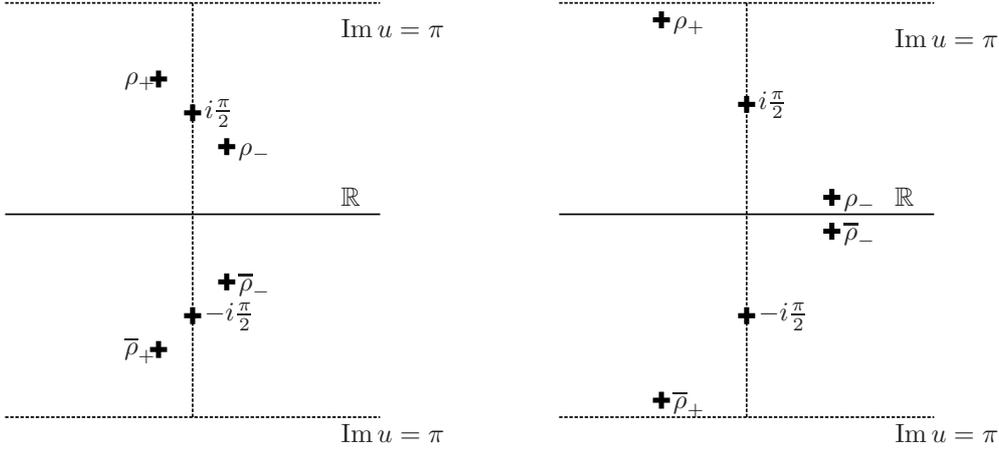}
\end{center}
\caption{\figlabel{fig:Singularitats} Location of the singularities $u=\rr_-,\rr_+,\ol\rr_-,\ol\rr_-$. In the left picture we show them for $\al$ small, in such a way that they are close to the singularities of the separatrix of the pendulum $u=\pm i\pi/2$. In the right picture we show them for $\al$ close to 1. Then, $\rr_-$ and $\ol \rr_-$ approach the real axis while $\rr_+$ and $\ol\rr_+$ approach the lines $\Im u=\pi$ and $\Im u=-\pi$ respectively.}
\end{figure}

In particular, if one considers $\al$ as a small parameter,
they have the expansions
\begin{equation}\label{eq:Expansio:Sing}
\rr_\pm(\al)=i\frac{\pi}{2}\pm (-1+i)\sqrt{\al}+\OO(\al)\,\,\text{ as }\al\rightarrow 0. 
\end{equation}
In the other limiting regime $\al\rightarrow 1$, their expansions are
\begin{equation}\label{eq:Expansio:Sing:alfa1}
\begin{split}
\rr_-(\al)&=\ln(1+\sqrt{2})+i (1-\al)^{1/2}+\OO(1-\al)\\
\rr_+(\al)&=-\ln(1+\sqrt{2})+\pi i-i (1-\al)^{1/2}+\OO(1-\al)
\end{split}\,\,\text{ as }\al\rightarrow 1.
\end{equation}
Then, one can see that the function $\beta$ for $-\pi<|\Im u|<\pi$ and $-1<\Re u<1$  behaves as 
\begin{equation}\label{def:beta:singularity}
\beta(u)\sim \frac{(u-i\pi/2)(u+i\pi/2)}{(u-\rr_-)^2(u-\rr_+)^2(u-\ol\rr_-)^2(u-\ol\rr_+)^2}.
\end{equation}
Using this fact, one can easily compute the size of the Melnikov function in
\eqref{def:Melnikov} using Residuum Theory.

\begin{proposition}\label{prop:Melnikov}
There exists $\eps_0>0$ such that $\eps\in (0,\eps_0)$ and $\alpha\in (0,1)$,
\begin{itemize}
 \item If $\al$ satisfies $0<\alpha\leq C\eps^\nu$ where $\nu>2$ and $C>0$ are
constants independent of $\eps$, the Melnikov function \eqref{def:Melnikov}
satisfies the asymptotic formula
\begin{equation}\label{def:Melnikov:BandaAmpla}
\MM\left(u,\tau;\eps,\al\right)=\frac{4\pi}{\eps^{2}}e^{\dps
-\tfrac{\pi}{2\eps}}\left(\cos(\tau-u/\eps)+\OO\left(\frac{\al}{\eps^2}, e^{\dps
-\tfrac{\pi}{2\eps}}
\right)\right).
\end{equation}
\item If $\al=\al_*\eps^2+\OO\left(\eps^3\right)$ for some constant $\al_*>0$, the Melnikov function \eqref{def:Melnikov}
satisfies the asymptotic formula
\begin{equation}\label{def:Melnikov:BandaIntermedia}
\MM\left(u,\tau;\eps,\al\right)=\frac{\la}{\eps^2}e^{\dps
-\tfrac{\pi}{2\eps}}
\sin\left(\tau-\phi_\ast-u/\eps\right)+\OO\left(\frac{1}{\eps}e^{\dps
-\tfrac{\pi}{2\eps}}\right)
\end{equation}
for certain constants $\la=\la(\al_*)$ and $\phi_*=\phi_*(\eps,\al_*)$.
\item If  $\alpha$  satisfies  $C\eps^\nu<\alpha<1$
where $\nu\in (0,2)$ and $C>0$ are constants independent of $\eps$, the Melnikov function \eqref{def:Melnikov}
satisfies the asymptotic formula
\begin{equation}\label{def:Melnikov:BandaEstreta}
\MM\left(u,\tau;\eps,\al\right)=\left|\frac{\de_2(\al)}{\eps}+\de_1(\al)\right|e^{\dps
-\tfrac{\Im\rr_-}{\eps}}
\sin\left(\tau-\phi-u/\eps\right)+\OO\left(e^{-\dps \tfrac{\pi}{2\eps}}\right).
\end{equation}
where $\de_i(\alpha)$ are given by
\begin{align}
 \de_1(\al)&=\frac{2\pi }{\left(1-\al^2\right)^{3/2}}\left(\sinh \rr_--i\left(1-\al^2\right)^{1/2}\right)\label{def:delta1}\\
 \de_2(\al)&=\frac{2\pi\sinh\rr_-}{\left(1-\al^2\right)\cosh\rr_-}\label{def:delta2}
\end{align}
with $\sinh\rr_-=\al+i\sqrt{1-\al^2}$, and $\phi=\phi(\eps,\al)$.

\end{itemize}
\end{proposition}

When $\al\ll\eps^2$, the first statement of Proposition \ref{prop:Melnikov} ensures that the Melnikov function is non-degenerate. In the case $\al\sim\eps^2$, one would have to analyze the behavior of the constant $\la(\al_*)$, which appears in formula \eqref{def:Melnikov:BandaIntermedia}, to check that it does not vanish. Even if this constant is computable in our example, we will not study it, since it is not the purpose of this paper to study this particular case. We just want to remark that for general $\al_*$, the first  asymptotic order of the Melnikov function has changed from  \eqref{def:Melnikov:BandaAmpla} to \eqref{def:Melnikov:BandaIntermedia}. In fact, as can be seen in formula \eqref{def:Melnikov:BandaEstreta}, $\al\sim\eps^2$ is the transition value in which the size of the Melnikov function changes drastically. Indeed, formula \eqref{def:Melnikov:BandaEstreta} and the expansion of $\rr_-$ given in \eqref{eq:Expansio:Sing:alfa1} show that the size of the Melnikov function becomes bigger when $\al$ approaches 1. This particular behavior will be analyzed in Section \ref{sec:MelnikovNonExp}.

Next corollary, whose proof is straightforward, analyzes the behavior of formula \eqref{def:Melnikov:BandaEstreta} for $\al$ in the range $\eps^2\ll\al\leq \al_0< 1$ for any fixed $\al_0\in (0,1)$. Let us observe that $\rr_-$ has the asymptotic expansion \eqref{eq:Expansio:Sing} when $\al\rightarrow 0$. Next corollary shows that the Melnikov function is exponentially small for this range of the parameter $\al$.

\begin{corollary}\label{coro:MidaMelnikov}
Let us fix any $\al_0\in (0,1)$. Then, the Melnikov function in \eqref{def:Melnikov:BandaEstreta} has the following asymptotic formulas for $\eps^2\ll\al\leq\al_0<1$.
\begin{itemize}
 \item If $\al$ satisfies $\alpha=C\eps^\nu$ for certain $\nu\in (0,2)$ and $C>0$, 
\[
 \MM\left(u,\tau;\eps,C\eps^\nu\right)=\frac{\left|\de_2^0\right|}{C\eps^{1+\nu/2}}e^{\dps
-\tfrac{\Im\rr_-}{\eps}}
\left(\sin\left(\tau-\phi-u/\eps\right)+\OO\left(\eps^{\nu/2}\right)\right
),
\]
where 
\[
\de_2^0=-\pi(1+i).
\] 
Therefore, in this case the Melnikov function is non-degenerate since $\de_2^0\neq 0$.
 \item If $\al\in (0,\al_0]$ is independent of $\eps$, 
\[
 \MM\left(u,\tau;\eps,\al\right)=\frac{|\de_2(\al)|}{\eps}e^{\dps
-\tfrac{\Im\rr_-}{\eps}}
\left(\sin\left(\tau-\phi-u/\eps\right)+\OO\left(\eps\right)\right
).
\]
It can be checked that $\de_2(\al)\neq0$ for any $\al\in (0,1)$, and therefore, in this case the Melnikov function is also non-degenerate.
\end{itemize}
\end{corollary}

\begin{remark}
If one takes $\al\sim\eps^\nu$ and $\nu\rightarrow 2$ in formula \eqref{def:Melnikov:BandaEstreta}
does not obtain formula \eqref{def:Melnikov:BandaIntermedia}. The reason is that  formula \eqref{def:Melnikov:BandaEstreta} only takes into account the
residuum of $\rr_-$ since the residuum of $\rr_+$ is exponentially small with
respect to the one of $\rr_-$. Nevertheless, this is not the case when
$\al\sim\eps^2$ (see \eqref{eq:Expansio:Sing}). Then, the residuums of both singularities $\rr_-$ and
$\rr_+$ make a contribution to the Melnikov function of the same exponentially small order $\OO(\eps^{-2}e^{-\pi/2\eps})$.
\end{remark}

\begin{remark}In the first statement of Proposition \ref{prop:Melnikov}, one can see that, when the
strip of analyticity \eqref{def:Strip} is \emph{wide enough}, namely taking $\al\sim\eps^\nu$ with
$\nu>2$, the Melnikov function at first order behaves as in the entire case
$\alpha=0$ \cite{Gelfreich97a}. In other words, the Melnikov function does not realize the finiteness
of the strip of analyticity  and the exponentially small coefficient is given by
the imaginary part of the singularities of the separatrix. 

On the other hand, when the strip of analyticity \eqref{def:Strip} is independent of $\eps$
or not extremely big with respect to $\eps$, namely taking $\alpha\sim\eps^\nu$
with $\nu\in [0,2)$, the exponentially small coefficient does not coincide with
the imaginary part of the singularity of the unperturbed separatrix. Instead it
is given by the imaginary part of this new singularity $\rr_-$, which appears
when one evaluates the perturbation along the unperturbed separatrix. Note
that even if one takes, for instance, $\al=\eps$, the strip of analyticity is of order $\ln(1/\eps)$ and one has that 
\[
\MM\left(u,\tau;\eps,\eps\right)\sim
\eps^{-\frac{3}{2}}e^{-\dps\tfrac{\pi-2\sqrt{\eps}}{2\eps}}.
\]
That is, even for perturbations with a wide strip of analyticity with respect to
$\eps$, it appears a correcting term in the exponential.

The case $\al\sim\eps^2$ is the boundary between these two
different behaviors. In this case, the exponential coefficient is given by the
imaginary part of the singularity of the unperturbed separatrix but the constant
in front of the exponential is not the one given in the first statement of
Proposition \ref{prop:Melnikov} but a different one, which is given in the second statement for the value  $\al=\al_*\eps^2+\OO(\eps^3)$.
\end{remark}

\begin{proof}[Proof of Proposition \ref{prop:Melnikov}]
The classical way to compute the Melnikov function is to change the path of integration to $\Im s=\pm\pi$ and apply Residuum Theory, using that the function $\beta$ has order two poles at $\rr_-$, $\rr_+$, $\ol\rr_-$ and $\ol\rr_+$ in the strip $-\pi<\Im s<\pi$. Nevertheless, to prove the first statement, instead of computing the residuums of both
$\rr_\pm$ and $\ol\rr_\pm$ and look for the cancellations, we just expand the Melnikov integral in power series of $\al$,
since it is uniformly convergent as a real integral. Thus, we obtain,
\[
 \MM\left(u,\tau;\eps,\al\right)=4\sum_{n=0}^\infty
(n+1)2^n\al^n\int_{-\infty}^{+\infty}
\frac{\sinh^{n+1}(u+s)}{\cosh^{2n+3}(u+s)}\sin\left(\tau+\frac{s}{\eps}
\right)ds.
\]
Then, the term for $n=0$ gives the first asymptotic order and can be computed using again Residuum Theory. To bound the other terms goes as follows. We Fourier-expand the terms in the series of the Melnikov integral in $\tau$ and we change the path of integration to $\Im s=\pi/2-\eps$ or $\Im s=-(\pi/2-\eps)$ depending on the harmonic. This gives us the exponentially small term $e^{-\frac{\pi}{2\eps}}$. To bound the other terms in the integral, we use that 
\[
 \left|\frac{\sinh^{n+1}(u+s)}{\cosh^{2n+1}(u+s)}\right|\leq \left(\frac{K}{\eps}\right)^{2n+1},
\]
for certain constant $K>0$ independent of $\eps$. Moreover using that 
\[
 \frac{1}{\left|\cosh^2(u+s)\right|}
\]
decays exponentially as $\Re s\rightarrow \pm\infty$ and that it has poles of order two $u+s=\pm i\pi/2$, one can see that
\[
 \int_{-\infty}^\infty \frac{1}{\left|\cosh^2(u+s\pm(i\pi/2-\eps))\right|}ds\leq \frac{K}{\eps}.
\]
Then, one has that 
\[
 \left|\int_{-\infty}^{+\infty}
\frac{\sinh^{n+1}(u+s)}{\cosh^{2n+3}(u+s)}\sin\left(\tau+\frac{s}{\eps}
\right)ds\right|\leq \left(\frac{K}{\eps}\right)^{2n+2}e^{-\dps\tfrac{\pi}{2\eps}}.
\]
Therefore, the remainder can be easily bounded provided $0<\al\ll\eps^2$. 

When $\al\sim\eps^2$ we cannot ensure that the preceding series is convergent and therefore we use directly Residuum Theory in the strip $0<\Im s<\pi$ or $-\pi<\Im s<0$ depending on the harmonic. For instance, for the positive harmonic, we take the strip $0<\Im s<\pi$, which contains the singularities $\rr_-$ and $\rr_+$, that are at a
distance of order $\OO(\eps)$ from $i\pi/2$ (see \eqref{eq:Expansio:Sing}). Then, it is enough to compute the residuum of $\beta(s)e^{is/\eps}$ at these singularities, which using \eqref{def:beta:singularity} satisfies
\[
 \mathrm{Res}\left(\beta(s)e^{is/\eps}, s=\rr_\pm\right)=\frac{A_\pm}{\eps^2}e^{-\dps\tfrac{\pi}{2\eps}}+\OO\left(\frac{1}{\eps}e^{-\dps\tfrac{\pi}{2\eps}}\right).
\]
for certain computable constants $A_\pm\in \CC$ independent of $\eps$.

In the case $\al\sim\eps^\nu$ with $0\leq \nu< 2$, we have that $\pi/2-\Im \rr_-, \Im \rr_+-\pi/2\sim \eps^{\nu/2} \gg \eps$. Therefore, 
\[
 e^{-\dps\tfrac{\Im \rr_+}{\eps}}\ll e^{-\dps\tfrac{\pi}{2\eps}}\ll e^{-\dps\tfrac{\Im \rr_-}{\eps}}.
\]
Using these facts, we can compute the Melnikov integral changing the path up to  $\Im s=\pi/2$ and we just need to consider the residuum at the singularity $\rr_-$, which can be explicitly computed. 
\end{proof}

\begin{remark}\label{remark:Singularitats}
All the singularities of the function $\beta(u)$ in
\eqref{def:MelnikovIntegrand} have different imaginary part for any $\al\in
(0,1)$. This is one of the reasons of the choice of the perturbation in
\eqref{def:ToyModel:ode}, since this is not always the case. Let us consider, for instance, the pendulum with a
different perturbation as
\begin{equation}\label{def:ToyModel2}
\ddot x= \sin x+\mu\eps^\eta \frac{\sin x}{(1-\al\cos x)^2}\sin\frac{t}{\eps},
\end{equation}
which has Hamiltonian function
\[
 \wt H\left(x,y,\frac{t}{\eps}\right)=\frac{y^2}{2}+\cos
x-1+\mu\eps^\eta\frac{1}{\al(1-\al\cos x)}\sin\frac{t}{\eps}.
\]
Then, the Melnikov function is given by
\[
 \wt
\MM\left(u,\frac{t}{\eps}\right)=4\int_{-\infty}^{+\infty}\frac{
\cosh(u+s)\sinh(u+s)}{\left((1-\al)\cosh^2(u+s)
+2\al\right)^2}\sin\left(\frac{t+s}{\eps}\right)
\]
To compute its size, one has to study the singularities of 
\[
 \wt\beta(u)=\frac{\cosh u\sinh u}{\left((1-\al)\cosh^2u +2\al\right)^2}.
\]
The closest ones to the real axis with positive
imaginary part are 
\[
 \wt \rr_\pm= i\frac{\pi}{2}\pm\mathrm{arcsinh}\sqrt{\frac{2\al}{1-\al}}
\]
and therefore both have the same imaginary part for any $\al\in (0,1)$. In these cases, the Melnikov function, which is given by the residuums of both
singularities, has the same size for any $\al$ satisfying $0\leq\al\leq \al_0<1$ for any $\al_0$ independent of $\eps$, but is divergent for $\al=1$.
\end{remark} 
\subsubsection{Narrow strip of analyticity: a drastic change on the size of the Melnikov function}\label{sec:MelnikovNonExp}

In Corollary \ref{coro:MidaMelnikov} we have seen that the Melnikov function is exponentially small with respect to $\eps$ provided $0<\al\leq \al_0< 1$ for any fixed $\al_0\in (0,1)$. We devote this section to study this function when $\al$ is close to 1, which will lead to a Melnikov function which is either exponentially small with a different exponential dependence on $\eps$ or even to a non-exponentially small Melnikov function.

We consider $\al=1-C\eps^r$ with $r>0$ and $C>0$. In this setting, one can see that the analyticity strip \eqref{def:Strip} of the Hamiltonian system \eqref{def:sistema} is very small, of order $\OO(\eps^{r/2})$. As a consequence, the singularity $\rr_-(\al)$ is very close to the real line. Indeed, from \eqref{eq:Expansio:Sing:alfa1}, one has
\[
\Im\rr_-(\al)=C^{1/2}\eps^{r/2}+\OO\left(\eps^r\right).
\]
\begin{corollary}\label{coro:NonExpSmall}
If $\al=1-C\eps^r$ with $r>0$ and $C>0$, the Melnikov function in \eqref{def:Melnikov:BandaEstreta} has the following asymptotic formulas
\begin{itemize}
\item If $0<r<2$
\[
 \MM\left(u,\tau;\eps,1-C\eps^r\right)=\frac{\pi}{\sqrt{2}C\eps^{1+r}}e^{\dps
-\tfrac{\Im\rr_-}{\eps}}
\left(\sin\left(\tau-\phi-u/\eps\right)+\OO\left(\eps^{r/2},\eps^{1-r/2}\right)\right
).
\]
\item If $r=2$,
\[
 \MM\left(u,\tau;\eps, 1-C\eps^2\right)=\frac{\pi  e^{-\sqrt{C}}}{\sqrt{2}C^{3/2}\eps^{3}}
\left(\sin\left(\tau-\phi-u/\eps\right)+\OO(\eps)\right),
\]
\item If $r>2$,
\[
 \MM\left(u,\tau;\eps, 1-C\eps^r\right)=\frac{\pi }{\sqrt{2}C^{3/2}\eps^{3r/2}}
\left(\sin\left(\tau-\phi-u/\eps\right)+\OO\left(\eps^{r/2-1}\right)\right),
\]
\end{itemize}
where $\phi=\phi(\eps,\al)$ is the constant given in Proposition \ref{prop:Melnikov}.

Therefore, in all these cases the Melnikov function is non-degenerate. Moreover, in the last two cases, it is non-exponentially small.
\end{corollary}

\begin{remark}
 To illustrate which size has the Melnikov function for the range of $\al$ considered in the first statement of Corollary \ref{coro:NonExpSmall}, we can take $\al=1-\eps$. Then, using the expansion \eqref{eq:Expansio:Sing:alfa1}, one can see that
\[
\MM\left(u,\tau;\eps,1-\eps\right)\sim
\eps^{-2}e^{-\dps\tfrac{1}{\sqrt{\eps}}}.
\]
Namely, the Melnikov function is still exponentially small but it has a different exponential dependence on $\eps$.
\end{remark}
\subsection{Validity of the Melnikov function}

Once we have computed the Melnikov function in Proposition \ref{prop:Melnikov} provided $0<\al\leq \al_0<1$ for a fixed $\al_0$, we can compute the prediction it
gives for the distance between the manifolds at the section $x=\pi$, which we call $d_0(\tau)$. Recall that by \eqref{def:PrimerOrdreMu} it is the first order in $\mu$ of the function  $d(\tau)$ given in \eqref{def:distancia}. 


If $\al$ satisfies $0<\al\leq C\eps^\nu$ where $\nu>2$ and $C>0$ are constants
independent of $\eps$, $d_0(\tau)$ is given by
\[
 d_0(\tau)=2\pi\mu\eps^{\eta-2}
e^{-\dps\tfrac{\pi}{2\eps}}\left(\cos\tau+\OO\left(\frac{\al}{\eps^2},e^{\dps
-\tfrac{\pi}{2\eps}}\right)\right).
\]
If $\al=\al_*\eps^2+\OO(\eps^3)$ for any constant $\al_*>0$, it is given by,
\[
 d_0(\tau)=\frac{\la(\al_*)}{2}\mu\eps^{\eta-2}e^{\dps
-\tfrac{\pi}{2\eps}}\sin(\tau-\phi_*)+\OO\left(\mu\eps^{\eta-1} e^{-\dps\tfrac{\pi}{2\eps}}
\right).
\]
Finally, if  $\al$ satisfies  $C\eps^\nu<\al\leq \al_0<1$ where
$\nu\in (0,2)$, $C>0$ and $\al_0\in(0,1)$ are constants independent of $\eps$, it is given by
\[
 d_0(\tau)=\frac{\left|\de_2(\al)\right|}{2}\mu\eps^{\eta-1} e^{\dps
-\tfrac{\Im\rr_-}{\eps}}\sin(\tau-\phi)+\OO\left(\mu\eps^{\eta} e^{\dps
-\tfrac{\Im\rr_-}{\eps}}\right).
\]
As we have already explained, direct
application of Melnikov theory only ensures that the distance between the manifolds is given
by
\[
 d(\tau)=d_0(\tau)+\OO\left(\mu^2\eps^{2\eta}\right).
\]

Therefore, the Melnikov function is the first order of the splitting
provided  $\mu$  is exponentially small with respect to $\eps$. However, it is
well known that often, even if Melnikov theory cannot be applied, the Melnikov
function is the true first order of the splitting (see \cite{HolmesMS88, DelshamsS92, Gelfreich94, DelshamsS97, Gelfreich97a, BaldomaF04, BaldomaF05, Gelfreich00, GuardiaOS10, BaldomaFGS11}). Next theorem shows that, under certain conditions, the Melnikov function gives the true first
order of the distance between the manifolds. We want to point out that the only available proof of the correct prediction of the Melnikov 
function for meromorphic perturbations is \cite{Gelfreich97a}, in which the author considers systems with analyticity strip wide enough with respect to $\eps$. All the other references deal with polynomial perturbations. Thus, this theorem is, as far as the authors know, the first one which shows the dependence of the size of the splitting on the width of the analyticity strip.

\begin{theorem}\label{th:Main}
Let us consider any $\mu_0>0$ and $\al_0\in (0,1)$. Then, there exists $\eps_0>0$ such that for
$|\mu|<\mu_0$, $\eps\in (0,\eps_0)$ and $\al\in (0,\al_0]$ such that 
$\eps^{\eta-1}(\eps+\sqrt{\al})$ is small enough,
\begin{itemize}
 \item If $\al$ satisfies $0<\alpha\leq C\eps^\nu$ where $\nu>2$ and $C>0$ are
constants independent of $\eps$, the invariant manifolds split and their distance at the section $x=\pi$ is given by
\[
d(\tau)=2\pi\mu\eps^{\eta-2}
e^{-\dps\tfrac{\pi}{2\eps}}\left(\cos\tau+\OO\left(\frac{\al}{\eps^2},
\mu\eps^\eta\right)\right).
\]
\item If $\al=\al_\ast\eps^2+\OO(\eps^3)$ for some constant $\al_\ast>0$ and the constant $\la(\al_\ast)$ introduced in Proposition \ref{prop:Melnikov} satisfies $\la(\al_\ast)\neq 0$, the invariant manifolds split and their distance at the section $x=\pi$  is given by
\[
d(\tau)=\frac{|\la(\al_\ast)|}{2}\mu\eps^{\eta-2}
e^{-\dps\tfrac{\pi}{2\eps}}\left(\sin(\tau-\phi_*)+\OO\left(
\mu\eps^\eta\right)\right).
\]
\item If  $\alpha$ satisfies   and $C\eps^\nu<\alpha<\al_0$
where $\nu\in (0,2]$ and $C>0$ are constants independent of $\eps$, the
invariant manifolds split and their distance  at the section $x=\pi$  is given by
\[
d(\tau)=\frac{\left|\de_2(\alpha)\right|}{2}\mu\eps^{\eta-1}e^{\dps
-\tfrac{\Im\rr_-}{\eps}}\left(\sin(\tau-\phi)+\OO\left(\mu\eps^{\eta-1}\sqrt{\al},\eps\right)\right).
\]
\end{itemize}

\end{theorem}
The proof of this theorem is deferred to Section \ref{sec:SketchProof}.
\begin{remark}\label{remark:SingularCase}
If one considers $\al\sim \eps^\nu$ with $\nu>2$ the condition
$\eps^{\eta-1}(\eps+\sqrt{\al})$  small enough is equivalent to requiring
$\eta>0$. This condition is the same that has to be required if one considers
the polynomial case $\al=0$ (see \cite{Treshev97, BaldomaFGS11}). In other words, in the case $\nu>2$ the validity of the
Melnikov function is the same as in the polynomial case.

On the other hand, if one considers $\al\sim \eps^\nu$ with $\nu<2$, the  condition
$\eps^{\eta-1}(\eps+\sqrt{\al})$ small enough is 
\[
 \eta-1+\frac{\nu}{2}>0.
\]
In particular, if one considers $\al$ as a parameter independent of $\eps$, one
has to require $\eta>1$.


The limit cases in the previous settings, $\eta=0$ and $\eta-1+\nu/2=0$ respectively, are usually called
singular cases (see \cite{BaldomaFGS11}). In Appendix \ref{sec:SingularCase} we make some remarks about how
these cases could be studied.
\end{remark}

\subsubsection{Narrow strip of analyticity: validity of the Melnikov function}\label{sec:NonExpSmallSplitting}
In Section \ref{sec:MelnikovNonExp} we have seen how the Melnikov function changes drastically its size when $\al=1-C\eps^r$ with $r>0$, even becoming non-exponentially small if $r\geq 2$. In this section we prove that it gives the correct first order for the splitting of separatrices. Note that, in the range of parameters for which the Melnikov function is not exponentially small, one can just apply classical perturbation techniques to prove that this function gives the first order of the splitting.

First, we give the prediction of the distance given by the Melnikov function when $\al=1-C\eps^r$ with $r>0$ and $C>0$, which can be deduced from Corollary \ref{coro:NonExpSmall}. Recall that, as we have explained in Section \ref{sec:Melnikov}, in this case we study the distance between the manifolds at the section $x=3\pi/2$. We have called $\wt d(\tau)$ to this distance (see \eqref{def:distancia:moguda}) and we call $\wt d_0(\tau)$ to the Melnikov prediction of this distance.

If $r\in (0,2)$, calling $u^\ast=\ln(1+\sqrt{2})$,the Melnikov prediction  is given by 
\begin{equation}\label{def:Area:rmenor}
\wt d_0(\tau)=\mu\eps^{\eta-r-1}\frac{2\pi}{C} e^{\dps
-\tfrac{\Im\rr_-}{\eps}}\left(\sin\left(\tau-\phi-u_*/\eps\right)+\OO\left(\eps^{r/2},\eps^{1-r/2}\right)\right),
\end{equation}
whereas if  $r=2$ is given by
\begin{equation}\label{def:Area:r2}
\wt d_0(\tau)=\mu\eps^{\eta-3}\frac{2\pi e^{-\sqrt{C}}}{C^{3/2}}\left(\sin(\tau-\phi-u^*/\eps)+\OO(\eps)\right).
\end{equation}
Finally, if $r>2$ is given by
\begin{equation}\label{def:Area:rMajor2}
\wt d_0(\tau)=\mu\eps^{\eta-3r/2}\frac{2\pi }{C^{3/2}}
\left(\sin(\tau-\phi-u^*/\eps)+\OO\left(\eps^{r/2-1}\right)\right).
\end{equation}
In these cases, direct
application of Melnikov theory only ensures that the distance between the manifolds is given
by
\[
 d(\tau)=d_0(\tau)+\OO\left(\mu^2\eps^{2(\eta-2r)}\right),
\]
since the perturbation has size $\OO(\mu\eps^{\eta-2r})$. Next theorem widens the range of the validity of this prediction under certain hypotheses.

\begin{theorem}\label{th:Main:NonExp}
Let us consider any $\mu_0>0$ and let us assume $\eta>\max\{r+1,3r/2\}$. Then, there exists $\eps_0>0$ such that for
$|\mu|<\mu_0$, $\eps\in (0,\eps_0)$ and $\al=1-C\eps^r$ with $r>0$ and $C>0$, the invariant manifolds split and their distance on the section $x=3\pi/2$ is given by,
\begin{itemize}
\item If $0<r<2$,
\[
\wt d(\tau)=\mu\eps^{\eta-r-1}\frac{2\pi}{C}e^{\dps
-\tfrac{\Im\rr_-}{\eps}}\left(\sin(\tau-\phi-u^*/\eps)+\OO\left(\mu\eps^{\eta-r-1},\eps^{1-r/2},\eps^{r/2}\right)\right).
\]
 \item If $r=2$,
\[
\wt d(\tau)=\mu\eps^{\eta-3}\frac{2\pi e^{-\sqrt{C}}}{C^{3/2}}\left(\sin(\tau-\phi-u^*/\eps)+\OO\left(\eps,\mu\eps^{\eta-3}\right)\right).
\]
 \item If  $r>2$,
\[
\wt d(\tau)=\mu\eps^{\eta-3r/2}\frac{2\pi }{C^{3/2}}\left(\sin(\tau-\phi-u^*/\eps)+\OO\left(\eps^{r/2-1},\mu\eps^{\eta-3r/2}\right)\right),
\]
\end{itemize}where $u^\ast=\ln(1+\sqrt{2})$ and $\phi=\phi(\eps,\al)$ is the constant given in Corollary \ref{coro:NonExpSmall}.
\end{theorem}

\section{The Melnikov function and its validity in the quasiperiodic case}\label{sec:QP} 
Analogously to the periodic case, we work in the extended phase space that now is $(x,y,\theta_1,\theta_2)\in \TT\times\RR\times\TT^2$. Note that we do not work in the whole symplectic space, namely we omit the actions $(I_1,I_2)$, since they do not have any dynamical interest. In this setting $\TTT=\{(0,0,\theta_1,\theta_2);(\theta_1,\theta_2)\in\TT^2\} $ is a normally hyperbolic torus and, for $\mu=0$, its stable and unstable invariant manifolds coincide along the homoclinic manifold 
\[
\begin{split}
 \WW^u(\TTT)=\WW^s(\TTT)&=\left\{(x,y,\theta_1, \theta_2);\, H_0(x,y)=0\right\}\\
&=\left\{(x,y,\theta_1, \theta_2)=(x_0(u),y_0(u),\theta_1, \theta_2);\, (u,\theta_1, \theta_2)\in\RR\times\TT^2\right\},
\end{split}
\]
where $(x_0(u),y_0(u))$ is the time parameterization of the homoclinic orbit, which is given in \eqref{def:separatrix}. 

Then, we look for the perturbed manifolds as
\[\left\{
 \begin{split}
  x&=x_0(u)\\
y&=y^{u,s}(u,\theta_1, \theta_2).
 \end{split}
\right.
\]
Here $(u,\theta_1, \theta_2)\in (-\infty, U)\times\TT^2$, for certain $U>0$, for the unstable manifold and $(u,\theta_1, \theta_2)\in (-U,+\infty)\times\TT^2$ for the stable one. Again, the Lagrangian character of the manifolds implies that the functions $y^{u,s}$ are  given by
\[
 y^{u,s}(u,\theta_1, \theta_2)=\frac{1}{y_0(u)}\pa_u T^{u,s}(u,\theta_1, \theta_2).
\]
Therefore, as in the periodic case,  a natural way to measure the difference between the manifolds is to compute 
\[
 D(u,\theta_1, \theta_2)=\pa_u T^{s}(u,\theta_1, \theta_2)-\pa_u T^{u}(u,\theta_1, \theta_2),
\]
whose first order in $\mu$ is given by the Melnikov function
\[
\MM\left(u,\theta_1, \theta_2\right)=\int_{-\infty}^{+\infty} \left\{H_0,H_1\right\}\left(x_0(u+s),y_0(u+s),\theta_1+\frac{s}{\eps},\theta_2+\ga\frac{s}{\eps}\right)ds.
\]
Namely, one has that 
\[
 D(u,\theta_1, \theta_2)=\mu\eps^\eta M(u,\theta_1, \theta_2)+\OO\left(\mu^2\eps^{2\eta}\right).
\]
As in the periodic case, for $\al$ small or fixed independently of $\eps$, we measure the distance at the section  $x=\pi$, which is given by
\begin{equation}\label{def:distancia:QP}
 d(\theta_1,\theta_2)=y^s(0,\theta_1,\theta_2)-y^u(0,\theta_1,\theta_2)=\frac{1}{2}D(0,\theta_1, \theta_2),
\end{equation}
whereas in the case $\al=1-C\eps^r$ with $r>0$ and $C>0$, we measure it at the section $x=3\pi/2$, which is given by
\begin{equation}\label{def:distancia:QP:moguda}
\wt d(\theta_1,\theta_2)=y^s\left(\ln\left(1+\sqrt{2}\right),\theta_1, \theta_2\right)-y^u\left(\ln\left(1+\sqrt{2}\right),\theta_1, \theta_2\right)=2\sqrt{2}D\left(\ln\left(1+\sqrt{2}\right),\theta_1, \theta_2\right).
\end{equation}

\subsection{The Melnikov function}
If one applies the Poincar\'e-Melnikov method to Hamiltonian  \eqref{def:ToyModel:Hamiltonian:QP}, obtains
\begin{equation}\label{def:Melnikov:QP}
\MM\left(u,\theta_1,\theta_2;\eps,\al\right)=4\int_{-\infty}^{+\infty} \beta(u+s) F\left(\theta_1+\frac{s}{\eps},\theta_2+\frac{\ga s}{\eps}\right)ds,
\end{equation}
where $\beta$ is the function defined in \eqref{def:MelnikovIntegrand} and $F$ is the quasiperiodic perturbation  \eqref{def:F:Fourier} satisfying \eqref{hyp:QP1} and \eqref{hyp:QP2}. 

In the quasiperiodic case, it is a well known fact that the size of the Melnikov function is not given by its first harmonic (see 
\cite{Simo94, DelshamsGJS97, DelshamsGJS99}). Instead, the leading harmonic depends on $\eps$. For this reason, we need to compute carefully the size of all harmonics. We compute them using Residuum Theory and the properties of the function $\beta$ given in Section \ref{sec:Melnikov}. We follow the same approach of \cite{DelshamsGJS97}. For this reason, let us first introduce certain functions and constants defined in that paper.

We define the $2\ln\ga$-periodic function $c(\de)$ defined by
\begin{equation}\label{def:c}
 c(\de)=C_0\cosh\left(\frac{\de-\de_0}{2}\right)\,\,\text{ for }\de\in [\de_0-\ln\ga,\de_0+\ln\ga],
\end{equation}
where
\[
C_0=2\sqrt{\frac{(\ga r_1+r_2)}{\ga+\ga\ii}},\,\,\,\de_0=\ln\eps^\ast,\,\,\,\eps^\ast=\frac{ \left(\ga+\ga\ii\right)}{\ga^2(r_1\ga+r_2)}
\]
and continued by $2\ln\ga$-periodicity onto the whole real axis. Note that the constant $C_0$ is slightly different from the one considered in \cite{DelshamsGJS97} since in that paper, it includes a $\pi/2$ coefficient coming from the imaginary part of the singularity of the unperturbed separatarix. Since in this paper the singularity changes with respect to $\al$, we have defined a new constant which is independent of it. Then, the formulas of the splitting, given in Proposition \ref{prop:Melnikov:QP} and Theorems \ref{th:Main:QP} and \ref{th:Main:NonExp:QP}, will contain the dependence on the imaginary part of the singularity explicitly. Following the lines in \cite{DelshamsGJS97}, one can see that the function $c(\de)$ oscillates and has lower and upper bounds independent of $\eps$.

One can use the function $c(\de)$ to give the size of the Melnikov function as done in \cite{DelshamsGJS97}. In the next proposition we see how this size changes depending on the relation between $\eps$ and $\al$. In all the results in the quasiperiodic case we include the case $\al=0$, since the proof we present is also valid in this case and slightly improves the results in the literature \cite{DelshamsGJS97, Sauzin01}.

\begin{proposition}\label{prop:Melnikov:QP}
Let us assume \eqref{hyp:QP1} and \eqref{hyp:QP2} and fix $u_0>0$. Then, there exists $\eps_0>0$ such that for $\eps\in (0,\eps_0)$ and $\alpha\in [0,1)$,
\begin{itemize}
 \item If $\al$ satisfies $0\leq\alpha\leq C\eps^\nu$ where $\nu>1$ and $C>0$ are
constants independent of $\eps$, the Melnikov function \eqref{def:Melnikov:QP}
satisfies that
\begin{equation}\label{def:Melnikov:BandaAmpla:QP}
\frac{C_1}{\eps}e^{ -c(\ln(2\eps/\pi)){\dps\sqrt{\tfrac{\pi}{2\eps}}}}\leq \sup_{(u,\theta_1,\theta_2)\in (-u_0,u_0)\times\TT^2}\left|\MM\left(u,\theta_1,\theta_2;\eps,\al\right)\right|\leq \frac{C_2}{\eps}e^{ -c(\ln(2\eps/\pi)){\dps\sqrt{\tfrac{\pi}{2\eps}}}} \end{equation}
for certain constants $u_0,C_1,C_2>0$.
\item In the intermediate case $\alpha=\al_*\eps+\OO(\eps^2)$, one can only give upper bounds of type 
\[
\sup_{(u,\theta_1,\theta_2)\in (-u_0,u_0)\times\TT^2}\left|\MM\left(u,\theta_1,\theta_2;\eps,\al\right)\right|\leq \frac{C_2}{\eps}e^{ -c(\ln(2\eps/\pi)){\dps\sqrt{\tfrac{\pi}{2\eps}}}}
\]
for certain constant $u_0,C_2>0$.
\item If  $\alpha$  satisfies  $C\eps^\nu<\alpha<1-C\eps^r$
where $\nu\in (0,1]$, $r\in [0,2)$ and $C>0$ are constants independent of $\eps$, the Melnikov function satisfies
\begin{multline}\label{def:Melnikov:BandaEstreta:QP}
\frac{C_1}{\sqrt{\eps\al}(1-\al)^{5/4}}e^{ -c(\ln(\eps/\Im\rr_-)){\dps\sqrt{\tfrac{\Im\rr_-}{\eps}}}}
\leq \sup_{(u,\theta_1,\theta_2)\in (-u_0,u_0)\times\TT^2}\left|\MM\left(u,\theta_1,\theta_2;\eps,\al\right)\right|\\
\leq \frac{C_2}{\sqrt{\eps\al}(1-\al)^{5/4}}e^{ -c(\ln(\eps/\Im\rr_-)){\dps\sqrt{\tfrac{\Im\rr_-}{\eps}}}} 
\end{multline}
for certain constants $C_1,C_2>0$.
\end{itemize}
\end{proposition}

\begin{proof}
As in the periodic case, to prove the first statement, instead of computing the residuums of both
$\rr_\pm$ and look for cancellations, we just expand the Melnikov integral,
which is uniformly convergent as a real integral, in power series of $\al$ as  
\begin{equation}\label{eq:Melnikov:Expansio:QP}
 \MM\left(u,\theta_1,\theta_2;\eps,\al\right)=\sum_{n=0}^\infty \al^n \MM_n\left(u,\theta_1,\theta_2;\eps\right),
\end{equation}
where
\[
\MM_n\left(u,\theta_1,\theta_2;\eps\right)=4 (n+1)2^n\int_{-\infty}^{+\infty}
\frac{\sinh^{n+1}(u+s)}{\cosh^{2n+3}(u+s)}F\left(\theta_1+\frac{s}{\eps},\theta_2+\frac{\ga s}{\eps}
\right)ds.
\]
The function $\MM_0$ was computed in \cite{DelshamsGJS97}, and in that paper it was shown that it can be bounded as formula \eqref{def:Melnikov:BandaAmpla:QP}. The rest of the functions $\MM_n$ can be bounded as follows.  We Fourier-expand $\MM_n$ in $(\theta_1,\theta_2)$ and we change the path of integration to $\Im s=\pi/2-\sqrt{\eps}$ or $\Im s=-(\pi/2-\sqrt{\eps})$ depending on the sign of $k\cdot\omega=k_1+\gamma k_2$. Then, we can bound  each harmonic as
\[
\left| \MM_n^{[k]}(u;\eps)\right|\leq \left(\frac{K}{\eps}\right)^{n+1} \left|F^{[k]}\right| e^{-\dps\tfrac{|k\cdot\omega|}{\eps}\left(\tfrac{\pi}{2}-\sqrt{\eps}\right)}.
\]
Therefore, proceeding as in \cite{DelshamsGJS97}, one can see that $\MM_n\left(u,\theta_1,\theta_2;\eps\right)$ can be bounded as 
\[
\left| \MM_n\left(u,\theta_1,\theta_2;\eps\right)\right|\leq \left(\frac{K}{\eps}\right)^{n+1} e^{ -c(\ln(2\eps/\pi)){\dps\sqrt{\tfrac{\pi}{2\eps}}}},
\]
where the function $c$ is given in \eqref{def:c}.

Therefore, if $\eps\ll\al$, the series \eqref{eq:Melnikov:Expansio:QP} is decreasing, and therefore, the leading term is given by the first order.

For the second and third statements, we just apply the Residuum Theory
to the Fourier harmonics of the Melnikov function \eqref{def:Melnikov:QP}. In can be easily seen that they  have size
\begin{equation}\label{eq:Bound:Melnikov:Fourier}
  \left| \MM^{[k]}(u)\right|= \left| F^{[k]}\right| \left|\frac{k\cdot\omega}{\eps}\de_2(\al)+ \de_1(\al)\right| e^{-\dps\tfrac{|k\cdot\omega|\Im\rr_-}{\eps}},
\end{equation}
where $\de_1(\al)$ and $\de_2(\al)$ are the functions defined in \eqref{def:delta1} and \eqref{def:delta2} respectively.
For the the range of $\al$ we are considering, the leading term is $\de_2$, which satisfies
\[
 \de_2(\al)\sim \frac{1}{\sqrt{\al}(1-\al)}.
\]
Finally one has to proceed as in \cite{DelshamsGJS97}. Even if the method in that paper does not apply directly since now $\rr_-$ depends on $\eps$, it is enough to consider as a new parameter $q=\eps/\Im\rr_-$, which is still small since we are assuming  $C\eps^\nu<\alpha<1-C\eps^r$
where $\nu\in (0,1]$, $r\in [0,2)$ and $C>0$ and therefore $\Im\rr_-\gg\eps$. With this new parameter $q$, it is straightforward to obtain  the size of the Melnikov function with the techniques in \cite{DelshamsGJS97}.
\end{proof}

\begin{remark}
Note that, as in the periodic case, the size of the Melnikov function depends strongly on $\al$. When the
strip of analyticity \eqref{def:Strip} is wide enough, that is $\al\sim\eps^\nu$ with
$\nu>1$, the Melnikov function behaves as in the entire case
$\alpha=0$ \cite{DelshamsGJS97}.  In other words, the Melnikov function does not notice the finiteness
of the strip of analyticity and the exponentially small coefficient is given by
the imaginary part of the singularities of the separatrix. Nevertheless, note that the condition is different. In the periodic case was needed $\al\ll\eps^2$ instead of $\al\ll\eps$.

When the strip of analyticity is independent of $\eps$
or not extremely big with respect to $\eps$, namely taking $\alpha\sim\eps^\nu$
with $\nu\in [0,1)$, the exponentially small coefficient multiplying the periodic function $c(\de)$ does not coincide with
the imaginary part of the singularity of the unperturbed separatrix. Instead, it
is given by the imaginary part of this new singularity $\rr_-$, which appears
when one evaluates the perturbation along the unperturbed separatrix. Note
that even if one takes, for instance, $\al=\sqrt{\eps}$, which gives a strip of analyticity of order $\frac{1}{2}\ln\frac{1}{\eps}$, one has that, using \eqref{eq:Expansio:Sing},
\[
\MM\left(u,\theta_1,\theta_2;\eps,\sqrt{\eps}\right)\sim\eps^{-\frac{3}{4}} e^{ -c(\ln(\eps/\Im\rr_-)){\dps\sqrt{\tfrac{\pi}{2\eps}}}\left(1-\frac{\eps^{1/4}}{\pi}\right)}.
\]
That is, even for perturbations with a wide strip of analyticity with respect to
$\eps$, there appears a correcting term in the exponential.

Notice that the case $\al\sim\eps$ is the boundary between these two
different behaviors. 
\end{remark}

\subsubsection{Narrow strip of analyticity: a drastic change on the size of the Melnikov function}\label{sec:MelnikovNonExp:QP}
In this section, we study how the Melnikov increases when the analyticity strip shrinks, namely when  $\al=1-C\eps^r$ with $r\geq 0$ and $C>0$. Next corollary  gives upper and lower bounds for it in this case.
\begin{corollary}\label{coro:NonExpSmall:QP}
Let us assume \eqref{hyp:QP1} and \eqref{hyp:QP2} and fix $u_0>0$. Then, if one takes $\al=1-C\eps^r$ with $C>0$ and $r>0$,  the Melnikov function in \eqref{def:Melnikov:BandaEstreta:QP} has the following upper and lower bounds
\begin{itemize}
\item If $0<r<2$, 
\begin{multline}\label{def:Melnikov:BandaEstreta:QP2}
\frac{C_1}{\sqrt{\al}\eps^{1/2+5r/4}}e^{ -c(\ln(\eps/\Im\rr_-)){\dps\sqrt{\tfrac{\Im\rr_-}{\eps}}}}
\leq \sup_{(u,\theta_1,\theta_2)\in (-u_0,u_0)\times\TT^2}\left|\MM\left(u,\theta_1,\theta_2;\eps,\al\right)\right|\\
\leq \frac{C_2}{\sqrt{\al}\eps^{1/2+5r/4}}e^{ -c(\ln(\eps/\Im\rr_-)){\dps\sqrt{\tfrac{\Im\rr_-}{\eps}}}} 
\end{multline}
\item If $r\geq 2$,
\[
\frac{C_1}{\eps^{3r/2}} \leq \sup_{(u,\theta_1,\theta_2)\in (-u_0,u_0)\times\TT^2}\left|\MM\left(u,\theta_1,\theta_2;\eps, 1-C\eps^r\right)\right|\leq \frac{C_2}{\eps^{3r/2}}
\]
\end{itemize}
for certain constants $C_1,C_2>0$.
\end{corollary}
The first statement of this corollary is just a rewriting of the third statement of Proposition \ref{prop:Melnikov:QP}. The second one is a direct consequence of formula \eqref{eq:Bound:Melnikov:Fourier}, if one takes into account the definition of $\de_1$ and $\de_2$ in \eqref{def:delta1} and \eqref{def:delta2} and the asymptotics for $\rr_-$ in \eqref{eq:Expansio:Sing:alfa1}.

\begin{remark}
When the strip of analyticity shrinks, which ocurrs when $\al$ approaches 1, the  Melnikov function increases exponentially. For instance, if one takes $\al=1-\eps$,  using \eqref{eq:Expansio:Sing:alfa1},
\[
\MM\left(u,\theta_1,\theta_2;\eps,1-\eps\right)\sim\eps^{-\frac{7}{4}} e^{\dps-\tfrac{c(\ln(\eps/\Im\rr_-))}{\eps^{1/4}}}.
\]
The exponentially small Melnikov function keeps increasing until the limiting case $\al\sim 1- \eps^2$ when it becomes non-exponentially small.
\end{remark}

\subsection{Validity of the Melnikov function}
Once we have computed the Melnikov function in Proposition \ref{prop:Melnikov:QP},  we compute the prediction it gives for the distance between the manifolds at the section $x=\pi$ for $0<\al\leq\al_0<1$ for any fixed $\al_0$

If $\al$ satisfies $0\leq\al\leq C\eps^\nu$ where $\nu>1$ and $C>0$ are constants
independent of $\eps$, we can ensure that $d_0$ satisfies
\begin{equation}\label{def:DistanciaMelnikov:BandaAmpla:QP}
C_1|\mu|\eps^{\eta-1}e^{ -c(\ln(2\eps/\pi)){\dps\sqrt{\tfrac{\pi}{2\eps}}}}\leq\max_{(\theta_1,\theta_2)\in \TT^2}\left|d_0(\theta_1,\theta_2)\right|\leq C_2|\mu|\eps^{\eta-1}e^{ -c(\ln(2\eps/\pi)){\dps\sqrt{\tfrac{\pi}{2\eps}}}}.
\end{equation}
If $\al=\al_*\eps+\OO(\eps^2)$ for any constant $\al_*>0$, $d_0$ satisfies
\begin{equation}\label{def:DistanciaMelnikov:BandaIntermedia:QP}
\max_{(\theta_1,\theta_2)\in \TT^2}\left|d_0(\theta_1,\theta_2)\right|\leq C_2|\mu|\eps^{\eta-1}e^{ -c(\ln(2\eps/\pi)){\dps\sqrt{\tfrac{\pi}{2\eps}}}}.
\end{equation}
Finally, if   $C\eps^\nu<\al\leq\al_0<1$ where
$\nu\in (0,2)$ and $C>0$ are constants independent of $\eps$, $d_0$ satisfies
\begin{multline}\label{def:DistanciaMelnikov:BandaEstreta:QP}
\frac{C_1|\mu|\eps^{\eta-\frac{1}{2}}}{\sqrt{\al}}e^{ -c(\ln(\eps/\Im\rr_-)){\dps\sqrt{\tfrac{\Im\rr_-}{\eps}}}} \leq\max_{(\theta_1,\theta_2)\in \TT^2}\left|d_0(\theta_1,\theta_2)\right|\\\leq \frac{C_2|\mu|\eps^{\eta-\frac{1}{2}}}{\sqrt{\al}}e^{ -c(\ln(\eps/\Im\rr_-)){\dps\sqrt{\tfrac{\Im\rr_-}{\eps}}}}.
\end{multline}
In all three formulas $C_1, C_2>0$ are constants independent of $\al$ and $\eps$.

As we have already explained, direct
application of Melnikov theory only ensures that the distance between the invariant manifolds on the section $x=\pi$ is given
by
\[
 d(\theta_1,\theta_2)=d_0(\theta_1,\theta_2)+\OO\left(\mu^2\eps^{2\eta}\right).
\]
Therefore, the Melnikov function \eqref{def:Melnikov:QP}, is, in principle, the first order
provided  $\mu$  is exponentially small with respect to $\eps$.  Next theorem shows for which range of
parameters $\al$, $\eps$ and $\eta$, the Melnikov function gives the true first
order. In the quasiperiodic case, it is not known whether this range is the optimal one for which the Melnikov function predicts the splitting correctly.

\begin{theorem}\label{th:Main:QP}
Let us consider any $\mu_0>0$ and $\al_0\in (0,1)$. Then, there exists $\eps_0>0$ such that for
$|\mu|<\mu_0$, $\eps\in (0,\eps_0)$ and $\al\in [0,\al_0]$ such that 
$\eps^{\eta-1}(\sqrt{\eps}+\sqrt{\al})$ is small enough,
\begin{itemize}
 \item If $\al$ satisfies $0\leq\alpha\leq C\eps^\nu$ where $\nu\geq 1$ and $C>0$ are
constants independent of $\eps$, the distance between the invariant manifolds on the section $x=\pi$ is given by
\[
d(\theta_1,\theta_2)=d_0(\theta_1,\theta_2)+\OO\left(|\mu|^2\eps^{2\eta-2}e^{ -c(\ln(\eps/\Im\rr_-)){\dps\sqrt{\tfrac{\Im\rr_-}{\eps}}}}\right).
\]
\item If  $\alpha$ satisfies  $C\eps^\nu\leq\alpha \leq \al_0<1$
where $\nu\in [0,1)$ and $C>0$ are constants independent of $\eps$, the distance between the invariant manifolds on the section $x=\pi$ is given by
\[
d(\theta_1,\theta_2)=d_0(\theta_1,\theta_2)+\OO\left(|\mu|^2\eps^{2\eta-2}e^{ -c(\ln(\eps/\Im\rr_-)){\dps\sqrt{\tfrac{\Im\rr_-}{\eps}}}}\right).\]
\end{itemize}
\end{theorem}
The proof of this theorem is deferred to Section \ref{sec:SketchProof:QP}. 
\begin{remark}
Note that even if this theorem holds true provided  $\eps^{\eta-1}(\sqrt{\eps}+\sqrt{\al})$ is small enough,
the Melnikov prediction is the true first order only if one imposes a more restrictive condition.

Comparing the size of the Melnikov prediction \eqref{def:DistanciaMelnikov:BandaAmpla:QP} and the size of the remainder, in the case $\al \ll\eps$, one has to impose the condition $\eta>1$. Let us observe that when $\al=0$, this theorem, which gives that Melnikov provides the true first order of the distance if $\eta>1$, is an improvement with respect to the  sharpest previous result \cite{DelshamsGJS97}, which needed the condition $\eta>3$ (see also \cite{Sauzin01}).

In the case  $\al \gg\eps$ one has to impose $\eps^{\eta-\frac{3}{2}}\sqrt{\al}\ll 1$. 

Finally, note that the intermediate case $\al=\al_*\eps+\OO(\eps^2)$ is included in the first statement of the theorem. Nevertheless, if one would check that $d_0$ is nondegenerate, Melinkov theory would give also in this case the correct prediction.
\end{remark}

\subsubsection{Narrow strip of analyticity: validity of the Melnikov function}\label{sec:NonExpSmallSplitting:QP}
In Section \ref{sec:MelnikovNonExp:QP} we have studied the size of the Melnikov function when $\al=1-C\eps^r$ with $C,r>0$. We devote this section to show for which range of parameters the Melnikov function gives the correct first order of the splitting.


First, we give the prediction of the distance given by the Melnikov function if $\al=1-C\eps^r$ with $r>0$ at the section $x=3\pi/2$. If $0<r<2$, 
\begin{multline}\label{def:DistanciaMelnikov:BandaMoltEstreta:QP}
C_1|\mu|\eps^{\eta-\frac{1}{2}-\frac{5r}{4}}e^{ -c(\ln(\eps/\Im\rr_-)){\dps\sqrt{\tfrac{\Im\rr_-}{\eps}}}} \leq\max_{(\theta_1,\theta_2)\in \TT^2}\left|\wt d_0(\theta_1,\theta_2)\right|\\\leq C_2|\mu|\eps^{\eta-\frac{1}{2}-\frac{5r}{4}}e^{ -c(\ln(\eps/\Im\rr_-)){\dps\sqrt{\tfrac{\Im\rr_-}{\eps}}}}
\end{multline}
whereas if $r>2$
\begin{equation}\label{def:distance:NonExp:QP}
C_1|\mu|\eps^{\eta-3r/2}\leq \max_{(\theta_1,\theta_2)\in\TT^2}\left|\wt d_0(\theta_1,\theta_2)\right|\leq C_2|\mu|\eps^{\eta-3r/2}.
\end{equation}

Next theorem proves the validity of this prediction under certain hypotheses.
\begin{theorem}\label{th:Main:NonExp:QP}
Let us consider any $\mu_0>0$ and $\al=1-C\eps^r$ with $r>0$, $C>0$, and let us assume $\eta>\max\{r+1,3r/2\}$. Then, there exists $\eps_0>0$ such that for
$|\mu|<\mu_0$ and $\eps\in (0,\eps_0)$, the invariant manifolds split and the distance between them on the section $x=3\pi/2$ is given by,
\begin{itemize}
\item If $0<r<2$,
\[
\wt d(\theta_1,\theta_2)=\wt d_0(\theta_1,\theta_2)+\OO\left(|\mu|^2\eps^{2\eta-2r-2}e^{ -c(\ln(\eps/\Im\rr_-)){\dps\sqrt{\tfrac{\Im\rr_-}{\eps}}}}\right).\]
\item If $r>2$,
\[
\wt d(\theta_1,\theta_2)=\wt d_0(\theta_1,\theta_2)+\OO\left(\mu\eps^{2\eta-3r}\right).
\]
\end{itemize}
\end{theorem}

\section{The periodic case: Proof of Theorem
\ref{th:Main} and \ref{th:Main:NonExp}}\label{sec:SketchProof}
In Section \ref{sec:SketchProof:1} we prove Theorem \ref{th:Main}, that is when the parameter $\al$ is bounded away from 1. Nevertheless, we prove the result for a wider range of the parameter $\al$. Therefore, our proof deals also with the first statement of Theorem \ref{th:Main:NonExp}. For this reason, during Section  \ref{sec:SketchProof:1}  we assume the condition
\begin{equation}\label{cond:ExpSmall}
0\leq \al\ll1-\eps^2.
\end{equation}
Note, that this range of parameters correspond to an exponentially small Melnikov function. The second statement of Theorem \ref{th:Main:NonExp} is proved in Section \ref{sec:SketchProof:NonExp}

\subsection{Proof of Theorem \ref{th:Main}}\label{sec:SketchProof:1}
To study the splitting of separatrices we follow the approach proposed in
\cite{LochakMS03, Sauzin01, GuardiaOS10}, which was inspired by Poincar\'e. Namely we use the
fact that the invariant manifolds are Lagrangian graphs. This allows us to look for
parameterizations of the invariant manifolds as graphs of the gradient of
certain generating functions. 

First, let us point out that we will consider $\tau$ as a complex variable to take advantage of the analyticity of the Hamiltonian with respect to this variable. To this end, we consider a fixed $\sigma>0$ and we take
\[
 \tau\in \TT_\sigma=\{\tau\in \CC/\ZZ: |\Im \tau|<\sigma\}.
\]
Note that since the Hamiltonian is entire in $\tau$, we can take any $\sigma$. From now on, all the constants appearing in this section will depend on $\sigma$. As we will see during the proof, in the periodic case, the analyticity strip with respect to $\tau$ does not play any role in the size of the splitting.

Let us consider consider the symplectic change of variables 
\begin{equation}\label{eq:CanviSimplecticSeparatriu}
\left\{\begin{array}{l}
x=x_0(u)=4\arctan\left(e^u\right)\\
\dps y=\frac{w}{y_0(u)}=\frac{\cosh u}{2}w,
\end{array}\right.
\end{equation}
which was introduced in \cite{Baldoma06} (see also \cite{LochakMS03, Sauzin01})
and reparameterize time as $\tau=t/\eps$. With these new variables, the
Hamiltonian function \eqref{def:ToyModel:Hamiltonian} reads
\begin{equation}\label{def:ToyModel:HamiltonianReparameterized}
\eps \ol H\left(u,w,\tau\right)=\eps H\left(x_0(u),\frac{w}{y_0(u)},\tau\right)
\end{equation}
and the unperturbed separatrix can be parameterized as a graph as 
\[
w=y_0^2(u)=\frac{4}{\cosh^2u}.
\]
Then, one can look for the perturbed invariant manifolds as graphs of the
gradient of generating functions $T^{u,s}(u,\tau)$. Namely, we look for
functions $T^{u,s}(u,\tau)$ such that the invariant manifolds are given by
$w=\pa_u T^{u,s}(u,\tau)$. Moreover, these functions are solutions of the
Hamilton-Jacobi equation
\begin{equation}\label{def:HJ1}
 \eps\ii\pa_\tau T+\ol H\left(u, \pa_u T,\tau\right)=0.
\end{equation}
This equation reads
\begin{equation}\label{def:HJ2}
 \eps\ii\pa_\tau T+\frac{\cosh^2
u}{8}\left(\pa_uT\right)^2-\frac{4}{\cosh^2u}+\mu\eps^\eta \Psi(u)\sin\tau=0
\end{equation}
with
\begin{equation}\label{def:Psi}
 \Psi(u)=\psi(x_0(u)),
\end{equation}
where $\psi$ is the function considered in \eqref{def:ToyModel:Hamiltonian} and
$x_0(u)$ is the first component of the time-parameterization of the separatrix
(see \eqref{eq:CanviSimplecticSeparatriu}).

Moreover, we impose the asymptotic conditions
\begin{align}
\dps \lim_{\Re u\rightarrow-\infty}y_0\ii(u)\cdot\pa_u T^u(u,\tau)=0 &
\text{\quad (for the unstable manifold)}
\label{eq:AsymptCondFuncioGeneradora:uns}\\
\dps \lim_{\Re u\rightarrow+\infty}y_0\ii(u)\cdot \pa_u
T^s(u,\tau)=0 & \text{\quad (for the stable manifold)}.
\label{eq:AsymptCondFuncioGeneradora:st}
\end{align}
One can easily see that for $\mu=0$, the solution of equation \eqref{def:HJ2} is
just
\begin{equation}\label{def:T0}
T_0(u) =4\frac{e^u}{\cosh u},
\end{equation}
which is the generating function of the unperturbed separatrix and has singularities  at $u=\pm i\pi/2$. Nevertheless, as we have
explained in Section \ref{sec:Melnikov}, the term $\Psi(u)$ has singularities at
$u=\rr_\pm, \ol \rr_\pm$. In particular, $u=\rr_-, \ol\rr_-$ are closer to the
real axis than $u=\pm i\pi/2$, $u=\rr_+$ and $u=\ol\rr_+$. Therefore, to study the exponentially small splitting of
separatrices for this kind of systems, one has to look for parameterizations
of the invariant manifolds, namely solutions of equation \eqref{def:HJ2}, in
complex domains up to a distance of order $\OO(\eps)$ of the singularities
$u=\rr_-, \ol\rr_-$. To this end, we define the domains
\begin{equation}\label{def:DominisOuter}
\begin{split}
\dps D^{u}_{\kk}&=\left\{u\in\CC; |\Im u|<-\tan \beta_1\left(\Re u-\Re
\rr_-\right)+\Im \rr_--\kk\eps\right\}\\
\dps D^{s}_{\kk}&=\left\{u\in\CC; |\Im u|<\tan \beta_1\left(\Re u-\Re
\rr_-\right)+\Im \rr_--\kk\eps\right\},
\end{split}
\end{equation}
where $\beta_1>0$ is an angle independent of $\eps$ and $\kappa >0$
is such that $\Im\rr_--\kappa \eps >0$ (see Figure \ref{fig:Outer}). 

\begin{figure}[h]
\begin{center}
\psfrag{a}{$i\frac{\pi}{2}$}\psfrag{b}{$-i\frac{\pi}{2}$}\psfrag{Ds}{$D^s_\kk$}\psfrag{Du}{$D^u_\kk$}\psfrag{r1}{$\rr_-$}\psfrag{r2}{$\ol\rr_-$}
\includegraphics[height=6cm]{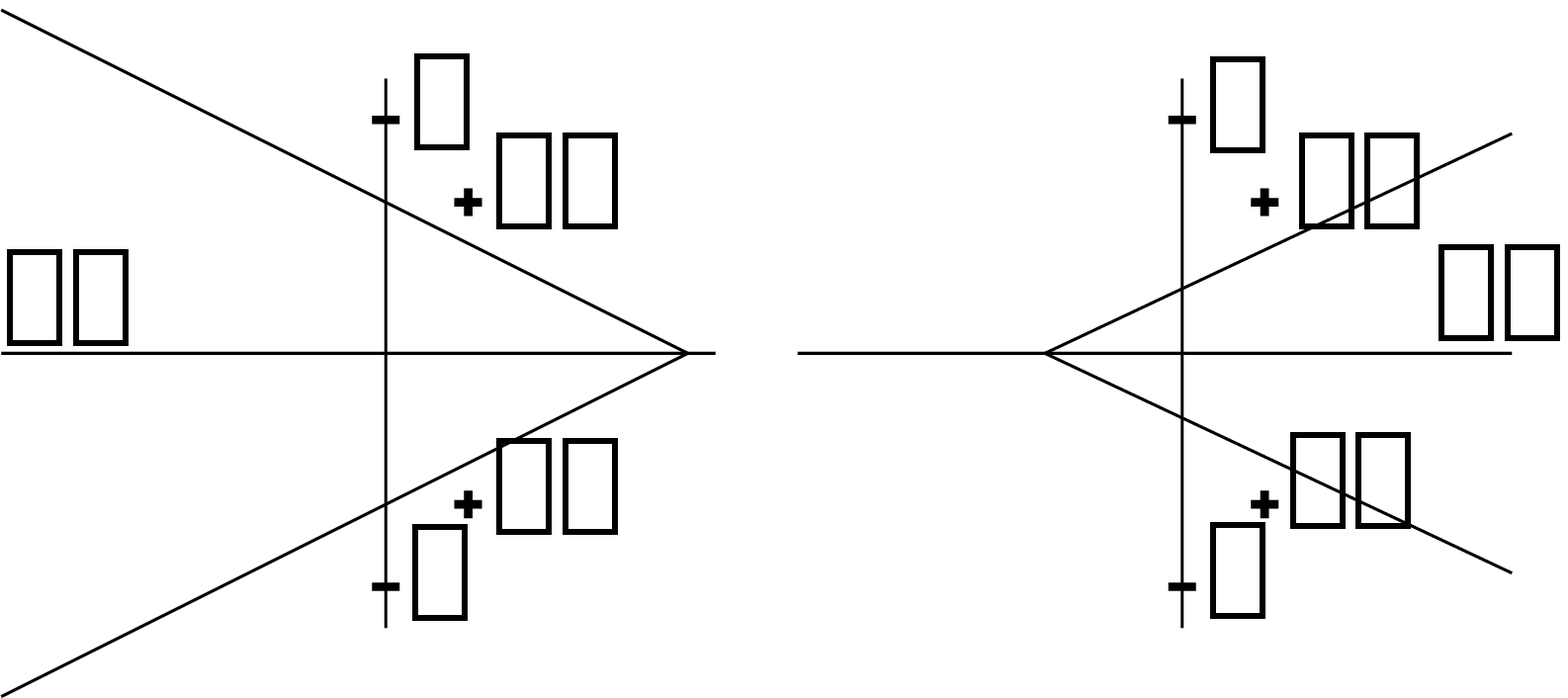}
\end{center}
\caption{\figlabel{fig:Outer}The domains $D^u_\kk$ and $D^s_\kk$ defined in \eqref{def:DominisOuter}.}
\end{figure}

 Note that this is a
significant difference in comparison with the papers dealing with the
exponentially small splitting of separatrices problem for algebraic or
trigonometric polynomial perturbations (see \cite{DelshamsS97, Gelfreich97,
BaldomaF04, BaldomaFGS11}). In these cases, one needs to look for the
parameterizations up to a distance of order $\OO(\eps)$ of the singularities of the unperturbed
separatrix. Nevertheless, in this paper, since the perturbation changes the
location of the singularities of the perturbed invariant manifolds so
drastically, one has to study the invariant manifolds close to these new
singularities.

Let us recall that when $\al$ is close to 1, namely, when the analyticity strip of the Hamiltonian is very narrow, the singularity $\rr_-$ is given by
\[
 \rr_-=\ln (1+\sqrt{2})+i(1-\al)^{1/2}+\OO(1-\al).
\] 
Therefore, $\Im\rr_-\ll \eps$ if $1-\al\ll\eps^2$. In this case, as the Melnikov function is not exponentially small (see Corollary \ref{coro:NonExpSmall}), we can use a classical perturbative approach to prove its validity. We leave this easier yet surprising case to Section \ref{sec:SketchProof:NonExp}. Thus, from now on, we assume the condition \eqref{cond:ExpSmall} and we proceed to prove Theorem \ref{th:Main}. This condition ensures that the intersection $D^u_\kk\cap D^s\kk\cap \RR$ contains a fundamental domain, since it is of size $\OO(\sqrt{1-\al})$. Note that if $\al$ is bounded away from 1, one can choose the angle $\beta_1$ so that the domain  $D^u_\kk\cap D^s\kk\cap \RR$  contains the point $u=0$ (which corresponds to the section $x=\pi$).  If $\al=1-C\eps^r$ with $r\in (0,2)$, the domain  $D^u_\kk\cap D^s\kk\cap \RR$  contains the point $u=\ln(1+\sqrt{2})$ (which corresponds to the section $x=3\pi/2$). One could change slightly the domain so that it would contain $u=0$ also in this latter case. Nevertheless, in this case one would have to take $\beta_1\sim \eps^{r/2}$. This would lead to worse estimates which would require a stronger condition on $\eta$ for the Melnikov function to be the true first order of the splitting.

The next theorem gives the existence of the invariant manifolds in
the domains $D^{\ast}_{\kk}$ with $\ast=u,s$ defined in
\eqref{def:DominisOuter}. We state the results for the unstable
invariant manifold. The stable one has analogous properties.

\begin{theorem}\label{th:ExistenceManifolds}
Let us fix $\kk_1>0$. Then, there exists $\eps_0>0$ such
that for $\eps\in(0,\eps_0)$, $\alpha\in (0,1)$, $\mu\in B(\mu_0)$, if
$\frac{\eps+\sqrt{\alpha}}{1-\al}\eps^{\eta-1}$ is small enough and \eqref{cond:ExpSmall} is satisfied, the
Hamilton-Jacobi equation \eqref{def:HJ1} has  a unique
(modulo an additive constant) real-analytic solution in
$D^{u}_{\kk_1}\times\TT_\sigma$ satisfying the asymptotic
condition \eqref{eq:AsymptCondFuncioGeneradora:uns}.

Moreover, there exists a real constant $b_1>0$ independent of $\eps$, $\mu$ and $\al$, such that for $(u,\tau)\in
D^{u}_{\kk_1}\times\TT_\sigma$,
\[
\begin{split}
\left|\pa_u T^u(u,\tau)-\pa_u T_0(u)\right|&\leq
\frac{b_1|\mu|\eps^{\eta-1}}{(\eps+\sqrt{\al})(1-\al)}\\
\left|\pa^2_u T^u(u,\tau)-\pa^2_u T_0(u)\right|&\leq
\frac{b_1|\mu|\eps^{\eta-2}}{(\eps+\sqrt{\al})(1-\al)}.
\end{split}
\]
Furthermore, if we define the half Melnikov function
\begin{equation}\label{def:HalfMelnikov:uns}
\MM^u(u,\tau)=-4\int_{-\infty}^{0} 
\frac{\sinh(u+s)\cosh(u+s)}{\left(\cosh^2(u+s)-2\alpha\sinh(u+s)\right)^2}
\sin\left(\tau+\frac{s}{\eps}\right)ds,
\end{equation}
the generating function $T^u$  satisfies that, for $(u,\tau)\in
D^{u}_{\kk_1}\times\TT_\sigma$,
\begin{equation}\label{def:HalfMelnikov:uns:cota}
\left|\pa_u T^u(u,\tau)-\pa_u T_0(u)-\mu\eps^\eta\MM^u(u,\tau)\right|\leq
b_1|\mu|^2\frac{\eps^{2\eta-2}}{(1-\al)^2}.
\end{equation}
\end{theorem}
The proof of this theorem is deferred to Section \ref{sec:ExistenceManifolds}.
The parameterization of the stable manifold has analogous properties. In
particular, we can define 
\begin{equation}\label{def:HalfMelnikov:st}
\MM^s(u,\tau)=4\int_{0}^{+\infty} 
\frac{\sinh(u+s)\cosh(u+s)}{\left(\cosh^2(u+s)-2\alpha\sinh(u+s)\right)^2}
\sin\left(\tau+\frac{s}{\eps}\right)ds,
\end{equation}
and then, for $(u,\tau)\in
D^{s}_{\kk_1}\times\TT_\sigma$,
\begin{equation}\label{def:HalfMelnikov:st:cota}
\left|\pa_u T^s(u,\tau)-\pa_u T_0(u)-\mu\eps^\eta\MM^s(u,\tau)\right|\leq
|\mu|^2\frac{\eps^{2\eta-2}}{(1-\al)^2}.
\end{equation}
Note that the Melnikov function defined in \eqref{def:Melnikov} is simply
\[
\MM(u,\tau;\eps,\al) =\MM^s(u,\tau)-\MM^u(u,\tau).
\]
From now on, we omit the dependence on $\eps$ and $\al$ of the Melnikov function $\MM$, which we denote by $\MM(u,\tau)$.

Once we know the existence of parameterizations of the invariant manifolds, the
next step is to study their difference. To this end, we define
\begin{equation}\label{def:DiffManifolds}
 \Delta (u,\tau)=T^s(u,\tau)-T^u(u,\tau).
\end{equation}
This function is defined in $ R_{\kk_1}\times\TT_\sigma$, where $ R_{\kk}$, for any $\kk>0$, is the romboidal domain
\begin{equation}\label{def:DominiInterseccio}
\dps R_{\kk}=D^{u}_{\kk}\cap D^{s}_{\kk}.
\end{equation}

Subtracting equation \eqref{def:HJ2} for both $T^s$ and $T^u$, one can easily
see that $\Delta\in\mathrm{Ker}\wt\LL_\eps$ for 
\begin{equation}\label{def:OperadorAnulador}
 \wt\LL_\eps=\eps\ii\pa_\tau+\left(\frac{\cosh^2u}{8}\left(\pa_uT^s(u,
\tau)+\pa_u T^u(u,\tau)\right)\right)\pa_u.
\end{equation}
Since Theorem \ref{th:ExistenceManifolds} ensures that the perturbed invariant
manifolds are well approximated by the unperturbed separatrix in the domains
$D^u_{\kk_1}$ and $D^s_{\kk_1}$, we know that the operator $\wt \LL_\eps$ is
close to the constant coefficients operator
\begin{equation}\label{def:Lde}
 \LL_\eps=\eps\ii\pa_\tau+\pa_u.
\end{equation}
As it is well known, any function which is defined for $(u,\tau)\in\{u\in
\CC:\Re u=a,\Im u\in [-r_0,r_0]\}\times\TT_\sigma$ and belongs to the kernel of
$\LL_\eps$, it is defined in  the strip $\{|\Im u|<r_0\}\times\TT_\sigma$ and
has exponentially small bounds for real values of the variables. This fact is
summarized in the next lemma, whose proof follows the same lines as the one of
the slightly different Lemma 3.10 of \cite{GuardiaOS10}.

\begin{lemma}\label{lemma:Lazutkin}
Let us consider a function $\zeta(u,\tau)$ analytic in
$(u,\tau)\in\{u\in \CC:\Re u=a,\Im u\in (-r_0,r_0)\}\times\TT_\sigma$ which is
solution of
$\LL_\eps\zeta=0$. Then, $\zeta$ can be extended analytically to
$\{|\Im u|<r_0\}\times\TT_{\sigma}$ and its mean value
\[
\langle\zeta\rangle=\frac{1}{2\pi}\int_0^{2\pi}\zeta(u,\tau)d\tau
\]
does not depend on $u$. Moreover, for $r\in(0,r_0)$ and
$\sigma'\in(0,\sigma)$, we define
\begin{equation}\label{def:M_r}
M_r=\max_{(u,\tau)\in \left[
-ir,ir\right]\times\overline{\TT}_{\sigma'}}\left|\pa_u\zeta(u,\tau)\right|.
\end{equation}
Then, provided $\eps$ is small enough, for $(u,\tau)\in \RR\times\TT_{\sigma'}$ the following inequality holds
\[
\left|\pa_u\zeta(u,\tau)\right|\leq 4  M_re^{-\frac{r}{\eps}}.
\]
\end{lemma}
To apply this lemma to study the difference between the invariant manifolds, we 
follow \cite{Sauzin01} (see also \cite{GuardiaOS10}). Namely, we look for a change of
variables which conjugates $\wt\LL_\eps$ in \eqref{def:OperadorAnulador} with $\LL_\eps$ in
\eqref{def:Lde}.

\begin{theorem}\label{th:Canvi}
Let us consider the constant $\kk_1$ defined in Theorem
\ref{th:ExistenceManifolds} and let us fix any $\kk_3>\kk_2>\kk_1$. Then, there
exists $\eps_0>0$ such that for $\eps\in (0,\eps_0)$ and $\alpha \in (0,1)$
satisfying \eqref{cond:ExpSmall} and that $\eps^{\eta-1}\frac{\eps+\sqrt{\al}}{1-\al}$ is small enough, there exists a
real-analytic function $\CCC$ defined in $R_{\kk_2}\times\TT_\sigma$ such that
the change
\begin{equation}\label{def:CanviConjugador}
 (u,\tau)=(v+\CCC(v,\tau),\tau)
\end{equation}
 conjugates the operators $\wt\LL_\eps$ and $\LL_\eps$ defined in
\eqref{def:OperadorAnulador} and \eqref{def:Lde} respectively. Moreover, for $(v,\tau)\in
R_{\kk_2}\times\TT_\sigma$, $v+\CCC(v,\tau) \in R_{\kk_1}$ and there exists a
constant $b_2>0$ such that
\[
\begin{split}
 \left|\CCC(v,\tau)\right|&\leq b_2|\mu|\eps^{\eta}\frac{\eps+\sqrt{\al}}{1-\al}\\
 \left|\pa_v\CCC(v,\tau)\right|&\leq b_2|\mu|\eps^{\eta-1}
\frac{\eps+\sqrt{\al}}{1-\al}.
\end{split}
\]
Furthermore, $(u,\tau)=(v+\CCC(v,\tau),\tau)$ is invertible and its inverse is
of the form $ (v,\tau)=(u+\VV(u,\tau),\tau)$ where $\VV$ is a function defined
for $(u,\tau)\in R_{\kk_3}\times\TT_\sigma$, which satisfies 
\[
 \left|\VV(u,\tau)\right|\leq b_2|\mu|\eps^\eta\frac{\eps+\sqrt{\al}}{1-\al}
\]
and that $u+\VV(u,\tau)\in R_{\kk_2}$ for $(u,\tau)\in
R_{\kk_3}\times\TT_\sigma$.
\end{theorem}
The proof of this theorem is deferred to Section \ref{sec:Canvi}.

Once we have obtained the change of variables $\CCC$, we are able to prove the validity of
the Melnikov prediction. We prove it by bounding
$\pa_u\Delta(u,\tau)-\MM(u,\tau)$, where $\Delta$ and $\MM$ are the functions
defined in \eqref{def:DiffManifolds} and \eqref{def:Melnikov} respectively. 

As a first step,  we
consider 
\[\pa_v\left(\Delta(v+\CCC(v,\tau),\tau)\right)-\MM(v,\tau)
\]
where $\CCC$ is the function obtained in Theorem \ref{th:Canvi}.

\begin{theorem}\label{th:CotaExpPetita}
 There exists $\eps_0>0$ and $b_3>0$ such that for any $\eps\in (0,\eps_0)$  and $\alpha \in (0,1)$
satisfying \eqref{cond:ExpSmall} and that $\eps^{\eta-1}\frac{\eps+\sqrt{\al}}{1-\al}$ is small enough, the following bound is satisfied
\[
\left|\pa_v\left(\Delta(v+\CCC(v,\tau),\tau)\right)-\MM(v,\tau)\right|\leq
b_3|\mu|^2\frac{\eps^{2\eta-2}}{(1-\al)^{2}}e^{-\dps\tfrac{\Im\rr_-}{\eps}}
\]
for $v\in
R_{\kk_2}\cap\RR$ and $\tau\in\TT$,
\end{theorem}

\begin{proof}
First, we define  the Melnikov potential  $L$ (see \cite{DelshamsG00}), namely a function such that $\pa_v L(v,\tau)=\MM(v,\tau)$,
and we study $\Delta(v+\CCC(v,\tau),\tau)-L(v,\tau)$.

As $\Delta\in \mathrm{Ker}\wt\LL_\eps$, where $\wt\LL_\eps$ is the operator in \eqref{def:OperadorAnulador}, by Theorem \ref{th:Canvi},
the function $\Phi(v,\tau)=\Delta(v+\CCC(v,\tau),\tau)-L(v,\tau)\in\mathrm{Ker}\LL_\eps$, where
$\LL_\eps$ is the operator defined in \eqref{def:Lde}. Therefore we can apply
 Lemma \ref{lemma:Lazutkin}. To this end, we have to bound $\pa_v\Phi(v,\tau)$
in the domain $R_{\kk_2}\times\TT_\sigma$. We split $\Phi$ as
$\Phi(v,\tau)=\Phi_1(v,\tau)+\Phi_2(v,\tau)$ where
\[
\begin{split}
 \Phi_1(v,\tau)&= \Delta\left(v+\CCC(v,\tau),\tau\right)-\Delta (v,\tau)\\
\Phi_2(v,\tau)&= \Delta(v,\tau)-L (v,\tau).
\end{split}
\]
where $\Delta$ is the functions defined in
\eqref{def:DiffManifolds}
and $L$ is the Melnikov potential.

To bound $\pa_v\Phi_1$, one has to take into account that 
$\Delta=(T^s-T_0)-(T^u-T_0)$ and therefore, it is enough to consider the bounds
obtained in Theorems \ref{th:ExistenceManifolds} and \ref{th:Canvi} to obtain
\[
|\pa_v\Phi_1(v,\tau)|\leq K|\mu|\eps^{2\eta-2}\frac{1}{(1-\al)^2}
\] for $(v,\tau)\in
R_{\kk_2}\times\TT_\sigma$. For the second term, it is enough to use bounds
\eqref{def:HalfMelnikov:uns:cota} and \eqref{def:HalfMelnikov:st:cota} to obtain
\[
|\pa_v\Phi_2(v,\tau)|\leq K|\mu|\eps^{2\eta-2}\frac{1}{(1-\al)^2}\] for $(v,\tau)\in
R_{\kk_2}\times\TT_\sigma$.
 
Therefore, we have that \[|\pa_v\Phi(v,\tau)|\leq K|\mu|\eps^{2\eta-2}\frac{1}{(1-\al)^{2}}\] for
$(v,\tau)\in R_{\kk_2}\times\TT_\sigma$ and then, it is enough to apply Lemma
\ref{lemma:Lazutkin} to finish the proof of Theorem \ref{th:CotaExpPetita}.
\end{proof}

From this result and the exponential smallness of $\MM$, given in Proposition \ref{prop:Melnikov}, and considering the inverse change $(v,\tau)=(u+\VV(u,\tau),\tau)$ obtained in Theorem \ref{th:Canvi}, it is
straightforward to obtain an exponentially small bound for 
$\pa_u\Delta(u,\tau)-\MM(u,\tau)$. It is stated in the next
corollary, whose proof is straightforward.
\begin{corollary}\label{coro:CotaExpPetita}
 There exists $\eps_0>0$ and $b_4>0$ such
that for any $\eps\in (0,\eps_0)$  and $\alpha \in (0,1)$
satisfying \eqref{cond:ExpSmall} and that $\eps^{\eta-1}\frac{\eps+\sqrt{\al}}{1-\al}$ is small enough, the following bound is satisfied
\[
\left|\pa_u\Delta (u,\tau)-\MM(u,\tau)\right|\leq
b_4|\mu|^2\frac{\eps^{2\eta-2}}{(1-\al)^{2}}e^{-\dps\tfrac{\Im\rr_-}{\eps}}
\]
for $u\in R_{\kk_3}\cap\RR$ and $\tau\in\TT$.
\end{corollary}
From this corollary and using that, by Proposition \ref{prop:Melnikov} and Corollary \ref{coro:MidaMelnikov}, the Melnikov function has non-degenerate zeros, one can see that if $\al$ is bounded away from 1 and satisfies that $\eps^{\eta-1}(\eps+\sqrt{\al})$ is small enough, the manifolds intersect transversally and their distance satisfy the desired asymptotic formula. This finishes the proof of Theorem \ref{th:Main} (see \cite{BaldomaFGS11}).

\subsubsection{The invariant manifolds: Proof of Theorem
\ref{th:ExistenceManifolds}}\label{sec:ExistenceManifolds}
Since the proof for both invariant manifolds is analogous, we only deal with the
unstable case. We look
for a solution of equation \eqref{def:HJ2} satisfying the asymptotic condition
\eqref{eq:AsymptCondFuncioGeneradora:uns}. We look for it as a perturbation
of the unperturbed separatrix $T_0$ in \eqref{def:T0} and therefore we work
with 
\begin{equation}\label{def:T1}
Q(u,\tau) = T (u,\tau) - T_0(u).
\end{equation}
Replacing $T$ in equation \eqref{def:HJ2}, it is straightforward
to see that the equation for $Q$ reads
\begin{equation}\label{eq:Existence:T1}
 \LL_\eps Q=\FF(\pa_u Q,u,\tau)
\end{equation}
where $\LL_\eps$ is the operator defined in \eqref{def:Lde} and
\begin{equation}\label{def:Existence:OperadorRHS}
 \FF(h,u,\tau)=-\frac{\cosh^2u}{8}h^2-\mu\eps^\eta\Psi(u)\sin\tau.
\end{equation}
where $\Psi(u)$ is the function defined in \eqref{def:Psi}.

We devote the rest of the section to obtain a solution of equation
\eqref{eq:Existence:T1} which is defined  in  $D^u_{\kk}\times\TT_\sigma$ and
satisfies the asymptotic condition \eqref{eq:AsymptCondFuncioGeneradora:uns}.

We start defining a norm for functions defined in the domain $D^u_{\kk}$ with
$\kk>0$. Since we want to capture their behavior both as $\Re u\rightarrow
-\infty$, namely its exponential decay, and for $u$ close to $u=ia$, we consider
weighted norms with different weights. For this reason we consider $U>0$ to
divide  $D^u_{\kk}$ by the vertical line $\Re u=-U$.  Then, given $\kk>0$ and an
analytic function $h:D^u_\kk\rightarrow \CC$, we consider
\begin{equation}\label{def:Norma:Existencia}
\begin{split}
 \|h\|=&\sup_{u\in D^u_\kk\cap\{\Re u<-U\}}\left|e^{-2 u}h(u)\right|\\
&+\sup_{u\in
D^u_\kk\cap\{\Re
u>-U\}}\left|(u-\rr_-)^2(u-\rr_+)(u-\ol\rr_-)^2(u-\ol\rr_+)h(u)\right|.
\end{split}
\end{equation}
For analytic functions $h:D^u_\kk\times\TT_\sigma\rightarrow \CC$,
 we consider the corresponding Fourier norm
\[
 \|h\|_\sigma=\sum_{k\in\ZZ}\left\|h^{[k]}\right\| e^{|k|\sigma}.
\]
and the following Banach  space
\begin{equation}\label{def:Existence:banach}
 \EE_{\kk,\sigma}=\left\{h:D^u_\kk\times\TT_\sigma\rightarrow \CC;
\text{real-analytic}, \|h\|_\sigma<\infty\right\}.
\end{equation}

To obtain the solutions of equation \eqref{eq:Existence:T1}, we need to solve an
equation of the form
$\LL_\eps h=g$, where $\LL_\eps$ is the differential operator
defined in~\eqref{def:Lde}. Note that  $\LL_\eps$ is invertible in
$\EE_{\kk,\sigma}$. It
turns out that its inverse is $\GG_\eps$ defined by
\begin{equation}\label{def:operadorGInfty}
\GG_\eps (h)(u,\tau)=\int_{-\infty}^0h(u+s,\tau+\eps\ii s)\,ds.
\end{equation}

\begin{lemma}\label{lemma:Existence:PropietatsGInfty}
The operator $\GG_\eps$  in \eqref{def:operadorGInfty}  satisfies the
following properties.
\begin{enumerate}
\item $\GG_\eps$ is linear from $\EE_{\kk,\sigma}$ to
itself, commutes with $\pa_u$ and satisfies
$\LL_\eps\circ\GG_\eps=\mathrm{Id}$.
\item If $h\in\EE_{\kk,\sigma}$, then
\[
\left\|\GG_\eps(h)\right\|_{\sigma}\leq
K\|h\|_{\sigma}.
\]
Furthermore, if $\langle h\rangle=0$, then
\[
\left\|\GG_\eps(h)\right\|_{\sigma}\leq
K\eps\|h\|_{\sigma}.
\]
\item If $h\in\EE_{\kk,\sigma}$, then $\pa_u\GG_\eps(h)\in
\EE_{\kk,\sigma}$ and
\[
\left\|\pa_u\GG_\eps(h)\right\|_{\sigma}\leq
K\|h\|_{\sigma}.
\]
\end{enumerate}
\end{lemma}
\begin{proof}
It follows the same lines as the proof of Lemma 5.5 in
\cite{GuardiaOS10}.
\end{proof}

Once we have obtained an inverse of the operator $\LL_\eps$ defined in
\eqref{def:Lde} we can obtain solutions of equation \eqref{eq:Existence:T1}
using a fixed point argument. Then, Theorem \ref{th:ExistenceManifolds} is a
straightforward consequence of the following proposition.

\begin{proposition}\label{prop:ExistenceT1:PtFix}
Let us fix $\kk_1>0$. There exists $\eps_0>0$ such that for any $\eps\in
(0,\eps_0)$ and $\al\in (0,1)$ satisfying \eqref{cond:ExpSmall} and that $\frac{\eps+\sqrt{\alpha}}{1-\al}\eps^{\eta-1}$
is small enough, there exists a function $Q$ defined in
$D^u_{\kk_1}\times\TT_\sigma$ such that $\pa_u Q\in \EE_{\kk_1,\sigma}$ is a
fixed point of the operator
\begin{equation}\label{def:ExistenceT1:Operator}
 \ol \FF^u(h)=\pa_u\GG_\eps\FF(h),
\end{equation}
where $\GG_\eps$ and $\FF$ are the operators defined in
\eqref{def:operadorGInfty} and \eqref{def:Existence:OperadorRHS} respectively.
Furthermore, there exists a constant $b_1>0$ such that,
\[
\begin{split}
 \left\|\pa_uQ\right\|_\sigma&\leq b_1|\mu|\eps^{\eta+1},\\
 \left\|\pa^2_uQ\right\|_\sigma&\leq b_1|\mu|\eps^{\eta}.
\end{split}
\]
Moreover, if we consider the half Melnikov function defined in
\eqref{def:HalfMelnikov:uns}, we have that
\begin{equation}\label{eq:CotaGeneradoraMenysMelnikov}
 \left\| \pa_u Q-\mu\eps^\eta\MM^u\right\|_\sigma\leq K|\mu|^2\eps^{2\eta}
\frac{\eps+\sqrt{\al}}{1-\al}.
\end{equation}
\end{proposition}
\begin{proof}
Let us consider $\kk_0<\kk_1$. It is straightforward to see that $\ol\FF^u$ is
well defined from $\EE_{\kk_0,\sigma}$ to itself. We are going to prove that
there exists a constant $b_1>0$ such that $\ol\FF^u$ sends $\ol
B(b_1|\mu|\eps^{\eta+1})\subset\EE_{\kk_0,\sigma}$ to itself and is contractive
there.

Let us first consider $\ol\FF^u(0)$. From the definition of $\ol\FF^u$ in
\eqref{def:ExistenceT1:Operator}, the definition of $\FF$ in
\eqref{def:Existence:OperadorRHS} and using Lemma
\ref{lemma:Existence:PropietatsGInfty}, we have that
\begin{equation}\label{eq:PrimeraIteracio:ExistenceQ}
 \ol\FF^u(0)(u,\tau)=\pa_u\GG_\eps\FF(0)(u,
\tau)=-\mu\eps^\eta\GG_\eps\left(\Psi'(u)\sin\tau\right).
\end{equation}
To bound it, first let us point out that $\Psi'(u)=\beta(u)$ where $\beta(u)$ is
the function defined in \eqref{def:MelnikovIntegrand}. We bound each term for $\Im u\geq 0$, the other case is analogous. Using
\eqref{def:beta:singularity} and taking into account that
\begin{equation}\label{eq:Cota:Singularitat:0}
 \left|\frac{u-i\pi/2}{u-\rr_+}\right|<1, 
\end{equation}
one can easily see that $\|\Psi'(u)\|<K$. Then, taking into account that
$\sin\tau$ has zero average and applying Lemma 
 \ref{lemma:Existence:PropietatsGInfty}, there exists a
constant $b_1>0$ such that
\[
 \left\|\ol\FF^u(0)\right\|_\si\leq\frac{b_1}{2}|\mu|\eps^{\eta+1}.
\]
To bound the Lipschitz constant, let us consider  $h_1,h_2\in \ol
B(b_1|\mu|\eps^{\eta+1})\in\EE_{\kk_0,\sigma}$. To bound $\|\ol \FF^u(h_2)-\ol
\FF^u(h_1)\|_\si$, we need first the following bounds for $u\in D_{\kk_0}^u$,
\begin{equation}\label{eq:Cota:Singularitat}
\left|\frac{u-i\pi/2}{u-i\rr_-} \right|\leq
1+\left|\frac{i\pi/2-\rr_-}{u-i\rr_-} \right|\leq 1+K\frac{\sqrt{\al}}{\eps}
\end{equation}
and 
\begin{equation}\label{eq:Cota:Singularitat:2}
\left|\frac{1}{(u-i\rr_-)(u-i\ol\rr_-)^2} \right|\leq \frac{K}{\eps (1-\al)},
\end{equation}
which are a direct consequence of \eqref{eq:Expansio:Sing} and \eqref{eq:Expansio:Sing:alfa1}.

Then, it is easy to see that 
\[
\begin{split}
\left \|\ol \FF^u(h_2)-\ol \FF^u(h_1)\right\|_\si&\leq
\frac{K}{\eps(1-\al)}\left(1+\frac{\sqrt{\al}}{\eps}\right)\|h_2+h_1\|_\si\|h_2-h_1\|_\si\\
&\leq K|\mu|\frac{\eps+\sqrt{\alpha}}{1-\al}\eps^{\eta-1}\|h_2-h_1\|_\si.
\end{split}
\]
Then, using that by hypothesis $\frac{\eps+\sqrt{\alpha}}{1-\al}\eps^{\eta-1}$ is
small enough, 
\[
\mathrm{Lip}\ol\FF^u= K|\mu|\frac{\eps+\sqrt{\alpha}}{1-\al}\eps^{\eta-1}<1/2
\]
 and therefore $\ol\FF^u$ is contractive
from the ball $\ol B(b_1|\mu|\eps^{\eta+1})\subset \EE_{\kk_0,\sigma}$ into
itself, and it has a unique fixed point $h^\ast$. Moreover, since it has
exponential decay as $\Re u\rightarrow -\infty$, we can take
\[
Q(u,\tau)=\int_{-\infty}^u h^\ast (v,\tau)dv.
\]
To obtain the bound for $\pa_u^2Q$, it is enough to apply Cauchy estimates to
the nested domains $D_{\kk_1}^u\subset D_{\kk_0}^u$ (see, for instance,
\cite{GuardiaOS10}), and rename $b_1$ if necessary.

Finally, to prove \eqref{eq:CotaGeneradoraMenysMelnikov}, it is enough to point
out that $\ol\FF^u(0)=\mu\eps^\eta\MM^u$ and consider the obtained bound for the
Lipschitz constant.
\end{proof}

\subsubsection{Straightening the operator $\wt\LL_\eps$: proof of Theorem
\ref{th:Canvi}}\label{sec:Canvi}
We devote this section to prove Theorem \ref{th:Canvi}. It is a well known fact,
see for instance Lemma 6.3 of \cite{GuardiaOS10}, that looking for a change
\eqref{def:CanviConjugador} which conjugates $\wt \LL_\eps$ and $\LL_\eps$
defined in \eqref{def:OperadorAnulador} and \eqref{def:Lde} is equivalent to
looking for a function $\CCC$ solution of the equation
\[
 \LL_\eps\CCC(v,\tau)=\left.\frac{\cosh^2u}{8}\left(\pa_uT^s(u,\tau)+\pa_u
T^u(u,\tau)\right)\right|_{u=v+\CCC(v,\tau)}-1.
\]
Taking into account the definition of $T_0$ and $Q$ in \eqref{def:T0} and
\eqref{def:T1}, this equation can be written as
\begin{equation}\label{eq:EDPCanvi}
 \LL_\eps\CCC=\JJ(\CCC)
\end{equation}
where
\begin{equation}\label{def:EDPCanvi:RHS}
\JJ(h)(v,\tau)= \left.\frac{\cosh^2u}{8}\left(\pa_uQ^s(u,\tau)+\pa_u
Q^u(u,\tau)\right)\right|_{u=v+h(v,\tau)}.
\end{equation}
To look for a solution of this equation, we start by defining some norms and
Banach spaces. Given $n\in \NN$ and a function $h:R_\kk\rightarrow\CC$, we
define 
\begin{equation}\label{def:Norma:Canvi}
 \|h\|_n=\sup_{v\in R_\kk}\left|(v-\rr_-)^n(v-\ol\rr_-)^nh(u)\right|.
\end{equation}
Moreover for analytic functions
$h:R_{\kk}\times\TT_\sigma\rightarrow \CC$, we define the corresponding Fourier norm
\[
\|h\|_{n,\sigma}=\sum_{k\in\ZZ}\left\|h^{[k]}\right\|_{n}e^{|k|\sigma}
\]
and the Banach space
\[
 \XX_{n,\sigma}=\{
h:R_\kk\times\TT_\sigma\rightarrow\CC;\,\,\text{real-analytic},
\|h\|_{n,\sigma}<\infty\}.
\]

To obtain a solution of equation \eqref{eq:EDPCanvi} in the domain $R_\kk$, we need to  solve
equations of the form
$\LL_\eps h=g$, where $\LL_\eps$ is the operator defined in
\eqref{def:Lde}. To find a right-inverse of this operator in
 $\XX_{n,\sigma}$ let us consider $u_1$ the upper vertex of $R_\kk$ and  $u_0$ the left
endpoint of $R_\kk$. Then, we define the operator
$\wt\GG_\eps$ as
\begin{equation}\label{def:operador:canvi}
\wt\GG_\eps(h)(v,\tau)=\sum_{k\in\ZZ}\wt\GG_\eps(h)^{[k]}(v)e^{ik\tau},
\end{equation}
 where its Fourier coefficients are given by
\begin{align*}
\dps \wt\GG_\eps(h)^{[k]}(v)&=\int_{-v_1}^v
e^{i\frac{k}{\eps}(w-v)}h^{\left[k\right]}(w)\,dw&\textrm{ if }k<0\\
\dps \wt\GG_\eps(h)^{[0]}(v)&= \int_{v_0}^{v} h^{\left[0\right]}(w)\,dw&\\
\dps \wt\GG_\eps(h)^{[k]}(v)&= -\int^{v_1}_v
e^{i\frac{k}{\eps}(w-v)}h^{\left[k\right]}(w)\,dw&\textrm{ if }k>0.
\end{align*}
The following lemma, which is proved in \cite{GuardiaOS10} (see Lemma 8.3 of
this paper), gives some properties of this operator.

\begin{lemma}\label{lemma:Canvi:Operador}
The operator $\wt\GG_\eps$ in \eqref{def:operador:canvi} satisfies the
following properties.
\begin{enumerate}
\item If $h\in \XX_{n,\sigma}$, then
$\wt\GG_\eps(h)\in \XX_{n,\sigma}$ and
\[
\left\|\wt\GG_\eps(h)\right\|_{n,\sigma}\leq K\|h\|_{n,\sigma}.
\]
Moreover, if $\langle h\rangle=0$,
\[
\left\|\wt\GG_\eps(h)\right\|_{n,\sigma}\leq K\eps\|h\|_{n,\sigma}.
\]
\item If $h\in\XX_{n,\sigma}$ with $n>1$, then $\wt\GG_\eps(h)\in
\XX_{n-1,\sigma}$ and
\[
\left\|\wt\GG_\eps(h)\right\|_{n-1,\sigma}\leq \frac{K}{\sqrt{1-\al}}\|h\|_{n,\sigma}.
\]
\end{enumerate}
\end{lemma}

In next proposition we obtain a solution of equation \eqref{eq:EDPCanvi} using a fixed point argument.

\begin{proposition}\label{prop:Canvi:PtFix}
Let us consider the constant $\kk_1>0$ defined in Theorem
\ref{th:ExistenceManifolds} and let us consider any $\kk_2>\kk_1$. There exists
$\eps_0>0$ such that for any $\eps\in (0,\eps_0)$ and $\al\in (0,1)$ satisfying \eqref{cond:ExpSmall} and
that $\frac{\eps+\sqrt{\al}}{1-\al}\eps^{\eta-1}$ is small enough, there exists a function
$\CCC\in \XX_{1,\sigma}$ defined in $R_{\kk_2}\times\TT_\sigma$ such that  is a
fixed point of the operator
\begin{equation}\label{def:Canvi:Operator}
 \ol \JJ(h)=\wt\GG_\eps\JJ(h)
\end{equation}
where $\wt\GG_\eps$ and $\JJ$ are the operators defined in
\eqref{def:operador:canvi} and \eqref{def:EDPCanvi:RHS} respectively.
Furthermore, $v+\CCC(v,\tau)\in R_{\kk_1}$ for $(v,\tau)\in
R_{\kk_2}\times\TT_\sigma$ and there exists a constant $b_2>0$ such that
\[
\begin{split}
 \left\|\CCC\right\|_{1,\sigma}&\leq
b_2|\mu|\eps^{\eta+1}\frac{\sqrt{\al}+\eps}{(1-\al)^{1/2}}\\
 \left\|\pa_v\CCC\right\|_{1,\sigma}&\leq
b_2|\mu|\eps^{\eta}\frac{\sqrt{\al}+\eps}{(1-\al)^{1/2}}.
\end{split}
\]
\end{proposition}
\begin{proof}
It is straightforward to see that $\ol\JJ$ is well defined from $\XX_{1,\sigma}$
to itself. We are going to prove that there exists a constant $b_2>0$ such that
$\ol\JJ$ sends $\ol
B(b_2|\mu|\eps^{\eta+1}\frac{\sqrt{\al}+\eps}{(1-\al)^{1/2}})\subset\XX_{1,\sigma}$ to
itself and is contractive there.

Let us first consider $\ol\JJ(0)$. From the definition of $\ol\JJ$ in
\eqref{def:Canvi:Operator}, the definition of $\JJ$ in \eqref{def:EDPCanvi:RHS}, we have that
\[
 \ol\JJ(0)(v,\tau)=\wt\GG_\eps\JJ(0)(v,\tau)=\wt\GG_\eps\left(\frac{\cosh^2v}{8}
\left(\pa_vQ^s(v,\tau)+\pa_v Q^u(v,\tau)\right)\right).
\]
Using that $\pa_v Q^u$ is a fixed point of the operator $\ol\FF^u$ in
\eqref{def:ExistenceT1:Operator} and that $\pa_v Q^s$ is a fixed point of an
analogous operator $\ol\FF^s$, we can split  
$\ol\JJ(0)(v,\tau)=\BB_1(v,\tau)+\BB_2(v,\tau)$ with 
\[
\begin{split}
 \BB_1(v,\tau)&=\wt\GG_\eps\left(\frac{\cosh^2v}{8}\left(\ol\FF^u(0)(v,
\tau)+\ol\FF^s(0)(v,\tau)\right)\right)\\
\BB_2(v,\tau)&=\wt\GG_\eps\left(\frac{\cosh^2v}{8}\left(\ol\FF^u(\pa_v
Q^u)(v,\tau)-\ol\FF^u(0)(v,\tau)+\ol\FF^s(\pa_v
Q^s)(v,\tau)-\ol\FF^s(0)(v,\tau)\right)\right)
\end{split}
\]
To bound $\BB_1$, it is enough to recall that, in the proof
of Proposition \ref{prop:ExistenceT1:PtFix}, we have seen that 
$\|\ol\FF(0)\|_\sigma\leq K|\mu|\eps^{\eta+1}$. Then, taking into account \eqref{eq:Cota:Singularitat:0} and
\eqref{eq:Cota:Singularitat}, one can see that 
\[
 \left\|\frac{\cosh^2v}{8}\left(\ol\FF^u(0)(v,\tau)+\ol\FF^s(0)(v,
\tau)\right)\right\|_{1,\sigma}\leq
K|\mu|\eps^{\eta}\frac{\sqrt{\al}+\eps}{(1-\al)^{1/2}}.
\]
Moreover, taking into account that $\left\langle \ol\FF^{u,s} (0)\right\rangle=0$
and applying Lemma \ref{lemma:Canvi:Operador}, we have that 
\[
\|\BB_1\|_{1,\si}\leq K|\mu|\eps^{\eta+1}\frac{\sqrt{\al}+\eps}{(1-\al)^{1/2}}.
\]
For the second term, let us first point out that
\[
 \ol\FF^*(\pa_vQ^*)(v,\tau)-\ol\FF^*(0)(v,
\tau)=-\pa_v\GG_\eps\left(\frac{\cosh^2v}{8}\left(\pa_v Q^\ast\right)^2\right),\,\,\ast=u,s.
\]
Using Proposition \ref{prop:ExistenceT1:PtFix} and \eqref{eq:Cota:Singularitat:0}, one has that 
\[
 \left|\frac{\cosh^2v}{8}\left(\pa_v Q^\ast(v,\tau)\right)^2\right|\leq \frac{K|\mu|^2\eps^{2\eta+2}}{(v-\rr_-)^4(v-\ol\rr_-)^4}.
\]
Therefore, analogously to Lemma \ref{lemma:Existence:PropietatsGInfty}, one can easily see that 
\[
 \left|\pa_v\GG_\eps\left(\frac{\cosh^2v}{8}\left(\pa_v Q^\ast(v,\tau)\right)^2\right)\right|\leq \frac{K|\mu|^2\eps^{2\eta+2}}{(v-\rr_-)^4(v-\ol\rr_-)^4}.
\]
Using inequalities \eqref{eq:Cota:Singularitat} and \eqref{eq:Cota:Singularitat:2}, 
\[
\left\|
\frac{\cosh^2v}{8}\left(\ol\FF^*(\pa_vQ^*)(v,\tau)-\ol\FF^*(0)(v,
\tau)\right)\right\|_{2,\si}\leq K|\mu|^2\eta^{2\eta+2}\frac{\left(\sqrt{\al}+\eps\right)^2}{1-\al}.
\]
Then, using Lemma \ref{lemma:Canvi:Operador}, one has that 
\[
 \|\BB_2\|_{1,\sigma}\leq  K|\mu|^2\eps^{2\eta+2}\frac{\left(\sqrt{\al}+\eps\right)^2}{(1-\al)^{3/2}}.
\]
Therefore, since $\eps^{\eta-1}\frac{\sqrt{\al}+\eps}{1-\al}\ll 1$, there exists
a constant $b_2>0$ such that
\[
 \left\| \ol\JJ(0)\right\|_{1,\si}\leq
\frac{b_2}{2}|\mu|\eps^{\eta+1}\frac{\sqrt{\al}+\eps}{(1-\al)^{1/2}}.
\]
To bound the Lipschitz constant, it is enough to apply the mean value theorem,
use the bounds of $\pa_u Q^{u,s}$ and $\pa^2_u Q^{u,s}$ given in Proposition \ref{prop:ExistenceT1:PtFix} and
Lemma \ref{lemma:Canvi:Operador} to see that 
\[
 \mathrm{Lip}\leq K|\mu|\eps^{\eta-1}\frac{\sqrt{\al}+\eps}{1-\al}.
\]
Then, using that $\eps^{\eta-1}\frac{\sqrt{\al}+\eps}{1-\al}\ll 1$, the operator
$\ol\JJ$ is contractive from  $\ol
B(b_2|\mu|\eps^{\eta+1}\frac{\sqrt{\al}+\eps}{(1-\al)^{1/2}})\subset\XX_{1,\sigma}$ to
itself and it has a unique fixed point $\CCC$. Finally, to obtain a bound for
$\pa_v\CCC$ it is enough to apply Cauchy estimates reducing slightly the domain
and renaming $b_2$ if necessary.
\end{proof}

\begin{proof}[Proof of Theorem \ref{th:Canvi}]
 Once we have proved Proposition \ref{prop:Canvi:PtFix},  it only remains to
obtain the inverse change given by the function $\VV$, which is straightforward
using a fixed point argument.
\end{proof}

\subsection{Proof of Theorem \ref{th:Main:NonExp}}\label{sec:SketchProof:NonExp}
The first statement of Theorem \ref{th:Main:NonExp} is a direct consequence of Corollary \ref{coro:CotaExpPetita} taking $\al=1-C\eps^r$ with $r\in (0,2)$. Note that the condition $\eps^{\eta-1}\frac{\eps+\sqrt{\al}}{1-\al}$ becomes $\eta>r+1$. The proof of the second  and third statements, which correspond to $r\geq 2$, are considerably simpler, since we do not need to prove any exponential smallness. The first observation is that $\Im\rr_-\sim\eps^{r/2}\leq \eps$. Even though in this case it is not necessary,  we keep the analyticity properties of the parameterizations of the invariant manifolds and we work in the domains $D^u_\kk\times \TT_\sigma$ and $D^s_\kk\times \TT_\sigma$ (see \eqref{def:DominisOuter}).

\begin{theorem}\label{th:ExistenceManifolds:NonExp}
Let us fix $\kk_1>0$. Then, there exists $\eps_0>0$ such
that for $\eps\in(0,\eps_0)$, $\alpha=1-C\eps^r$ with $C>0$ and $r\geq 2$, $\mu\in B(\mu_0)$, if
$\eta-3r/2>0$, the
Hamilton-Jacobi equation \eqref{def:HJ1} has  a unique
(modulo an additive constant) real-analytic solution in
$\DD_{\kk_1}^u\times\TT_\sigma$ satisfying the asymptotic
condition \eqref{eq:AsymptCondFuncioGeneradora:uns}.

Moreover, there exists a real constant $b_4>0$ independent of $\eps$
and $\mu$, such that for $(u,\tau)\in
\DD_{\kk_1}^{u}\times\TT_\sigma$,
\[
\left|\pa_u T^u(u,\tau)-\pa_u T_0(u)\right|\leq
b_4|\mu|\eps^{\eta-3r/2}.
\]

Furthermore, for $(u,\tau)\in
\DD_{\kk_1}^{u}\times\TT_\sigma$,
the generating function $T^u$  satisfies that
\begin{equation}\label{def:HalfMelnikov:uns:cota:NonExp}
\left|\pa_u T^u(u,\tau)-\pa_u T_0(u)-\mu\eps^\eta\MM^u(u,\tau)\right|\leq
b_4|\mu|^2\eps^{2\eta-3r}
\end{equation}
where $\MM^u$ is the function defined in \eqref{def:HalfMelnikov:uns}.
\end{theorem}

\begin{proof}
The proof follows the same lines of Theorem \ref{th:ExistenceManifolds}. We use the modified norm
\begin{equation}\label{def:Norma:NonExp}
 \|h\|=\sup_{u\in \DD^u\cap\{\Re u<-U\}}\left|e^{-2 u}h(u)\right|+\sup_{u\in
\DD^u\cap\{\Re
u>-U\}}\left|h(u)\right|.
\end{equation}
Then, using \eqref{eq:PrimeraIteracio:ExistenceQ}, one can bound $\ol\FF(0)$ as
\[
\left\|\ol \FF(0)\right\|=K|\mu|\eps^\eta\left(K+\int_{-U}^u \frac{1}{|v-\rr_-|^2|v-\ol\rr_-|^2}\right) \leq \frac{b_4}{2}|\mu|\eps^{\eta-3r/2}.
\]
Proceeding as before, it is straightforward to see that $\ol\FF$ is contractive from the ball $\ol B(\mu\eps^{\eta-3r/2})$ to itself with Lipschitz constant satisfying
\[
 \mathrm{Lip}\leq |\mu|\eps^{\eta-3r/2},
\]
which gives the desired result.
\end{proof}

The function $T^s$ satisfies the same properties in the symmetric domain $\DD_{\kk_1}^s$. Note that $D_{\kk_1}^u\cap D_{\kk_1}^s\cap\RR$ is an interval of size $\OO(\eps^r/2)$ centered at $u=\ln(1+\sqrt{2})$. Therefore, $D_{\kk_1}^u\cap D_{\kk_1}^s\cap\RR$ does not contain a fundamental domain. Nevertheless, it suffices to deal with this domain to compute the distance between the manifolds in the section $x=3\pi/2$.

From Theorem \ref{th:ExistenceManifolds:NonExp} the formula of the distance follows. Thus, to finish the proof of Theorem \ref{th:Main:NonExp}, it is enough to use Corollary \ref{coro:NonExpSmall}, for $r\geq 2$, to check that the Melnikov function $\MM(\ln(1+\sqrt{2}),\tau)$ has simple zeros so that the manifolds intersect transversally.

\section{The quasiperiodic case: Proof of Theorems
\ref{th:Main:QP} and \ref{th:Main:NonExp:QP}}\label{sec:SketchProof:QP}
As we did in the periodic case, the proof in next section includes at the same time the results in Theorem \ref{th:Main:QP} and in the first statement of Theorem \ref{th:Main:NonExp:QP}. Then, Section \ref{sec:SketchProof:QP:NonExp} contains the proof of the second and third statement of Theorem \ref{th:Main:NonExp:QP}.

\subsection{Proof of Theorem \ref{th:Main:QP}}

We follow the same approach as in the periodic case. Therefore, we only point out the main differences with respect to it. We perform the symplectic change of variables 
\eqref{eq:CanviSimplecticSeparatriu} to Hamiltonian \eqref{def:ToyModel:Hamiltonian:QP}. In the new variables, it reads
\begin{equation}\label{def:ToyModel:HamiltonianReparameterized:QP}
\eps \ol K\left(u,w,\theta_1,\theta_2,I_1,I_2\right)=\eps K\left(x_0(u),\frac{w}{y_0(u)},\theta_1,\theta_2,I_1,I_2\right).
\end{equation}
As in the periodic case,  we look for the perturbed invariant manifolds as graphs of the
gradient of generating functions $T^{u,s}(u,\theta_1,\theta_2)$, which are solutions of the
Hamilton-Jacobi equation
\begin{equation}\label{def:HJ1:QP}
\ol K\left(u, \pa_u T,\theta_1,\theta_2,\pa_{\theta_1}T,\pa_{\theta_2}T\right)=0.
\end{equation}
This equation reads
\begin{equation}\label{def:HJ2:QP}
 \eps\ii\pa_{\theta_1} T+\eps\ii\ga\pa_{\theta_2} T+\frac{\cosh^2
u}{8}\left(\pa_uT\right)^2-\frac{4}{\cosh^2u}+\mu\eps^\eta \Psi(u)F(\theta_1,\theta_2)=0,
\end{equation}
where $\Psi(u)$ is the function defined in \eqref{def:Psi}.

We impose the same asymptotic conditions \eqref{eq:AsymptCondFuncioGeneradora:uns} and \eqref{eq:AsymptCondFuncioGeneradora:st} as in the periodic case. For $\mu=0$, the solution of equation \eqref{def:HJ2:QP} is
\eqref{def:T0}. We study the existence of solutions of equation \eqref{def:HJ2:QP}  in
complex domains which have points close to  the singularities
$u=\rr_-, \ol\rr_-$. Nevertheless, in this case, we only need to stay at a distance of order $\OO(\sqrt{\eps})$ of the singularity instead of $\OO(\eps)$ as happened in the periodic case (see \cite{DelshamsGJS97}).  To this end, we define the modified domains
\begin{equation}\label{def:DominisOuter:QP}
\begin{split}
\dps D^{u}_{\kk}&=\left\{u\in\CC; |\Im u|<-\tan \beta_1\left(\Re u-\Re
\rr_-\right)+\Im \rr_--\kk\sqrt{\eps}\right\}\\
\dps D^{s}_{\kk}&=\left\{u\in\CC; |\Im u|<\tan \beta_1\left(\Re u-\Re
\rr_-\right)+\Im \rr_--\kk\sqrt{\eps}\right\}.
\end{split}
\end{equation}
Note that the only difference with respect to the domains \eqref{def:DominisOuter} is the change of $\eps$ by $\sqrt{\eps}$.

In the quasiperiodic case, the complex domain of the angular variables plays a crucial role. As it was done in \cite{DelshamsGJS97, Sauzin01}, we will prove the existence of these generating functions in very concrete domains in $(\theta_1,\theta_2)$. To this end let us define the complexified torus for any $\sigma=(\sigma_1,\sigma_2)\in \RR^2$ with $\sigma_1,\sigma_2>0$,
\[
 \TT^2_\sigma=\left\{(\theta_1,\theta_2)\in \left(\CC/\ZZ\right)^2: |\Im\theta_i|\leq\sigma_i\right\},
\]
 with 
\begin{equation}\label{def:BandesAngulars}
\sigma_i=r_i-d_i\sqrt{\eps},
\end{equation}
where $r_i$ are the constants defined in \eqref{hyp:QP1} and $d_i>0$ are any constants independent of $\eps$.

The next theorem gives the existence of the invariant manifolds in
the domains $D^{\ast}_{\kk}\times\TT^2_\sigma$ with $\ast=u,s$ (see \eqref{def:DominisOuter:QP}). We state the results for the unstable
invariant manifold. The stable one has analogous properties.

\begin{theorem}\label{th:ExistenceManifolds:QP}
Let us fix any $\kk_1,d_1,d_2>0$. Then, there exists $\eps_0>0$ such
that for $\eps\in(0,\eps_0)$, $\alpha\in (0,1)$, $\mu\in B(\mu_0)$, satisfying \eqref{cond:ExpSmall} and that
$\eps^{\eta-1}\frac{\sqrt{\eps}+\sqrt{\alpha}}{1-\al}$ is small enough the
Hamilton-Jacobi equation \eqref{def:HJ2:QP} has  a unique
(modulo an additive constant) real-analytic solution in
$D^{u}_{\kk_1}\times\TT^2_\sigma$, with $\sigma$ defined in \eqref{def:BandesAngulars}, satisfying the asymptotic
condition \eqref{eq:AsymptCondFuncioGeneradora:uns}.

Moreover, there exists a real constant $b_5>0$ independent of $\eps$
and $\mu$, such that for $(u,\theta_1,\theta_2)\in
D^{u}_{\kk_1}\times\TT^2_\sigma$,
\[
\begin{split}
\left|\pa_u T^u(u,\theta_1,\theta_2)-\pa_u T_0(u)\right|&\leq
\frac{b_5|\mu|\eps^{\eta-1}}{(\sqrt{\eps}+\sqrt{\al})(1-\al)}\\
\left|\pa^2_u T^u(u,\theta_1,\theta_2)-\pa_u^2 T_0(u)\right|&\leq
\frac{b_5|\mu|\eps^{\eta-\frac{3}{2}}}{(\sqrt{\eps}+\sqrt{\al})(1-\al)}.
\end{split}
\]
Furthermore, if we define the half Melnikov function
\begin{equation}\label{def:HalfMelnikov:uns:QP}
\MM^u(u,\theta_1,\theta_2)=-4\int_{-\infty}^{0} 
\frac{\sinh(u+s)\cosh(u+s)}{\left(\cosh^2(u+s)-2\alpha\sinh(u+s)\right)^2}
F\left(\theta_1+\frac{s}{\eps},\theta_2+\frac{\ga s}{\eps}\right)ds,
\end{equation}
the generating function $T^u$  satisfies that, for $(u,\theta_1,\theta_2)\in
D^{u}_{\kk_1}\times\TT^2_\sigma$,
\begin{equation}\label{def:HalfMelnikov:uns:QP:bound}
\left|\pa_u T^u(u,\theta_1,\theta_2)-\pa_u T_0(u)-\mu\eps^\eta\MM^u(u,\theta_1,\theta_2)\right|\leq
\frac{b_5|\mu|^2\eps^{2\eta-2}}{(1-\al)^2}.
\end{equation}
\end{theorem}
The proof of this theorem is deferred to Section \ref{sec:ExistenceManifolds:QP}.
The parameterization of the stable manifold has analogous properties. In
particular, we can define 
\begin{equation}\label{def:HalfMelnikov:st:QP}
\MM^s(u,\theta_1,\theta_2)=4\int_{0}^{+\infty} 
\frac{\sinh(u+s)\cosh(u+s)}{\left(\cosh^2(u+s)-2\alpha\sinh(u+s)\right)^2}
F\left(\theta_1+\frac{s}{\eps},\theta_2+\frac{\ga s}{\eps}\right)ds,
\end{equation}
and then, for $(u,\theta_1,\theta_2)\in
D^{s}_{\kk_1}\times\TT^2_\sigma$,
\begin{equation}\label{def:HalfMelnikov:st:QP:bound}
\left|\pa_u T^s(u,\theta_1,\theta_2)-\pa_u T_0(u)-\mu\eps^\eta\MM^s(u,\theta_1,\theta_2)\right|\leq
\frac{b_5|\mu|^2\eps^{2\eta-2}}{(1-\al)^2}.
\end{equation}
We consider the function
\begin{equation}\label{def:DiffManifolds:QP}
 \Delta (u,\theta_1,\theta_2)=T^s(u,\theta_1,\theta_2)-T^u(u,\theta_1,\theta_2).
\end{equation}
This function is defined in $R_{\kk_1}\times\TT_\sigma^2$, where $R_\kk$ is the romboidal domain defined by
\begin{equation}\label{def:DominiInterseccio:QP}
\dps R_{\kk}=D^{u}_{\kk}\cap D^{s}_{\kk},
\end{equation}
where $D^\ast_\kk$ are the domains defined in \eqref{def:DominisOuter:QP}.

Subtracting equation \eqref{def:HJ2:QP} for both $T^s$ and $T^u$, one can easily
see that $\Delta\in\mathrm{Ker}\wt\LL_\eps$ for 
\begin{equation}\label{def:OperadorAnulador:QP}
 \wt\LL_\eps=\eps\ii\pa_{\theta_1}+\eps\ii\ga\pa_{\theta_2}+\left(\frac{\cosh^2u}{8}\left(\pa_uT^s(u,
\theta_1,\theta_2)+\pa_u T^u(u,\theta_1,\theta_2)\right)\right)\pa_u.
\end{equation}
Since Theorem \ref{th:ExistenceManifolds:QP} ensures that the perturbed invariant
manifolds are well approximated by the unperturbed separatrix in the domains
$D^u_{\kk_1}\times\TT^2_\sigma$ and $D^s_{\kk_1}\times\TT^2_\sigma$, we know that the operator $\wt \LL_\eps$ is
close to the constant coefficients operator
\begin{equation}\label{def:Lde:QP}
 \LL_\eps=\eps\ii\pa_{\theta_1}+\eps\ii\ga\pa_{\theta_2}+\pa_u
\end{equation}
in the domain $R_{\kk_1}\times\TT_\sigma^2$.

Any function which is defined in $\{u\in\CC;\Re u=a,\Im u\in [-r_0,r_0]\}\times\TT^2_\sigma$, for any $a\in\RR$, and belongs to the kernel of
$\LL_\eps$  is defined in all the strip $\{|\Im u|<r_0\}\times\TT^2_\sigma$ and
has exponentially small bounds for real values of the variables. This fact is
summarized in the next lemma, whose proof follows the same lines as the one of
 Lemma 4.1 of \cite{Sauzin01}.

\begin{lemma}\label{lemma:Lazutkin:QP}
Let us consider a function $\zeta(u,\theta_1,\theta_2)$ analytic in
$(u,\theta_1,\theta_2)\in\{u\in \CC:\Re u=\Re \rr_-,|\Im u|< \Im\rr_--\kk\sqrt{\eps}\}\times\TT^2_\sigma$, where $\sigma=(\sigma_1,\si_2)$ with $\sigma_i=r_i-d_i\sqrt{\eps}$, which is
solution of
$\LL_\eps\zeta=0$. Then, $\zeta$ can be extended analytically to
$\{|\Im u|<\Im\rr_--\kk\sqrt{\eps}\}\times\TT^2_{\sigma}$ and its mean value
\[
\langle\zeta\rangle=\frac{1}{(2\pi)^2}\int_0^{2\pi}\int_0^{2\pi}\zeta(u,\theta_1,\theta_2)d\theta_1 d\theta_2
\]
does not depend on $u$. Moreover, for $\kk'>\kk$ and
$d_i'>d_i$, we define
\begin{equation}\label{def:M_r:QP}
M=\max_{(u,\theta_1,\theta_2)\in \left[
-\Im\rr_-+\kk'\sqrt{\eps},\Im\rr_--\kk'\sqrt{\eps}\right]\times\overline{\TT}^2_{\sigma'}}\left|\pa_u\zeta(u,\theta_1,\theta_2)\right|
\end{equation}
where $\si'=(\si'_1,\si'_2)$ with $\si'_i=r_i-d'_i\sqrt{\eps}$. Then, provided $\eps$ is small enough, for $(u,\theta_1,\theta_2)\in \RR\times\TT^2$,
\[
\left|\pa_u\zeta(u,\theta_1,\theta_2)\right|\leq 4  Me^{-c(\ln(\eps/\Im\rr_-)){\dps\sqrt{\tfrac{\Im\rr_-}{\eps}}}}.
\]
where $c$ is the periodic function defined in \eqref{def:c}.
\end{lemma}
To apply this lemma to study the difference between the invariant manifolds,
following \cite{Sauzin01} (see also \cite{GuardiaOS10}) we look for a change of
variables which conjugates $\wt\LL_\eps$ in \eqref{def:OperadorAnulador:QP} with $\LL_\eps$
in \eqref{def:Lde:QP}.

\begin{theorem}\label{th:Canvi:QP}
Let us consider the constant $\kk_1>0$ defined in Theorem
\ref{th:ExistenceManifolds} and let us fix any $\kk_3>\kk_2>\kk_1$. Then, there
exists $\eps_0>0$ such that for $\eps\in (0,\eps_0)$, $\alpha \in (0,1)$ and $\mu\in B(\mu_0)$
satisfying  \eqref{cond:ExpSmall} and that $\eps^{\eta-1}\frac{\sqrt{\eps}+\sqrt{\al}}{1-\al}$ is small enough, there exists a
real-analytic function $\CCC$ defined in $R_{\kk_2}\times\TT^2_\sigma$ such that
the change
\begin{equation}\label{def:CanviConjugador:QP}
 (u,\theta_1,\theta_2)=(v+\CCC(v,\theta_1,\theta_2),\theta_1,\theta_2)
\end{equation}
 conjugates the operators $\wt\LL_\eps$ and $\LL_\eps$ defined in
\eqref{def:OperadorAnulador:QP} and \eqref{def:Lde:QP}. Moreover, for $(v,\theta_1,\theta_2)\in
R_{\kk_2}\times\TT_\sigma$, $v+\CCC(v,\theta_1,\theta_2) \in R_{\kk_1}$ and there exists a
constant $b_6>0$ such that
\[
\begin{split}
 \left|\CCC(v,\theta_1,\theta_2)\right|&\leq b_6|\mu|\eps^{\eta-\frac{1}{2}}\frac{\sqrt{\al}+\sqrt{\eps}}{1-\al}|\ln\eps|\\
 \left|\pa_v\CCC(v,\theta_1,\theta_2)\right|&\leq b_6|\mu|\eps^{\eta-1}
\frac{\sqrt{\al}+\sqrt{\eps}}{1-\al}|\ln\eps|.
\end{split}
\]
Furthermore, $(u,\theta_1,\theta_2)=(v+\CCC(v,\theta_1,\theta_2),\theta_1,\theta_2)$ is invertible and its inverse is
of the form $ (v,\theta_1,\theta_2)=(u+\VV(u,\theta_1,\theta_2),\theta_1,\theta_2)$ where $\VV$ is a function defined
for $(u,\theta_1,\theta_2)\in R_{\kk_3}\times\TT^2_\sigma$ and satisfies 
\[
 \left|\VV(u,\theta_1,\theta_2)\right|\leq b_6|\mu|\eps^{\eta-\frac{1}{2}}\frac{\sqrt{\al}+\sqrt{\eps}}{1-\al}|\ln\eps|
\]
and that $u+\VV(u,\theta_1,\theta_2)\in R_{\kk_2}$ for $(u,\theta_1,\theta_2)\in
R_{\kk_3}\times\TT_\sigma^2$.
\end{theorem}
The proof of this theorem is deferred to Section \ref{sec:Canvi:QP}.

Next step is to prove the validity of the Melnikov function. To this end, we bound
$\pa_u\Delta(u,\theta_1,\theta_2)-\MM(u,\theta_1,\theta_2)$, where $\Delta$ and $\MM$ are the functions
defined in \eqref{def:DiffManifolds:QP} and \eqref{def:Melnikov:QP} respectively. As in the periodic case, as a first step,  we bound
 $\pa_v\left(\Delta(v+\CCC(v,\theta_1,\theta_2),\theta_1,\theta_2)\right)-\MM(v,\theta_1,\theta_2)$ where $\CCC$ is the function obtained in Theorem \ref{th:Canvi:QP}.
\begin{theorem}\label{th:CotaExpPetita:QP}
 There exists $\eps_0$ and $b_7>0$ such that for any $\eps\in (0,\eps_0)$, $\al\in (0,1)$ and $\mu\in B(\mu_0)$, satisfying  \eqref{cond:ExpSmall} and that $\eps^{\eta-1}\frac{\sqrt{\eps}+\sqrt{\al}}{1-\al}$ is small enough, the following bound is satisfied
\[
\left|\pa_v\left(\Delta (v+\CCC(v,\theta_1,\theta_2),\theta_1,\theta_2)\right)-\MM(v,\theta_1,\theta_2)\right|\leq
\frac{b_7|\mu|^2\eps^{2\eta-2}|\ln\eps|}{(1-\al)^2}e^{-c(\ln(\eps/\Im\rr_-)){\dps\sqrt{\tfrac{\Im\rr_-}{\eps}}}}
\]
for $v\in
R_{\kk_3}\cap\RR$ and $(\theta_1,\theta_2)\in\TT^2$,
\end{theorem}

\begin{proof}
First we define the Melnikov potential $L$, namely a function such that $\pa_uL=\MM$  (see \cite{DelshamsG00}).As $\Delta\in \mathrm{Ker}\wt\LL_\eps$, where $\wt\LL_\eps$ is the operator in \eqref{def:OperadorAnulador:QP}, by Theorem \ref{th:Canvi:QP},
the function $\Phi(v,\theta_1,\theta_2)=\Delta(v+\CCC(v,\theta_1,\theta_2),\theta_1,\theta_2)-L(v,\tau)\in\mathrm{Ker}\LL_\eps$, where
$\LL_\eps$ is the operator defined in \eqref{def:Lde:QP}. Therefore we can apply
 Lemma \ref{lemma:Lazutkin:QP}.

 To this end, we have to bound $\pa_v\Phi(v,\theta_1,\theta_2)$
in the domain $R_{\kk_2}\times\TT^2_\sigma$. We split $\Phi$ as
$\Phi(v,\theta_1,\theta_2)=\Phi_1(v,\theta_1,\theta_2)+\Phi_2(v,\theta_1,\theta_2)$ where
\[
\begin{split}
 \Phi_1(v,\theta_1,\theta_2)&=\Delta(v+\CCC(v,\theta_1,\theta_2),\theta_1,\theta_2)-\Delta (v,\theta_1,\theta_2)\\
\Phi_2(v,\theta_1,\theta_2)&= \Delta(v,\theta_1,\theta_2)-L (v,\theta_1,\theta_2),
\end{split}
\]
where $\Delta$ is the function defined in
\eqref{def:DiffManifolds:QP}
and $L$ is the Melnikov potential.

To bound $\pa_v\Phi_1$, one has to take into account that 
$\Delta=(T^s-T_0)-(T^u-T_0)$ and therefore, it is enough to consider the bounds
obtained in Theorems \ref{th:ExistenceManifolds:QP} and \ref{th:Canvi:QP} to obtain
\[
|\pa_v\Phi_1(v,\theta_1,\theta_2)|\leq \frac{K|\mu|^2\eps^{2\eta-2}}{(1-\al)^2}|\ln\eps|
\] for $(v,\theta_1,\theta_2)\in
R_{\kk_2}\times\TT^2_\sigma$. For the second term, it is enough to use bounds
\eqref{def:HalfMelnikov:st:QP:bound} and \eqref{def:HalfMelnikov:uns:QP:bound} to obtain
\[
|\pa_v\Phi_2(v,\theta_1,\theta_2)|\leq \frac{K|\mu|^2\eps^{2\eta-2}}{(1-\al)^2}|\ln\eps|\]
 for $(v,\theta_1,\theta_2)\in
R_{\kk_2}\times\TT^2_\sigma$.
 
Therefore, we have that 
\[
|\pa_v\Phi(v,\theta_1,\theta_2)|\leq\frac{K|\mu|^2\eps^{2\eta-2}}{(1-\al)^2}|\ln\eps|
\] for
$(v,\theta_1,\theta_2)\in R_{\kk_2}\times\TT^2_\sigma$ and then, it is enough to apply Lemma
\ref{lemma:Lazutkin:QP} to finish the proof of Theorem \ref{th:CotaExpPetita:QP}.
\end{proof}

From this result and considering the inverse change $(v,\theta_1,\theta_2)=(u+\VV(u,\theta_1,\theta_2),\theta_1,\theta_2)$ obtained in Theorem \ref{th:Canvi:QP}, it is
straightforward to obtain exponentially small bounds for 
$\pa_u\Delta(u,\theta_1,\theta_2)-\MM(u,\theta_1,\theta_2)$ and its derivative. They are stated in the next
corollary, whose proof is straightforward.
\begin{corollary}\label{coro:CotaExpPetita:QP}
 Let us consider any $\kk_4>\kk_3$. Then, there exists $\eps_0$ and $b_7>0$ such
that for any $\eps\in (0,\eps_0)$  and $\al\in (0,1)$, satisfying  \eqref{cond:ExpSmall} and that $\eps^{\eta-1}\frac{\sqrt{\eps}+\sqrt{\al}}{1-\al}$ is small enough,  the
following bound is satisfied
\[
\left|\pa_u\Delta (u,\theta_1,\theta_2)-\MM(u,\theta_1,\theta_2)\right|\leq
\frac{b_7|\mu|^2\eps^{2\eta-2}|\ln\eps|}{(1-\al)^2}e^{-c(\ln(\eps/\Im\rr_-)){\dps\sqrt{\tfrac{\Im\rr_-}{\eps}}}}
\]
for $u\in R_{\kk_4}\cap\RR$ and $(\theta_1,\theta_2)\in\TT^2$.
\end{corollary}
This corollary finishes the  proof of Theorem \ref{th:Main:QP}. Note that when $\al$ is bounded away from 1, we just need the simpler condition $\eps^{\eta-1}(\sqrt{\eps}+\sqrt{\al})$ small enough.

\subsubsection{The invariant manifolds: proof of Theorem
\ref{th:ExistenceManifolds:QP}}\label{sec:ExistenceManifolds:QP}
We look
for a solution of equation \eqref{def:HJ2:QP} satisfying the asymptotic condition
\eqref{eq:AsymptCondFuncioGeneradora:uns} as a perturbation
of  $T_0$ in \eqref{def:T0}. As in the periodic case, we define
\begin{equation}\label{def:T1:QP}
Q(u,\theta_1,\theta_2) = T (u,\theta_1,\theta_2) - T_0(u), 
\end{equation}
which is solution of 
\begin{equation}\label{eq:Existence:T1:QP}
 \LL_\eps Q=\FF(\pa_u Q,u,\theta_1,\theta_2),
\end{equation}
where $\LL_\eps$ is the operator defined in \eqref{def:Lde:QP} and
\begin{equation}\label{def:Existence:OperadorRHS:QP}
 \FF(h,u,\theta_1,\theta_2)=-\frac{\cosh^2u}{8}h^2-\mu\eps^\eta\Psi(u)F(\theta_1,\theta_2),
\end{equation}
where $\Psi(u)$ is the function defined in \eqref{def:Psi}.
 
We devote the rest of the section to obtain a solution of equation
\eqref{eq:Existence:T1:QP} which is defined  in  $D^u_{\kk}\times\TT^2_\sigma$ and
satisfies the asymptotic condition \eqref{eq:AsymptCondFuncioGeneradora:uns}. Recall that $D^u_\kk$ has been defined in \eqref{def:DominisOuter:QP} and $\sigma$ in \eqref{def:BandesAngulars}.

We use analogous norms as the ones in the periodic case. For analytic functions $h:D^u_\kk\times\TT^2_\sigma\rightarrow \CC$,
 we define the Fourier norm
\[
 \|h\|_\sigma=\sum_{k\in\ZZ^2}\left\|h^{[k]}\right\| e^{|k_1|\sigma_1+|k_2|\sigma_2},
\]
where $\|\cdot\|$ is the norm defined in \eqref{def:Norma:Existencia}. We consider the Banach  space
\begin{equation}\label{def:Existence:banach:QP}
 \EE_{\kk,\sigma}=\left\{h:D^u_\kk\times\TT^2_\sigma\rightarrow \CC;
\text{real-analytic}, \|h\|_\sigma<\infty\right\}.
\end{equation}

First we solve the
equation $\LL_\eps h=g$, where $\LL_\eps$ is the differential operator
defined in \eqref{def:Lde:QP}. This operator is  invertible in
$\EE_{\kk,\sigma}$ and its inverse can be defined as
\begin{equation}\label{def:operadorGInfty:QP}
\GG_\eps (h)(u,\theta)=\int_{-\infty}^0h(u+s,\theta_1+\eps\ii s,\theta_2+\ga\eps\ii s )\,dt.
\end{equation}

\begin{lemma}\label{lemma:Existence:PropietatsGInfty:QP}
The operator $\GG_\eps$  in \eqref{def:operadorGInfty:QP}   satisfies the
following properties.
\begin{enumerate}
\item $\GG_\eps$ is linear from $\EE_{\kk,\sigma}$ to
itself, commutes with $\pa_u$ and satisfies
$\LL_\eps\circ\GG_\eps=\mathrm{Id}$.
\item If $h\in\EE_{\kk,\sigma}$, then
\[
\left\|\GG_\eps(h)\right\|_{\sigma}\leq
K\|h\|_{\sigma}.
\]
Furthermore, one can bound each Fourier coefficient  $\GG^{[k]}_\eps(h)$ with $k\neq 0$ as
\[
 \left\| \GG^{[k]}(h)\right\|\leq \frac{K\eps}{|k\cdot \omega|}\left\|h^{[k]}\right\|.
\]
\item If $h\in\EE_{\kk,\sigma}$, then $\pa_u\GG_\eps(h)\in
\EE_{\kk,\sigma}$ and
\[
\left\|\pa_u\GG_\eps(h)\right\|_{\sigma}\leq
K\|h\|_{\sigma}.
\]
\end{enumerate}
\end{lemma}
\begin{proof}
The proof is analogous to the proof of Lemma 5.5 of \cite{GuardiaOS10}
\end{proof}

We can obtain solutions of equation \eqref{eq:Existence:T1:QP}
using a fixed point argument. Theorem \ref{th:ExistenceManifolds:QP} is a
straightforward consequence of the following proposition.

\begin{proposition}\label{prop:ExistenceT1:PtFix:QP}
Let us fix $\kk_1>0$. There exists $\eps_0>0$ such that for any $\eps\in
(0,\eps_0)$ and $\al\in (0,1)$ satisfying condition \eqref{cond:ExpSmall} and that $\eps^{\eta-1}\frac{\sqrt{\eps}+\sqrt{\al}}{1-\al}$
is small enough, there exists a function $Q$ defined in
$D^u_{\kk_1}\times\TT^2_\sigma$ such that $\pa_u Q\in \EE_{\kk_1,\sigma}$ is a
fixed point of the operator
\begin{equation}\label{def:ExistenceT1:Operator:QP}
 \ol \FF^u(h)=\pa_u\GG_\eps\FF(h),
\end{equation}
where $\GG_\eps$ and $\FF$ are the operators defined in
\eqref{def:operadorGInfty:QP} and \eqref{def:Existence:OperadorRHS:QP} respectively.
Furthermore, there exists a constant $b_5>0$ such that,
\[
\begin{split}
 \left\|\pa_uQ\right\|_\sigma&\leq b_5|\mu|\eps^{\eta},\\
 \left\|\pa^2_uQ\right\|_\sigma&\leq b_5|\mu|\eps^{\eta-\frac{1}{2}}.
\end{split}
\]
Moreover, if we consider the half Melnikov function defined in
\eqref{def:HalfMelnikov:uns:QP}, 
\begin{equation}\label{eq:CotaGeneradoraMenysMelnikov:QP}
 \left\| \pa_u Q-\mu\eps^\eta\MM^u\right\|_\sigma\leq b_5|\mu|^2\eps^{2\eta-1}
\frac{\sqrt{\eps}+\sqrt{\al}}{1-\al}.
\end{equation}
\end{proposition}
\begin{proof}
Let us consider $\kk_0<\kk_1$. It is straightforward to see that $\ol\FF^u$ is
well defined from $\EE_{\kk_0,\sigma}$ to itself. We are going to prove that
there exists a constant $b_5>0$ such that $\ol\FF^u$ sends $\ol
B(b_5|\mu|\eps^{\eta})\subset\EE_{\kk_0,\sigma}$ to itself and is contractive
there.

Let us first consider $\ol\FF^u(0)$. From the definition of $\ol\FF^u$ in
\eqref{def:ExistenceT1:Operator:QP}, the definition of $\FF$ in
\eqref{def:Existence:OperadorRHS:QP} and using Lemma
\ref{lemma:Existence:PropietatsGInfty:QP}, we have that
\[
 \ol\FF^u(0)(u,\theta_1,\theta_2)=\pa_u\GG_\eps\FF(0)(u,
\theta_1,\theta_2)=-\mu\eps^\eta\GG_\eps\left(\Psi'(u)F(\theta_1,\theta_2)\right).
\]
To bound it, we bound first each Fourier coefficient and we take advantage of the fact that $\langle F\rangle=0$. By Lemma \ref{lemma:Existence:PropietatsGInfty:QP}, formula \eqref{hyp:QP1} and recalling that $\Psi'(u)=\beta(u)$, where $\beta$ is
the function defined in \eqref{def:MelnikovIntegrand}, satisfies $\|\Psi'\|\leq K$, we have that
\[
\left\|  \ol\FF^u(0)^{[k]}\right\|\leq \frac{K|\mu|\eps^{\eta+1}}{|k\cdot\omega|}e^{-r_1|k_1|-r_2|k_2|}
\]
and therefore we have that 
\[
\left\|  \ol\FF^u(0)\right\|_\sigma\leq K|\mu|\eps^{\eta+1}\sum_{k\in\ZZ\setminus\{0\}}\frac{1}{|k\cdot\omega|}e^{-d_1|k_1|\sqrt{\eps}-d_2|k_2|\sqrt{\eps}}.
\]
To bound this sum, we split it in two depending $|k\cdot\omega|>1/2$ or  $|k\cdot\omega|<1/2$. For the first one,
\[
 \sum_{\substack{k\in\ZZ^2\setminus\{0\}\\|k\cdot\omega|>1/2}}\frac{1}{|k\cdot\omega|}e^{-d_1|k_1|\sqrt{\eps}-d_2|k_2|\sqrt{\eps}}\leq  2\sum_{k\in\ZZ^2\setminus\{0\}}e^{-d_1|k_1|\sqrt{\eps}-d_2|k_2|\sqrt{\eps}}\leq \frac{K}{\eps}.
\]
For the second one, we take into account that, for a fixed $k_2$, there exists only one $k_1$ satisfying $|k\cdot\omega|<1/2$. Moreover, it satisfies
\[
 -|k_1|\leq -\ga |k_2|+|k\cdot\omega|\leq -\ga |k_2|+\frac{1}{2}.
\]
Therefore, with this inequality and taking into account that for any $k\in\ZZ^2\setminus\{0\}$,
\[
|k\cdot\omega|>\frac{C}{|k|}  
 \]
for certain $C>0$, we can see that 
\[
 \sum_{\substack{k\in\ZZ^2\setminus\{0\}\\|k\cdot\omega|<1/2}}\frac{1}{|k\cdot\omega|}e^{-d_1|k_1|\sqrt{\eps}-d_2|k_2|\sqrt{\eps}}\leq  K\sum_{k_2\in\ZZ\setminus\{0\}}|k_2|e^{-(d_1\ga+d_2)|k_2|\sqrt{\eps}+\frac{d_1}{2}\sqrt{\eps}}\leq \frac{K}{\eps}.
\]
Then, one can easily see that there exists a
constant $b_5>0$ such that
\[
 \left\|\ol\FF^u(0)\right\|_\si\leq\frac{b_5}{2}|\mu|\eps^{\eta}.
\]
To bound the  Lipschitz constant, let us consider  $h_1,h_2\in \ol
B(b_5|\mu|\eps^{\eta})\in\EE_{\kk_0,\sigma}$. Then, one can proceed as in the periodic case recalling that now
\[
\left|\frac{u-i\pi/2}{u-\rr_-} \right|\leq
1+\left|\frac{i\pi/2-\rr_-}{u-\rr_-} \right|\leq 1+K\frac{\sqrt{\al}}{\sqrt{\eps}}
\]
and
\[
 \left|\frac{1}{u-\ol\rr_-} \right|\leq\frac{K}{\sqrt{1-\al}}.
\]
Therefore,
\[
\begin{split}
\left \|\ol \FF^u(h_2)-\ol \FF^u(h_1)\right\|_\si&\leq
\frac{K}{\sqrt{\eps}(1-\al)}\left(1+\frac{\sqrt{\al}}{\sqrt{\eps}}\right)\|h_2+h_1\|_\si\|h_2-h_1\|_\si\\
&\leq K|\mu|\eps^{\eta-1}\frac{\sqrt{\eps}+\sqrt{\alpha}}{1-\al}\|h_2-h_1\|_\si.
\end{split}
\]
Then, using that by hypothesis $\eps^{\eta-1}\frac{\sqrt{\eps}+\sqrt{\alpha}}{1-\al}$ is
small enough, 
\[
\mathrm{Lip}\ol\FF^u=K|\mu|\eps^{\eta-1}\frac{\sqrt{\eps}+\sqrt{\alpha}}{1-\al}<1/2
\] and therefore $\ol\FF^u$ is contractive
from the ball $\ol B(b_5|\mu|\eps^{\eta})\subset \EE_{\kk_0,\sigma}$ into
itself, and it has a unique fixed point $h^\ast$. Moreover, since it has
exponential decay as $\Re u\rightarrow -\infty$, we can take
\[
Q(u,\tau)=\int_{-\infty}^u h^\ast (v,\tau)dv.
\]
To obtain the bound for $\pa_u^2Q$, it is enough to apply Cauchy estimates to
the nested domains $D_{\kk_1}^u\subset D_{\kk_0}^u$ (see, for instance,
\cite{GuardiaOS10}), and rename $b_5$ if necessary.

Finally, to prove \eqref{eq:CotaGeneradoraMenysMelnikov}, it is enough to point
out that $\ol\FF^u(0)=\mu\eps^\eta\MM^u$ and consider the obtained bound for the
Lipschitz constant.
\end{proof}

\subsubsection{Straightening the operator $\wt\LL_\eps$: proof of Theorem
\ref{th:Canvi:QP}}\label{sec:Canvi:QP}
As we have explained in the periodic case, we just need to 
look for a function $\CCC$ solution of the equation
\[
 \LL_\eps\CCC(v,\theta_1,\theta_2)=\left.\frac{\cosh^2u}{8}\left(\pa_uT^s(u,\theta_1,\theta_2)+\pa_u
T^u(u,\theta_1,\theta_2)\right)\right|_{u=v+\CCC(v,\theta_1,\theta_2)}-1.
\]
Taking into account the definition of $T_0$ and $Q$ in \eqref{def:T0} and
\eqref{def:T1:QP}, this equation can be written as
\begin{equation}\label{eq:EDPCanvi:QP}
 \LL_\eps\CCC=\JJ(\CCC)
\end{equation}
where
\begin{equation}\label{def:EDPCanvi:RHS:QP}
\JJ(h)(v,\theta_1,\theta_2)= \left.\frac{\cosh^2u}{8}\left(\pa_uQ^s(u,\theta_1,\theta_2)+\pa_u
Q^u(u,\theta_1,\theta_2)\right)\right|_{u=v+h(v,\theta_1,\theta_2)}.
\end{equation}
To look for a solution of this equation, we start by defining some norms and
Banach spaces. Given $n\in \NN$ and an analytic function
$h:R_{\kk}\times\TT^2_\sigma\rightarrow \CC$, we define the  Fourier norm
\[
\|h\|_{n,\sigma}=\sum_{k\in\ZZ^2}\left\|h^{[k]}\right\|_{n}e^{|k_1|\si_1+|k_2|\si_2},
\]
where $\|\cdot\|_n$ is the norm defined in \eqref{def:Norma:Canvi}. Thus, we introduce the following Banach spaces
\[
 \XX_{n,\sigma}=\{
h:R_\kk\times\TT^2_\sigma\rightarrow\CC;\,\,\text{real-analytic},
\|h\|_{n,\sigma}<\infty\}.
\]

To find a right-inverse of the operator $\LL_\eps$ in \eqref{def:Lde:QP} in
$R_{\kk}\times\TT_\sigma^2$ we consider $u_1$ the upper vertex of $R_\kk$ and  $u_0$ the left
endpoint of $R_\kk$. Then, we define the operator
$\wt\GG_\eps$ as
\begin{equation}\label{def:operador:canvi:QP}
\wt\GG_\eps(h)(v, \theta)=\sum_{k\in\ZZ^2}\wt\GG_\eps(h)^{[k]}(v)e^{ik\cdot \theta},
\end{equation}
 where its Fourier coefficients are given by
\begin{align*}
\dps \wt\GG_\eps(h)^{[k]}(v)&=\int_{-v_1}^v
e^{i\frac{k\cdot\omega}{\eps}(s-v)}h^{\left[k\right]}(s)\,ds&\textrm{ if }k\cdot\omega<0\\
\dps \wt\GG_\eps(h)^{[0]}(v)&= \int_{v_0}^{v} h^{\left[0\right]}(s)\,ds&\\
\dps \wt\GG_\eps(h)^{[k]}(v)&= -\int^{v_1}_v
e^{i\frac{k\cdot\omega}{\eps}\omega(s-v)}h^{\left[k\right]}(s)\,ds&\textrm{ if }k\cdot\omega>0.
\end{align*}
The following lemma, which can be proved analogously to Lemma 8.3 of \cite{GuardiaOS10}, gives some properties of this operator.

\begin{lemma}\label{lemma:Canvi:Operador:QP}
The operator $\wt\GG_\eps$ in \eqref{def:operador:canvi:QP} satisfies if $h\in\XX_{1,\sigma}$, then $\wt\GG_\eps(h)\in
\XX_{0,\sigma}$ and
\[
\left\|\wt\GG_\eps(h)\right\|_{0,\sigma}\leq K\frac{|\ln\eps|}{\sqrt{1-\al}}\|h\|_{1,\sigma}.
\]
\end{lemma}

Next, Theorem \ref{th:Canvi:QP} is a straightforward consequence of the following
proposition.

\begin{proposition}\label{prop:Canvi:PtFix:QP}
Let us consider the constant $\kk_1>0$ defined in Theorem
\ref{th:ExistenceManifolds:QP} and let us consider any $\kk_2>\kk_1$. There exists
$\eps_0>0$ such that for any $\eps\in (0,\eps_0)$ and $\al\in (0,1)$ satisfying condition \eqref{cond:ExpSmall} and 
that $\eps^{\eta-1}\frac{\sqrt{\eps}+\sqrt{\al}}{1-\al}$ is small enough, there exists a function
$\CCC\in \XX_{1,\sigma}$ defined in $R_{\kk_2}\times\TT^2_\sigma$ such that  is a
fixed point of the operator
\begin{equation}\label{def:Canvi:Operator:QP}
 \ol \JJ(h)=\wt\GG_\eps\JJ(h),
\end{equation}
where $\wt\GG_\eps$ and $\JJ$ are the operators defined in
\eqref{def:operador:canvi:QP} and \eqref{def:EDPCanvi:RHS:QP} respectively.
Furthermore, $v+\CCC(v,\theta_1,\theta_2)\in R_{\kk_1}$ for $(v,\theta_1,\theta_2)\in
R_{\kk_2}\times\TT^2_\sigma$ and there exists a constant $b_6>0$ such that
\[
\begin{split}
 \left\|\CCC\right\|_{0,\sigma}\leq
b_6|\mu|\eps^{\eta-\frac{1}{2}}\frac{\sqrt{\eps}+\sqrt{\al}}{1-\al}|\ln\eps|\\
 \left\|\pa_v\CCC\right\|_{0,\sigma}\leq
b_6|\mu|\eps^{\eta-1}\frac{\sqrt{\eps}+\sqrt{\al}}{1-\al}|\ln\eps|.
\end{split}
\]
\end{proposition}
\begin{proof}
It is straightforward to see that $\ol\JJ$ is well defined from $\XX_{1,\sigma}$
to itself. We are going to prove that there exists a constant $b_6>0$ such that
$\ol\JJ$ sends $\ol
B(b_6|\mu|\eps^{\eta-\frac{1}{2}}\frac{\sqrt{\eps}+\sqrt{\al}}{1-\al}|\ln\eps|)\subset\XX_{1,\sigma}$ to
itself and is contractive there.

Let us first consider $\ol\JJ(0)$. From the definition of $\ol\JJ$ in
\eqref{def:Canvi:Operator:QP}, the definition of $\JJ$ in \eqref{def:EDPCanvi:RHS:QP}, we have that
\[
 \ol\JJ(0)(v,\theta_1,\theta_2)=\wt\GG_\eps\JJ(0)(v,\theta_1,\theta_2)=\wt\GG_\eps\left(\frac{\cosh^2v}{8}
\left(\pa_vQ^s(v,\theta_1,\theta_2)+\pa_v Q^u(v,\theta_1,\theta_2)\right)\right).
\]
Then, it is enough to apply Lemma \ref{lemma:Canvi:Operador:QP} and Proposition \ref{prop:ExistenceT1:PtFix:QP} to see that there exists a constant $b_6>0$ such that
\[
 \left\| \ol\JJ(0)\right\|_{0,\si}\leq
\frac{b_6}{2}|\mu|\eps^{\eta-\frac{1}{2}}\frac{\sqrt{\eps}+\sqrt{\al}}{1-\al}|\ln\eps|.
\]
To bound the  Lipschitz constant, it is enough to apply the mean value theorem,
use the bound of $\pa^2_u Q$ of Proposition \ref{prop:ExistenceT1:PtFix:QP} and
Lemma \ref{lemma:Canvi:Operador:QP} to see that 
\[
 \mathrm{Lip}\leq K|\mu|\eps^{\eta-1}\frac{\sqrt{\eps}+\sqrt{\al}}{1-\al}|\ln\eps|.
\]
Then, using that $\eps^{\eta-1}\frac{\sqrt{\eps}+\sqrt{\al}}{1-\al}|\ln\eps|\ll 1$, the operator
$\ol\JJ$ is contractive from  $\ol
B(b_6|\mu|\eps^{\eta-\frac{1}{2}}\frac{\sqrt{\eps}+\sqrt{\al}}{1-\al}|\ln\eps|)\subset\XX_{1,\sigma}$ to
itself and it has a unique fixed point $\CCC$. Finally, to obtain a bound for
$\pa_v\CCC$ it is enough to apply Cauchy estimates reducing slightly the domain
and renaming $b_6$ if necessary.
\end{proof}

\begin{proof}[Proof of Theorem \ref{th:Canvi:QP}]
 Once we have proved Proposition \ref{prop:Canvi:PtFix:QP},  it only remains to
obtain the inverse change given by the function $\VV$, which is straightforward
using a fixed point argument.
\end{proof}

\subsection{Proof of Theorem \ref{th:Main:NonExp:QP}}\label{sec:SketchProof:QP:NonExp}
The first statement of Theorem \ref{th:Main:NonExp:QP} is a direct consequence of Corollary \ref{coro:CotaExpPetita:QP} taking $\al=1-C\eps^r$ with $r\in (0,2)$. Note that the condition $\eps^{\eta-1}\frac{\sqrt{\eps}+\sqrt{\al}}{1-\al}$ becomes, as in the periodic cas, $\eta>r+1$. The proof of the second  and third statements, which correspond to $r\geq 2$, are considerably simpler, since we do not need to prove any exponential smallness.  It follows the same lines as the proof of Theorem \ref{th:Main:NonExp}.

\begin{theorem}\label{th:ExistenceManifolds:NonExp:QP}
Let us fix $\kk_1>0$. Then, there exists $\eps_0>0$ such
that for $\eps\in(0,\eps_0)$, $\alpha=1-C\eps^r$ with $C>0$ and $r\geq 2$, $\mu\in B(\mu_0)$, if
$\eta-3r/2>0$, the
Hamilton-Jacobi equation \eqref{def:HJ1:QP} has  a unique
(modulo an additive constant) real-analytic solution in
$\DD^u_{\kk_1}\times\TT^2_\sigma$ satisfying the asymptotic
condition \eqref{eq:AsymptCondFuncioGeneradora:uns}.

Moreover, there exists a real constant $b_8>0$ independent of $\eps$
and $\mu$, such that for $(u,\theta_1,\theta_2)\in
\DD^{u}_{\kk_1}\times\TT^2_\sigma$,
\[
\left|\pa_u T^u(u,\theta_1,\theta_2)-\pa_u T_0(u)\right|\leq
b_8|\mu|\eps^{\eta-3r/2}.
\]
Furthermore, for $(u,\theta_1,\theta_2)\in
\DD_{\kk_1}^{u}\times\TT^2_\sigma$,
the generating function $T^u$  satisfies that
\begin{equation}\label{def:HalfMelnikov:uns:cota:NonExp:QP}
\left|\pa_u T^u(u,\theta_1,\theta_2)-\pa_u T_0(u)-\mu\eps^\eta\MM^u(u,\theta_1,\theta_2)\right|\leq
b_8|\mu|^2\eps^{2\eta-3r},
\end{equation}
where $\MM^u$ is the function defined in \eqref{def:HalfMelnikov:uns:QP}.
\end{theorem}
\begin{proof}
It follows the same lines of the proof of Theorem \ref{th:ExistenceManifolds:NonExp}. We use the modified norm \eqref{def:Norma:NonExp} and we bound $\ol\FF(0)$ as
\[
\left\|\ol \FF(0)\right\|=K|\mu|\eps^\eta\left(K+\int_{-U}^u \frac{1}{|v-\rr_-|^2|v-\ol\rr_-|^2}\right) \leq \frac{b_4}{2}|\mu|\eps^{\eta-3r/2}.
\]
Finally, it is straightforward to see that $\ol\FF$ is contractive from the ball $\ol B(\mu\eps^{\eta-3r/2})$ to itself with Lipschitz constant equal to 
\[
 \mathrm{Lip}\leq |\mu|\eps^{\eta-3r/2},
\]
which gives the desired result.
\end{proof}

The function $T^s$ satisfies the same properties in the symmetric domain $\DD_{\kk_1}^s$.  From this theorem, the formula of the distance $\wt d(\theta_1,\theta_2)$ follows.

\section*{Acknowledgements}
The authors have been partially supported by the Spanish MCyT/FEDER grant
MTM2009-06973 and the Catalan SGR grant 2009SGR859. 
In addition, the research of M. G. has been supported by the
Spanish PhD grant FPU AP2005-1314. Part of this work was done while M. G. was doing stays in  the Pennsylvania State University and in the Fields Institute.  He wants to thank these institutions for their hospitality and support.

\appendix
\section{Some remarks on the singular periodic case}\label{sec:SingularCase}
We devote this section to give some remarks and conjectures about the singular
case, namely when the Melnikov function does not predict correctly the splitting of
separatrices. We restrict this discussion to the case $0<\al\leq \al_0<1$ for any fixed $\al_0$, where Theorem \ref{th:Main} holds.

The main point in the proof of the exponentially small splitting of separatrices is to give a good approximation of the invariant manifolds not only in the real line but in a complex strip wich reaches a neighborhood of order $\OO(\eps)$ of the singularities.
In this paper, this is the result given in Theorem \ref{th:ExistenceManifolds}.
To prove this theorem, one has to impose the condition
$\eps^{\eta-1}(\eps+\sqrt{\al})$ small enough. Looking at the first order of the perturbed invariant manifolds given by the half Melnikov functions \eqref{def:HalfMelnikov:uns} and \eqref{def:HalfMelnikov:st} one can see that this condition is necessary. In fact, if $\eps^{\eta-1}(\eps+\sqrt{\al})$ is of order 1, the half Melnikov functions have the same size as the separatrix when $u-\rr_-\sim\eps$. A more careful analysis shows that  $\pa_uT_0(u)$ and the remainder $\pa_uT(u,\tau)-\pa_uT_0(u)$ also become
of the same size when $u-\rr_-\sim\eps$. This implies that, when $\eps^{\eta-1}(\eps+\sqrt{\al})$ is
not small,  at a distance $\OO(\eps)$ of the singularities the unperturbed
separatrix is not a good approximation of the perturbed invariant manifolds and
then in these cases Melnikov fails to predict correctly the splitting of
separatrices. The correct approach when $\eps^{\eta-1}(\eps+\sqrt{\al})\sim 1$ is to look for
the first order of the invariant manifolds at a distance $\OO(\eps)$ of the
singularities in a different way. As was first pointed out in
\cite{Lazutkin84russian} in the study of the Standard Map, these new first orders are solutions of a different equation usually called \emph{inner equation}, which is independent of $\eps$ (see \cite{Gelfreich97, Gelfreich00,GelfreichS01,BaldomaS08, MartinSS10b, GaivaoG11}). For classical Hamiltonian Systems, the inner equation is a new Hamilton-Jacobi
equation, which has been studied in several models in \cite{OliveSS03, Baldoma06, GuardiaOS10,
BaldomaFGS11}. Nevertheless, in all these works the inner equation was
considered for points at a distance $\OO(\eps)$ of the singularity of the unperturbed
separatrix since it was this singularity which was giving the exponentially small
coefficient in the splitting. Nevertheless, for system
\eqref{def:sistema} one has to proceed more carefully since this
coefficient depends on $\al$  (see Proposition \ref{prop:Melnikov} and Theorem \ref{th:Main}). Therefore, we will obtain different inner equations
depending on the relation between $\al$ and $\eps$. Namely, in some cases ($0<\al\leq\eps^2$) we
will have to study the inner equation close to $u=i\pi/2$ and in others ($\eps^2\ll\al\leq \al_0<1$) close to
$u=\rr_-$.

We start with the case $0<\al\ll\eps^2$, which corresponds to a wide
analyticity strip. As we have explained Proposition
\ref{prop:Melnikov}, the exponential coefficient in the Melinkov function is
given by the imaginary part of the singularity of the unperturbed separatrix.
Namely, the analyticity strip is so wide that the Melnikov function behaves as in
the entire case $\al=0$. Then, it can be easily seen that
the singular change that one has to perform is given by $u=i\pi/2+\eps z$ and
$\varphi^{u,s} (z,\tau)=\eps T^{u,s}(i\pi/2+\eps z,\tau)$ (see \cite{OliveSS03}). Now the function
$\varphi$ is the solution of a new Hamilton-Jacobi equation. If we let
$\eps\rightarrow 0$ in this equation we obtain the inner equation
\begin{equation}\label{def:Inner:cas1}
 \pa_\tau\varphi_0-\frac{z^2}{8}\left(\pa_z\varphi_0\right)^2+\frac{2}{z^2}
(1-\mu\sin\tau)=0,
\end{equation}
which does not depend either on $\eps$ or $\al$.

Certain solutions $\varphi_0^{u,s}$ of this equation are the candidates to be the first order of the
functions $\varphi^{u,s}$, namely the first order of the parameterizations of
the invariant manifolds close to the singularity. Then, the study of their
difference would give the first order of the difference between the invariant
manifolds. This equation was already studied in \cite{OliveSS03} using
Resurgence Theory and in \cite{Baldoma06} using classical functional analysis
techniques.

The case $\al\sim\eps^2$ is the transition case. In Proposition
\ref{prop:Melnikov} we have seen that  the exponential coefficient
is given by $\pi/2$ but  the residuums of both $u=\rr_\pm$ make a contribution
to the Melnikov function. Let us assume that 
\[
 \al(\eps)=\al_*\eps^2+\OO\left(\eps^3\right).
\]
The natural change to inner variable is given by $u=i\pi/2+\eps z$
and the rescaling in the generating function by $\varphi^{u,s}(z,\tau)=\eps
T^{u,s}(i\pi/2+\eps z,\tau)$, obtaining , taking into account \eqref{eq:Expansio:Sing}, a new inner equation
\[
 \pa_\tau\varphi_0-\frac{z^2}{8}\left(\pa_z\varphi_0\right)^2+\frac{2}{z^2}
+\mu\frac{i}{2\left(z^2+2i\al_*\right)}\sin\tau
=0.
\]

Finally we deal with the case $\al\gg\eps^2$. We first consider the case $\al$ 
independent of $\eps$, and from it we will deduce the case $\al\sim\eps^\nu$ with $\nu\in (0,1)$. 
If we take a fixed $\al$, we have seen that one has to 
study the parameterizations of the invariant manifolds close to $u=\rr_-$ and
the singularity $u=\rr_+$ does not play any role in the size of the splitting.
Note that now the limiting case is $\eta=1$, that is, we deal with
equation
\[
 \eps\ii\pa_\tau T+\frac{\cosh^2
u}{8}\left(\pa_uT\right)^2-\frac{4}{\cosh^2u}+\mu\eps \Psi(u)\sin\tau=0.
\]
As a first step, we can expand the parameterization of the invariant manifold 
$T^{u,s}(u,\tau)$ as a power series of $\eps$. It can be easily seen that 
\[
 T^{u,s}(u,\tau)\sim\sum_{k\geq 0}\eps^kT_k(u,\tau)
\]
where $T_0$ corresponds to the separatrix \eqref{def:T0}. One can see that the terms of the series satisfy
that for $k\geq 0$,
\[
 \pa_u T_k(u,\tau)\sim \frac{1}{(u-\rr_-)^k}
\]
and therefore they all become of the same size at a distance $\eps$ of the
singularity. Nevertheless, in this case one has to be more careful, since if one
considers the asymptotic size of the power series terms of the generating
function $T$ instead of $\pa_u T$, we have that
\[
 T_0(u)\sim 1\,\,\,\text{ and }\,\,\, T_k(u,\tau)\sim
\frac{1}{(u-\rr_-)^{k-1}}\,\,\text{ for }k\geq 1.
\]
Namely, at a distance of order $\OO(\eps)$ of $u=\rr_-$ all the terms with
$k\geq 1$  become of the same order but $T_0$ is still bigger. Therefore, it is more convenient to deal with the function $Q=T-T_0$. Now, one can
consider the change to inner variable  $u=\rr_-+\eps z$ and the rescaling in the
generating function by $\phi(z,\tau)=\eps Q(\rr_-+\eps z,\tau)$,  obtaining the inner equation
\[
 \pa_\tau\phi+\pa_z\phi+\frac{\cosh^2\rr_-}{8}\left(\pa_z\phi\right)^2-\mu 
\frac{\de_2(\al)}{z}\sin\tau=0,
\]
where $\de_2(\al)$ is the funtion defined in \eqref{def:delta2}.

Proceeding analogously, one can deduce the inner equation for the case $\al=\al_*\eps^\nu$ with $\nu\in (0,2)$. Recall that for this range of $\al$ the limiting case was $\eta=1-\nu/2$ (see Remark \ref{remark:SingularCase}). Namely, we deal with the equation following equation in $Q(u,\tau)=T(u,\tau)-T_0(u)$,
\[
 \eps\ii\pa_\tau Q+\pa_uQ+\frac{\cosh^2
u}{8}\left(\pa_uT\right)^2+\mu\eps^{1-\frac{\nu}{2}} \Psi(u)\sin\tau=0.
\]
As in the previous case, we study the inner equation close to $u=\rr_-$ and therefore the change to inner variable is still 
 $u=\rr_-+\eps z$. Nevertheless, now the rescaling in the
generating function is given by $\phi(z,\tau)=\eps^{1-\nu/2} Q(\rr_-+\eps z,\tau)$. Then, proceeding as before and taking into account the definition of $\de_2(\al)$ in \eqref{def:delta2}, one can obtain the following inner equation
\[
 \pa_\tau\phi+\pa_z\phi-\frac{i\al_0}{4}\left(\pa_z\phi\right)^2+\frac{\mu}{(1+i)\sqrt{\al_0}} 
\frac{1}{z}\sin\tau=0.
\]

One could expect that studying all these inner equations, one could obtain the true first asymptotic order of the difference between the perturbed invariant manifolds.

\bibliography{references}

\def\cprime{$'$} \def\cprime{$'$}
\begin{thebibliography}{FGKR11}

\bibitem[Arn64]{Arnold64}
V.I. Arnold.
\newblock Instability of dynamical systems with several degrees of freedom.
\newblock {\em Sov. Math. Doklady}, 5:581--585, 1964.

\bibitem[Bal06]{Baldoma06}
I.~Baldom{\'a}.
\newblock The inner equation for one and a half degrees of freedom rapidly
  forced {H}amiltonian systems.
\newblock {\em Nonlinearity}, 19(6):1415--1445, 2006.

\bibitem[Ber08]{Bernard08}
Patrick Bernard.
\newblock The dynamics of pseudographs in convex {H}amiltonian systems.
\newblock {\em J. Amer. Math. Soc.}, 21(3):615--669, 2008.

\bibitem[BF04]{BaldomaF04}
I.~Baldom{\'a} and E.~Fontich.
\newblock Exponentially small splitting of invariant manifolds of parabolic
  points.
\newblock {\em Mem. Amer. Math. Soc.}, 167(792):x--83, 2004.

\bibitem[BF05]{BaldomaF05}
I.~Baldom{\'a} and E.~Fontich.
\newblock Exponentially small splitting of separatrices in a weakly hyperbolic
  case.
\newblock {\em J. Differential Equations}, 210(1):106--134, 2005.

\bibitem[BFGS11]{BaldomaFGS11}
I.~Baldom{\'a}, E.~Fontich, M.~Gu\`ardia, and T.~M. Seara.
\newblock Exponentially small splitting of separatrices for hamiltonian systems
  of one and a half degrees of freedom.
\newblock {\em To be submitted soon}, 2011.

\bibitem[Bou10]{Bounemoura10}
A.~Bounemoura.
\newblock Nekhoroshev estimates for finitely differentiable quasi-convex
  hamiltonians.
\newblock {\em Journal of Differential Equations}, (11):2905--2920, 2010.

\bibitem[BS08]{BaldomaS08}
I.~Baldom{\'a} and T.~M. Seara.
\newblock The inner equation for generic analytic unfoldings of the {H}opf-zero
  singularity.
\newblock {\em Discrete Contin. Dyn. Syst. Ser. B}, 10(2-3):323--347, 2008.

\bibitem[CG94]{ChierchiaG94}
L.~Chierchia and G.~Gallavotti.
\newblock Drift and diffusion in phase space.
\newblock {\em Ann. Inst. H. Poincar\'e Phys. Th\'eor.}, 60(1):144, 1994.

\bibitem[CY04]{ChengY04}
C.Q. Cheng and J.~Yan.
\newblock Existence of diffusion orbits in a priori unstable {H}amiltonian
  systems.
\newblock {\em J. Differential Geom.}, 67(3):457--517, 2004.

\bibitem[DdlLS06]{DelshamsLS06a}
A.~Delshams, R.~de~la Llave, and T.M. Seara.
\newblock A geometric mechanism for diffusion in hamiltonian systems overcoming
  the large gap problem: heuristics and rigorous verification on a model.
\newblock {\em Mem. Amer. Math. Soc.}, 2006.

\bibitem[DG96]{DelshamsG96a}
A.~Delshams and P.~Guti{\'e}rrez.
\newblock Effective stability and {KAM} theory.
\newblock {\em J. Differential Equations}, 128(2):415--490, 1996.

\bibitem[DG00]{DelshamsG00}
A.~Delshams and P.~Guti{\'e}rrez.
\newblock Splitting potential and the {P}oincar\'e-{M}elnikov method for
  whiskered tori in {H}amiltonian systems.
\newblock {\em J. Nonlinear Sci.}, 10(4):433--476, 2000.

\bibitem[DG03]{DelshamsG03}
A.~Delshams and P.~Guti{\'e}rrez.
\newblock Exponentially small splitting of separatrices for whiskered tori in
  {H}amiltonian systems.
\newblock {\em Zap. Nauchn. Sem. S.-Peterburg. Otdel. Mat. Inst. Steklov.
  (POMI)}, 300(Teor. Predst. Din. Sist. Spets. Vyp. 8):87--121, 287, 2003.

\bibitem[DGJS97]{DelshamsGJS97}
A.~Delshams, V.~Gelfreich, {\`A}.~Jorba, and T.M. Seara.
\newblock Exponentially small splitting of separatrices under fast
  quasiperiodic forcing.
\newblock {\em Comm. Math. Phys.}, 189(1):35--71, 1997.

\bibitem[DGS04]{DelshamsGS04}
A.~Delshams, P.~Guti{\'e}rrez, and T.M. Seara.
\newblock Exponentially small splitting for whiskered tori in {H}amiltonian
  sysems: flow-box coordinates and upper bounds.
\newblock {\em Discrete Contin. Dyn. Syst.}, 11(4):785--826, 2004.

\bibitem[DH09]{DelshamsH09}
A.~Delshams and G.~Huguet.
\newblock Geography of resonances and {A}rnold diffusion in a priori unstable
  {H}amiltonian systems.
\newblock {\em Nonlinearity}, 22(8):1997--2077, 2009.

\bibitem[DJSG99]{DelshamsGJS99}
A.~Delshams, {\`A}.~Jorba, T.~M. Seara, and V.~Gelfreich.
\newblock Splitting of separatrices for (fast) quasiperiodic forcing.
\newblock In {\em Hamiltonian systems with three or more degrees of freedom
  ({S}'{A}gar\'o, 1995)}, volume 533 of {\em NATO Adv. Sci. Inst. Ser. C Math.
  Phys. Sci.}, pages 367--371. Kluwer Acad. Publ., Dordrecht, 1999.

\bibitem[DS92]{DelshamsS92}
A.~Delshams and T.~M. Seara.
\newblock An asymptotic expression for the splitting of separatrices of the
  rapidly forced pendulum.
\newblock {\em Comm. Math. Phys.}, 150(3):433--463, 1992.

\bibitem[DS97]{DelshamsS97}
A.~Delshams and T.M. Seara.
\newblock Splitting of separatrices in {H}amiltonian systems with one and a
  half degrees of freedom.
\newblock {\em Math. Phys. Electron. J.}, 3:Paper 4, 40 pp. (electronic), 1997.

\bibitem[FGKR11]{FejozGKR11}
J.~F{\'e}joz, M.~Guardia, V.~Kaloshin, and P~Rold\'an.
\newblock Diffusion along mean motion resonances for the restricted planar
  three body probem.
\newblock 2011.
\newblock In preparation.

\bibitem[Fon93]{Fontich93}
E.~Fontich.
\newblock Exponentially small upper bounds for the splitting of separatrices
  for high frequency periodic perturbations.
\newblock {\em Nonlinear Anal.}, 20(6):733--744, 1993.

\bibitem[Fon95]{Fontich95}
E.~Fontich.
\newblock Rapidly forced planar vector fields and splitting of separatrices.
\newblock {\em J. Differential Equations}, 119(2):310--335, 1995.

\bibitem[Gel94]{Gelfreich94}
V.~G. Gelfreich.
\newblock Separatrices splitting for the rapidly forced pendulum.
\newblock In {\em Seminar on Dynamical Systems (St.\ Petersburg, 1991)},
  volume~12 of {\em Progr. Nonlinear Differential Equations Appl.}, pages
  47--67. Birkh\"auser, Basel, 1994.

\bibitem[Gel97a]{Gelfreich97a}
V.~G. Gelfreich.
\newblock Melnikov method and exponentially small splitting of separatrices.
\newblock {\em Phys. D}, 101(3-4):227--248, 1997.

\bibitem[Gel97b]{Gelfreich97}
V.~G. Gelfreich.
\newblock Reference systems for splittings of separatrices.
\newblock {\em Nonlinearity}, 10(1):175--193, 1997.

\bibitem[Gel00]{Gelfreich00}
V.~G. Gelfreich.
\newblock Separatrix splitting for a high-frequency perturbation of the
  pendulum.
\newblock {\em Russ. J. Math. Phys.}, 7(1):48--71, 2000.

\bibitem[GG11]{GaivaoG11}
J.~P. Gaiv{\~a}o and V.~Gelfreich.
\newblock Splitting of separatrices for the {H}amiltonian-{H}opf bifurcation
  with the {S}wift-{H}ohenberg equation as an example.
\newblock {\em Nonlinearity}, 24(3):677--698, 2011.

\bibitem[GGM99]{GalGM99}
G.~Gallavotti, G.~Gentile, and V.~Mastropietro.
\newblock Separatrix splitting for systems with three time scales.
\newblock {\em Comm. Math. Phys.}, 202(1):197--236, 1999.

\bibitem[GH83]{GuckenheimerH83}
J.~Guckenheimer and P.~Holmes.
\newblock {\em Nonlinear Oscillations, Dynamical Systems, and Bifurcations of
  Vector Fields}.
\newblock Springer-Verlag, 1983.

\bibitem[GOS10]{GuardiaOS10}
M.~Guardia, C.~Oliv\'e, and T.~Seara.
\newblock Exponentially small splitting for the pendulum: a classical problem
  revisited.
\newblock {\em Journal of Nonlinear Science}, 20(5):595--685, 2010.

\bibitem[GS01]{GelfreichS01}
V.~Gelfreich and D.~Sauzin.
\newblock Borel summation and splitting of separatrices for the {H}\'enon map.
\newblock {\em Ann. Inst. Fourier (Grenoble)}, 51(2):513--567, 2001.

\bibitem[HMS88]{HolmesMS88}
P.~Holmes, J.~Marsden, and J.~Scheurle.
\newblock Exponentially small splittings of separatrices with applications to
  {KAM} theory and degenerate bifurcations.
\newblock In {\em Hamiltonian dynamical systems}, volume~81 of {\em Contemp.
  Math.} 1988.

\bibitem[Laz84]{Lazutkin84russian}
V.~F. Lazutkin.
\newblock Splitting of separatrices for the {C}hirikov standard map.
\newblock VINITI 6372/82, 1984.
\newblock Preprint (Russian).

\bibitem[LMS03]{LochakMS03}
P.~Lochak, J.-P. Marco, and D.~Sauzin.
\newblock On the splitting of invariant manifolds in multidimensional
  near-integrable {H}amiltonian systems.
\newblock {\em Mem. Amer. Math. Soc.}, 163(775):viii+145, 2003.

\bibitem[LS80]{SimoL80}
J.~Llibre and C.~Sim{\'o}.
\newblock Oscillatory solutions in the planar restricted three-body problem.
\newblock {\em Math. Ann.}, 248(2):153--184, 1980.

\bibitem[Mel63]{Melnikov63}
V.~K. Melnikov.
\newblock On the stability of the center for time periodic perturbations.
\newblock {\em Trans. Moscow Math. Soc.}, 12:1--57, 1963.

\bibitem[MP94]{MartinezP94}
R.~Mart{\'{\i}}nez and C.~Pinyol.
\newblock Parabolic orbits in the elliptic restricted three body problem.
\newblock {\em J. Differential Equations}, 111(2):299--339, 1994.

\bibitem[MS02]{MarcoS02}
J.P. Marco and D.~Sauzin.
\newblock Stability and instability for {G}evrey quasi-convex near-integrable
  {H}amiltonian systems.
\newblock {\em Publ. Math. Inst. Hautes \'Etudes Sci.}, (96):199--275, 2002.

\bibitem[MSS10]{MartinSS10b}
P~Mart\'in, D.~Sauzin, and T.~M. Seara.
\newblock Resurgence of inner solutions for perturbations of the mcmillan map.
\newblock {\em \emph{submitted}}, 2010.

\bibitem[Nek77]{Nekhoroshev77}
N.~N. Nekhoro{\v{s}}ev.
\newblock An exponential estimate of the time of stability of nearly integrable
  {H}amiltonian systems.
\newblock {\em Uspehi Mat. Nauk}, 32(6(198)):5--66, 287, 1977.

\bibitem[Oli06]{Olive06}
C.~Oliv\'e.
\newblock {\em C\`alcul de l'escissi\'o de separatrius usant t\`ecniques de
  matching complex i ressurg\`encia aplicades a l'equaci\'o de
  Hamilton-Jacobi}.
\newblock {\tt http://www.tdx.cat/TDX-0917107-125950}, 2006.

\bibitem[OSS03]{OliveSS03}
C.~Oliv{\'e}, D.~Sauzin, and T.~M. Seara.
\newblock Resurgence in a {H}amilton-{J}acobi equation.
\newblock In {\em Proceedings of the International Conference in Honor of
  Fr\'ed\'eric Pham (Nice, 2002)}, volume 53(4), pages 1185--1235, 2003.

\bibitem[Poi99]{Poincare99}
H.~Poincar{\'e}.
\newblock {\em Les m\'ethodes nouvelles de la m\'ecanique c\'eleste}, volume 1,
  2, 3.
\newblock Gauthier-Villars, Paris, 1892--1899.

\bibitem[P{\"o}s93]{Poschel93}
J.~P{\"o}schel.
\newblock Nekhoroshev estimates for quasi-convex {H}amiltonian systems.
\newblock {\em Math. Z.}, 213(2):187--216, 1993.

\bibitem[Sau95]{Sauzin95}
D.~Sauzin.
\newblock R\'esurgence param\'etrique et exponentielle petitesse de l'\'ecart
  des s\'eparatrices du pendule rapidement forc\'e.
\newblock {\em Ann.Ins.Fourier}, 45(2):453--511, 1995.

\bibitem[Sau01]{Sauzin01}
D.~Sauzin.
\newblock A new method for measuring the splitting of invariant manifolds.
\newblock {\em Ann. Sci. \'Ecole Norm. Sup. (4)}, 34, 2001.

\bibitem[Sim94]{Simo94}
C.~Sim{\'o}.
\newblock Averaging under fast quasiperiodic forcing.
\newblock In {\em Hamiltonian mechanics (Toru\'n, 1993)}, volume 331 of {\em
  NATO Adv. Sci. Inst. Ser. B Phys.}, pages 13--34. Plenum, New York, 1994.

\bibitem[SMH91]{HolmesMS91}
J.~Scheurle, J.~E. Marsden, and P.~Holmes.
\newblock Exponentially small estimates for separatrix splittings.
\newblock In {\em Asymptotics beyond all orders (La Jolla, CA, 1991)}, volume
  284 of {\em NATO Adv. Sci. Inst. Ser. B Phys.}, pages 187--195. Plenum, New
  York, 1991.

\bibitem[SV01]{SimoV01}
C.~Sim{\'o} and C.~Valls.
\newblock A formal approximation of the splitting of separatrices in the
  classical {A}rnold's example of diffusion with two equal parameters.
\newblock {\em Nonlinearity}, 14(6):1707--1760, 2001.

\bibitem[Tre97]{Treshev97}
D.~Treschev.
\newblock Separatrix splitting for a pendulum with rapidly oscillating
  suspension point.
\newblock {\em Russ. J. Math. Phys.}, 5(1):63--98, 1997.

\bibitem[Tre04]{Treschev04}
D.~Treschev.
\newblock Evolution of slow variables in a priori unstable hamiltonian systems.
\newblock {\em Nonlinearity}, 17(5):1803--1841, 2004.

\bibitem[Xia92]{Xia92}
Z.~Xia.
\newblock Mel\cprime nikov method and transversal homoclinic points in the
  restricted three-body problem.
\newblock {\em J. Differential Equations}, 96(1):170--184, 1992.

\end{thebibliography}
\bibliographystyle{alpha}
\end{document}